\documentclass[12pt,a4paper,reqno]{amsart}
\usepackage[utf8]{inputenc}
\usepackage[T1]{fontenc}
\usepackage[english]{babel}
\usepackage{amsfonts,amsmath, amssymb,latexsym,times,mathrsfs}

\usepackage{graphics}

\usepackage{graphicx}

\usepackage{ae,aecompl}
\usepackage{pstricks}
\usepackage{pstricks-add}
\usepackage{float}

\usepackage[colorlinks=true]{hyperref}
\hypersetup{ pdfstartview=FitH }

\newcommand{\cev}[1]{\reflectbox{\ensuremath{\vec{\reflectbox{\ensuremath{#1}}}}}}

\numberwithin{equation}{section}
\def\opl{\Xi}
\def\II{{\mathrm{\mathbf{I\!I}}}}
\def\I{{\mathbf{1}}}

\def\NN{{\mathbb{N}}}
  
\def\ZZ{{\mathbb{Z}}}  \def\SS{{\mathbb{S}}}   \def\CC{{\mathbb{C}}}
\def\RR{{\mathbb{R}}}   \def\EE{{\mathbb{E}}} 
\def\cO{{\mathcal{O}}}
\def\cG{{\mathcal{G}}}

\def\cJ{{\mathcal{J}}}

\def\scrK{{\mathscr{K}}}
\def\scrH{{\mathscr{H}}}

\def\fg{{\mathfrak{g}}}

\def\fJ{{\mathfrak{J}}}
\def\cL{{\mathcal{L}}}

\def\cM{{\mathscr{M}}}
\def\scrM{{\mathscr{M}}}
\def\cC{{\mathcal{C}}}

\def\scrU{{\mathscr{U}}}
\def\scrV{{\mathscr{V}}}
\def\scrL{{\mathscr{L}}}

\def\cE{{\mathcal{E}}}

\def\cK{{\mathcal{K}}}

\def\cD{{\mathcal{D}}}

\def\cP{{\mathcal{P}}}
\def\cE{{\mathcal{E}}}
\def\cW{{\mathcal{W}}}

\def\Ham{\mathrm{Ham}}

\def\vert{\mathbf{v}}
\def\zert{\mathbf{z}}
\def\face{\mathbf{f}}
\def\edge{\mathbf{e}}

\def\grad{\mathrm{grad}}

\def\fg{{\mathfrak{g}}}
\def\fC{{\mathfrak{C}}}
\def\Green{{{G}}}

\renewcommand{\epsilon}{\varepsilon}
\newcommand{\ip}[1]{\langle #1 \rangle}
\newcommand{\ipp}[1]{\langle \!  \langle #1 \rangle\! \rangle}

\newcommand{\Lie}{\mathrm{Lie}}
\newcommand{\vol}{\mathrm{vol}}
\newcommand{\area}{\mathrm{Area}}

\newcommand{\id}{\mathrm{id}}

\newcommand{\cT}{\mathcal{T}}
\newcommand{\cQ}{\mathcal{Q}}

\newcommand{\del}{\partial}

\renewcommand{\Im}{\mathrm{Im}}

\newcommand{\geuc}{g_{\mathrm{euc}}}

\newcommand{\hkto}{\hookrightarrow}

\newtheorem{theointro}{Theorem}

\newtheorem{corintro}[theointro]{Corollary}

\newtheorem{lemma}[subsubsection]{Lemma}
\newtheorem{prop}[subsubsection]{Proposition}
\newtheorem{cor}[subsubsection]{Corollary}
\newtheorem{theo}[subsubsection]{Theorem}
\newtheorem{notation}[subsubsection]{Notation}
\newtheorem{convention}[subsubsection]{Convention}
\newtheorem{dfn}[subsubsection]{Definition}

\theoremstyle{definition}

\theoremstyle{remark}
\newtheorem{rmk}[subsubsection]{Remark}
\newtheorem{example}[subsubsection]{Example}

\newtheorem*{rmk*}{Remark}
\newtheorem*{rmks*}{Remarks}
\newtheorem{rmks}[subsubsection]{Remarks}


\title{Discrete geometry and isotropic surfaces}

\author{François Jauberteau}
\address{François Jauberteau, Laboratoire Jean Leray, Universit\'e de Nantes}
\email{francois.jauberteau@univ-nantes.fr}

\author{Yann Rollin}
\address{Yann Rollin, Laboratoire Jean Leray, Universit\'e de Nantes}
\email{yann.rollin@univ-nantes.fr}

\author{Samuel Tapie}
\address{Samuel Tapie, Laboratoire Jean Leray, Universit\'e de Nantes}
\email{samuel.tapie@univ-nantes.fr}

\begin{document}
\begin{abstract}
  We consider  smooth isotropic immersions from the $2$-dim\-en\-sion\-al torus
  into $\RR^{2n}$, for $n\geq 2$. When $n=2$ the image of such map is an immersed Lagrangian torus of $\RR^4$.
  We prove that such isotropic immersions  can be approximated by
  arbitrarily $\cC^0$-close piecewise linear isotropic maps. If $n\geq
  3$ the piecewise linear isotropic maps can be chosen so that they are piecewise linear isotropic immersions as well.

  The proofs are obtained using analogies with an infinite
  dimensional moment map geometry due to Donaldson. As a byproduct of
  these considerations, we introduce a numerical flow in finite
  dimension,  whose limits  provide, from an experimental perspective, many
  examples of piecewise linear Lagrangian tori in $\RR^4$.
  \href{http://www.math.sciences.univ-nantes.fr/~rollin/index.php?page=flow}{The
    DMMF program,} 
  which is freely available, is based on the Euler
  method and  shows the evolution equation of discrete surfaces in real time, as
  a movie.
\end{abstract}

{\Huge \sc \bf\maketitle}

\section{Introduction}
\subsection{Original motivations and background}
\emph{Lagrangian submanifolds} are natural objects, arising in the context of \emph{Hamiltonian mecanics} and \emph{dynamical systems}. Their prominent role in symplectic topology and geometry should not come as a surprise. In spite of tremendous efforts, the classification of Lagrangian submanifolds, up to \emph{Hamiltonian isotopy}, is generally an open
problem: for instance, Lagrangian tori of the Euclidean symplectic space $\RR^4$ are not classified up to Hamiltionian isotopy.
Lagrangian submanifolds are also key objects of various gauge theories. For example, 
the \emph{Lagrangian Floer theory}  is defined by counting pseudoholomorphic discs with boundary contained in some prescribed Lagrangian submanifolds.
 Many examples of \emph{smooth} Lagrangian submanifolds are known. They are easy
 to construct and to deform. In a nutshell, Lagrangian submanifolds  are typical, rather \emph{flexible} objects, from symplectic
 topology.

 An elementary  construction of Lagrangian submanifold is provided by
 considering the  $0$-section of the cotangent bundle of a smooth
 manifold $T^*L$, 
 endowed with its natural symplectic structure $\omega=d\lambda$,
 where $\lambda$ is the \emph{Liouville form}. More generally, it is
 well known that any section of $T^*L$ given by a closed $1$-form
 is a Lagrangian submanifold. Furthermore, such Lagrangian submanifolds are  Hamiltonian isotopic to the $0$-section if, and only if, the corresponding $1$-form is exact.
 These examples  provide a large class of Lagrangian submanifolds which
 admit 
 as many Hamiltonian deformations as smooth
 function on $L$ modulo  constants.

 By the \emph{Lagrangian
 neighborhood theorem}, every Lagrangian submanifold $L$ of a symplectic manifold admits a neighborhood symplectomorphic to a
 neighborhood of the $0$-section of $T^*L$. It follows that  the local
 Hamiltonian deformations discussed above (in the case of
 $T^*L$) also provide deformations for  Lagrangian
 submanifolds of \emph{any} symplectic manifold.

 The geometric notion of \emph{stationary Lagrangians submanifolds} was introduced
 by Oh~\cite{O90, O93} 
 in order to seek canonical representatives,  in a
 given isotopy class of Lagrangian submanifolds. Stationary Lagrangian
 submanifolds can be though of as analogues of \emph{minimal
 submanifolds} in the framework of symplectic geometry.
 As in the case of minimal surfaces, one can define various
 modified versions of the \emph{mean curvature flow},
 which are expected to converge toward stationary
 Lagrangians submanifolds in a given isotopy class.
 In an attempt to implement numerical versions of these flows~\cite{JR},
 we ended up facing theoretical problems of
 a \emph{discrete geometric nature}.
 Indeed, from a numerical point of view, surfaces are
 usually understood as some type of~\emph{mesh} and their mathematical
 counterpart is \emph{discrete geometry} and sometimes \emph{piecewise
   linear} geometry. 
Two obstacles arose in order to provide a sound numerical
simulation of geometric flows for Lagrangian submanifolds, namely:
\begin{enumerate}
\item To the best of our knowledge, discrete Lagrangian surfaces of $\RR^4$ and more generally discrete isotropic surfaces of
  $\RR^{2n}$
  are poorly understood, in fact hardly studied. We had no available
  examples of discrete Lagrangian tori in $\RR^4$ in our toolbox, save
  some discrete analogues of product or Chekanov
  tori~(cf Section \ref{sec:examples}). Furthermore, we had no deformation theory
 that we could rely upon contrarily to the smooth case. 
 Implementing a geometric evolution equation for discrete
 Lagrangian surfaces,
 with  so few examples to start the flow was not an enticing project.
 \item As far as a program is based on a numerical implementation,
   using floating point numbers, it is not natural
   to check if a symplectic form
   vanishes \emph{on the nose} along a plane. It only makes sense to
   test if the symplectic density is rater small, which means that we
   have an approximate solution of our problem.
    From an experimental point of view, we dread our numerical flow 
    would exhibit  some spurious drift of the symplectic density.
    We feared such instabilities may jeopardize our numerical
    simulations for flowing Lagrangian submanifolds. 
\end{enumerate}
   These issues led us to consider an auxiliary flow. Ideally, the auxiliary flow should
   attract \emph{any} discrete surface toward Lagrangian discrete surfaces. The
   utility of the auxiliary flow would be $2$-fold: its limits would
   provide examples of Lagrangian discrete surfaces for our experiments. It may also be
   used to    prevent instabilities of evolution equation along
   the moduli space of discrete Lagrangian surfaces.

These questions are part of a larger ongoing project. They have not been fully
investigated yet but  stirred many questions
   of a discrete differential geometric nature, in the context of
   symplectic geometry.
   This paper delivers a few answers to some of the simplest questions
   arising,  as a spin-off to our initial motivations.

\subsection{Statement of results}
We consider smooth maps $\ell:\Sigma\to\RR^{2n}$,
where $\Sigma$ is a
surface and $n\geq 2$.  The Euclidean space
$\RR^{2n}$ is endowed with its standard symplectic form $\omega$.
A map $\ell$ is said to be \emph{isotropic} if $\ell^*\omega=0$.
Lagrangian tori of $\RR^4$ are the submanifolds obtained as the image
of  $\Sigma$ by $\ell$, in the particular case  where $2n=4$, $\Sigma$
is diffeomorphic to a torus and $\ell$ is an isotropic
embedding.

In this paper, we construct approximations of smooth isotropic 
immersions of
the torus in $\RR^{2n}$  by \emph{piecewise linear isotropic maps}. The idea is
to consider a discretization of the torus by a square grid and
approximate the smooth map by a quadrangular mesh. This mesh is almost
isotropic, in a suitable sense. A perturbative argument shows that there exists a nearby
isotropic quadrangular mesh, which is used to build a piecewise linear
map. We provide a more precise statement of the above claims in the rest of the introduction.
\subsubsection{Piecewise linear isotropic maps}
We recall some usual definitions before stating one of our main results.
A \emph{triangulation} of $\RR^2$ 
is a locally finite \emph{simplicial complex} that covers $\RR^2$ entirely.
In this paper, points, line segments, triangles of triangulations are understood as geometrical Euclidean
objects  of the plane. Similarly, we shall consider triangulations of
quotients of $\RR^2$ by a lattice $\Gamma$, obtained by quotient of
$\Gamma$-invariant triangulations of $\RR^2$.

A \emph{piecewise linear map} $f:\RR^2\to\RR^m$ is a
continuous map such that, for some  triangulation of
$\RR^2$, the restriction of $f$ to any triangle is an affine map to
$\RR^m$.

We consider  smooth
isotropic  immersions $\ell:\Sigma\to \RR^{2n}$, where $\Sigma$ is
diffeomorphic to a $2$-dimensional torus and $n\geq 2$. The Euclidean metric $g$ of
$\RR^{2n}$ induces a conformal structure on $\Sigma$. The uniformization
theorem implies that the conformal structure of $\Sigma$ actually
comes from a quotient of $\RR^2$, with its canonical conformal
structure, by a lattice. Thus, we have a conformal covering map
$$
p:\RR^2\to\Sigma,
$$
with group of deck transformations $\Gamma$, a lattice of $\RR^2$.

A triangulation (resp. quadrangulation) of $\Sigma$ is called an
\emph{Euclidean triangulation } (resp. \emph{quadrangulation}) of $\Sigma$ if the boundary
of every face lifts to an Euclidean triangle (resp.quadrilateral) of
$\RR^2$ via $p$.

Similarly, a function $f:\Sigma\to\RR^m$ is a \emph{piecewise linear
  map} if it lifts to a piecewise
   linear map $\RR^2\to\RR^m$ via $p$. Given a  piecewise linear map
 $\hat \ell:\Sigma\to \RR^{2n}$, the pull-back of the symplectic form
 $\omega$ of $\RR^{2n}$ makes sense on each triangle of the triangulation
 subordinate to $\hat \ell$. We say that $\hat\ell$ is a \emph{isotropic piecewise
 linear map} if the pull back of $\omega$ vanishes along each face of the triangulation.
A piecewise linear map which is locally injective is called a \emph{piecewise linear immersion.}

The main result of this paper can be stated as follows:
\begin{theointro}
  \label{theo:maindiscr}
  Let $\ell:\Sigma\to\RR^{2n}$ be a smooth isotropic immersion, where
  $\Sigma$ is a 
  surface  diffeomorphic to a compact torus and $n\geq 2$.    
  Then, for every $\epsilon >0$,
   there exists a piecewise linear isotropic   map
  $\hat\ell:\Sigma\to\RR^{2n}$ such that  for every $x\in \Sigma $, we have
   $$\|\ell(x)-\hat\ell(x)\|\leq \epsilon.$$
   Furthermore, if $n\geq 3$, we may assume that $\hat \ell$ is an
   immersion.
   If $n=2$, we may assume that $\hat\ell$ is an immersion away from a
   finite union of embedded circles in $\Sigma$. 
\end{theointro}
Loosely stated,  Theorem~\ref{theo:maindiscr} says that every
isotropic immersion $\ell$ of a torus into $\RR^{2n}$ can be approximated by a
piece linear map arbitrarily $\cC^0$-close to~$\ell$. If $n\geq 3$ the
last statement of the theorem provides the following corollary:
\begin{corintro}
\label{cor:corintro}
  Let $n$ be an integer such that $n\geq 3$. Let $\Sigma$ be a smoothly
 immersed surface in $\RR^{2n}$, which is isotropic and
 diffeomorphic to a compact torus.
  Then, there exist piecewise
  linear  immersed surfaces  in $\RR^{2n}$, which are
  isotropic, homeomorphic
  to a compact torus and arbitrarily close to $\Sigma$ with respect to the
  Hausdorff distance.
\end{corintro}

\begin{rmk}
Our technique does not
allow to get much better results than a rather rough
$\cC^0$-closedness between $\ell$ and its approximation $\hat\ell$.
The best evidence for this weakness is the existence of a certain
\emph{shear action} on the space of isotropic quadrangular meshes
(cf. \S\ref{sec:shear}).
It would be most interesting to understand whether
these limitations are inherent to  the techniques we employed here,
or if there are geometric obstructions to get better
estimates. 
\end{rmk}

\subsubsection{Isotropic quadrangular meshes}
The main tool  to prove Theorem~\ref{theo:maindiscr}
relies on \emph{quadrangulations} of
$\Sigma$ and \emph{quadrangular meshes}.
Quadrangulations of $\RR^2$ are
 CW-complex decomposition of $\RR^2$,
where edges are line segments of $\RR^2$ and the boundary of every face
is an Euclidean quadrilateral. Nevertheless, the precise general definition of
quadrangulations 
is unimportant for our purpose. Indeed, we shall only work with
particular standard quadrangulations $\cQ_N(\RR^2)$ of $\RR^2$, pictured as a regular grid with
step size $N^ {-1}$ tiled by Euclidean squares. 

Particular Euclidean quadrangulations of 
$\cQ_N(\Sigma)$, are defined at \S\ref{sec:quadconv}. They
are obtained as quotients of $\cQ_N(\RR^2)$ by certain lattices $\Gamma_N$
of $\RR^2$. The associated moduli space of \emph{quadrangular meshes} is by definition
$$
\scrM_N = C^0(\cQ_N(\Sigma))\otimes\RR^{2n}.
$$
A mesh $\tau\in\scrM_N$ is an object that associates $\RR^{2n}$-coordinates
to every vertex of the quadrangulation $\cQ_N(\Sigma)$. A quadrilateral
of $\RR^{2n}$ is called an \emph{isotropic quadrilateral}, if the integral of the Liouville
form $\lambda$ along the quadrilateral vanishes. By extension, we say that a mesh $\tau\in
\scrM_N$ is isotropic if the quadrilateral in $\RR^{2n}$ associated to each  face of
$\cQ_n(\Sigma)$ via $\tau$ is isotropic.
The main strategy for proving Theorem~\ref{theo:maindiscr} involves the following approximation result:
\begin{theointro}
  \label{theo:mainquad}
  Given an isotropic immersion  $\ell:\Sigma\to\RR^{2n}$,  there exists a
family of isotropic quadrangular meshes $\rho_N\in\scrM_N$ defined for every $N$ sufficiently large, with the following property:
for every $\epsilon>0$, there exists $N_0>0$ such that for every $N\geq N_0$ and every
vertex $v$ of $\cQ_N(\Sigma)$, we have 
$$
 \|\rho_N(v)-\ell(v)\|\leq \epsilon.
$$
\end{theointro}
An isotropic quadrilateral of $\RR^{2n}$  is always the base of
an isotropic pyramid  in $\RR^{2n}$ (cf. \S\ref{sec:pyramid}), which is
easily
found as the solution
of a linear system. This remark allows to pass from an
isotropic quadrangular mesh to an isotropic triangular mesh. Together with Theorem~\ref{theo:mainquad} this provides essentially the proof of Theorem~\ref{theo:maindiscr}.

\subsubsection{Flow for quadrangular meshes}
Our approach for proving Theorem~\ref{theo:mainquad} has been
inspired to a large extent by the beautiful~\emph{moment map geometry}
introduced by Donaldson~\cite{Don99}. We shall provide a careful
presentation of this infinite dimensional  geometry at
\S\ref{sec:dream}, and merely state a few facts in this introduction:
the moduli space of maps
$$
\scrM=\{f :\Sigma\to\RR^{2n}\},
$$
from a surface $\Sigma$ endowed with a volume form $\sigma$ into
$\RR^{2n}$ admits a natural formal K\"ahler structure, with a formal Hamiltonian action of $\Ham(\Sigma,\sigma)$. The moment map of the action is given by
$$
\mu(f)=\frac{f^*\omega}{\sigma}.
$$
Zeroes of the moment map are precisely  isotropic maps.
It is tempting to make an 
analogy with the
Kempf-Ness theorem, which holds in the finite dimensional setting.
We may conjecture  that a
map $f$ admits an isotropic representative in its
\emph{complexified orbit} provided some type of algebro-geometric hypothesis of
stability.
Furthermore,  one can also define a \emph{moment map flow}, which is
naturally defined in the context of a K\"ahler manifold enfowed with a
Hamiltonian group action. Such flow is
essentially the downward gradient of the function $\|\mu\|^2$ on the
moduli space, which is expected, in favorable circumstances, to converge
toward a zero of the moment map in a prescribed orbit.

\begin{rmk}
We shall not state any significant  results aside  the description of
this geometric framework. For instance, it is an open question whether
the moment map flow exists for short time in this context, which is
part of a broader ongoing program.  
\end{rmk}

In an attempt to define a finite dimensional analogue of
this infinite dimensional moment map picture,
we define a flow  analogous to the moment map flow
on the moduli space of meshes $\scrM_N$, called the discrete moment
map flow. This flow is now just an ODE and its behavior can readily be explored
from a numerical perspective, using the Euler method.
We provide a computer program called DMMF, available on the homepage
\begin{center}
\href{http://www.math.sciences.univ-nantes.fr/~rollin/index.php?page=flow}{http://www.math.sciences.univ-nantes.fr/\textasciitilde{}rollin/},
\end{center}
which is
a numerical simulation of the discrete moment map flow.
 From an experimental point of view,
 the flow seems to be converging quickly toward isotropic
 quadrangular meshes, for any initial quadrangular mesh (cf. \S\ref{sec:dmmf}).

\subsection{Open questions}
Theorem~\ref{theo:maindiscr}
is a fundamental tool for the discrete geometry of isotropic tori,
since it provides a vast class of examples  of piecewise linear
objects by approximation of the smooth ones.
Here is a list a questions that arise immediately in this new
territory of discrete symplectic geometry:
\begin{enumerate}
\item Is there a converse to Theorem~\ref{theo:maindiscr} or
  Corollary~\ref{cor:corintro}? Given a piecewise linear isotropic
  surface in $\RR^{2n}$, is it possible to find a nearby smooth
  isotropic surface?
\item
More generally, to what extent does the moduli space of piecewise linear Lagrangian submanifolds
retain the properties of the moduli space of smooth Lagrangian submanifolds?
In spite of groundbreaking progess in symplectic topology, the classification of Lagrangian submanifolds up to Hamiltonian
    isotopy remains open. It is known that there exists several types of
    Lagrangian tori in $\RR^4$, which are not Hamiltonian
    isotopic:  namely,  product tori and  Chekanov
    tori~\cite{Chekanov96}. On the other hand, Luttinger found infinitely many
    obstructions in~\cite{Luttinger95} to the existence
    of certain type of knotted Lagrangian tori in $\RR^4$. In
    particular spin knots provide knotted tori in $\RR^4$ which cannot
    by isotopic to Lagrangian tori according to Luttinger's theorem.
    This thread of ideas led to the conjecture that product and
    Checkanov tori are the only classes of Lagrangian tori in $\RR^4$,
    up to Hamiltonian isotopy. Although the result was claimed before,
    the conjecture is  still open for the time being~\cite{D17}. However it was proved by Dimitroglou Rizell, Goodman and Ivrii that all embedded Lagrangian tori of $\RR^4$ are isotropic trough Lagrangian isotopies~\cite{DGI}.    Perhaps an interesting approach to
    tackle such conjecture, and more generally any questions involving
    some type of $h$-principle,  would be to recast the question in the finite
    dimensional framework of piecewise linear Lagrangian tori of $\RR^4$.
\item The moment map framework, in an infinite dimensional context,
  presented at \S\ref{sec:dream}, has been a great endeavor for
  proving our main results and introducing a finite dimensional version
  of the moment map flow. However, only a faint shadow of the moment
  map geometry is recovered in the finite dimensional world. More
  precisely, there exists a finite dimensional analogue $\mu_N^r$ of the moment map
  $\mu$ on $\scrM_N$. But it is not clear whether  $\mu_N^r$ is
  actually a
  moment map and for which group action on $\scrM_N$. It would be most interesting to
  define a finite dimensional analogue of the group $\Ham(\Sigma,\sigma)$,
  and try to make sense of the Kempf-Ness theorem in this setting.
\end{enumerate}

\subsection{Comments on the  proofs}
The proofs given in this paper come with a strong differential
geometric flavor, involving \emph{uniformization theorem} for Riemann
surfaces, \emph{discrete analysis}, \emph{discrete elliptic operators},
\emph{discrete Schauder estimates}, \emph{Riemannian geometry},
\emph{discrete spectral gap theorem}, \emph{Gau{\ss} curvature} and
its discrete analogues.

Many of the techniques employed here may be adapted to more general
settings. Working with tori seems to be a key
fact however: indeed, Fourier and discrete Fourier transforms
are well adapted for the analysis on tori and their quadrangulations
and do not seem to extend easily. The discrete Schauder estimates
derived by Thom\'ee~\cite{Thomee68}, which are a crucial ingredient of our fixed point principle, are
proved using Fourier transforms.

Although geometric analysis is quite often  a powerful tool for proving topological theorems,
 symplectic topologists may still 
expect more flexible constructions, like some
sort of jiggling lemma. We have not explored thoroughly  this approach,
 and would be more than happy to see an alternative
proof of our theorems along these lines.

\subsection{Organization of the paper}
Section~\S\ref{sec:dream} of this paper is a presentation of the moment map
geometry of a certain infinite dimensional moduli spaces introduced by
Donaldson. Finite dimensional analogues of this geometry are used in
the rest of the paper. For instance
a discrete
version of the moment map  flow is  introduced  at~\S\ref{sec:dmmf} and
implemented on a computer,  in order to obtain examples of Lagrangian
piecewise linear surfaces from an experimental point of view.
In \S\ref{sec:anal}, we introduce suitable spaces of discrete functions
on tori, together with the analysis suited for implementing
the fixed point principle. This section contains the definition of
quadrangulations, discrete functions,
discrete H\"older norms, together with the relevant notions of
convergence, culminating with a type of  Ascoli-Arzela theorem
(cf. Theorem~\ref{theo:ascoli}).
The equations for Lagrangian quadrangular meshes are introduced
at~\S\ref{sec:pert}, where their linearization is also computed. As
the dimension of the discrete problem goes to infinity, we
show that the finite dimensional linearized problem 
converges toward 
a smooth differential operator at~\S\ref{sec:limit}.
Some
uniform estimates on the spectrum of the finite dimensional  linearized problem are derived as a corollary.
The proof of Theorem~\ref{theo:mainquad} is completed
at~\S\ref{sec:fpt}, using the fixed point principle.
The proof of  Theorem~\ref{theo:maindiscr} follows and is completed at
\S\ref{sec:quadtri} after introducing some generic perturbations in
order to obtain piecewise linear immersions.

\subsection{Acknowledgements}
The authors would like to thank   Vestislav Apostolov, Gilles Carron,
Stéphane Guillermou, 
Xavier Saint-Raymond, Pascal Romon and Carl Tipler for some useful discussions.

A significant part of this research was carried out while the second author was
 working at the CNRS  UMI-CRM
research lab in Montr\'eal, during the academic year 2017-2018.
We are grateful to all the colleagues from l'Universit\'e du Qu\'ebec
\`a Montr\'eal (UQAM) and le Centre Inter-universitaire de Recherche
en G\'eom\'etrie et Topologie (CIRGET), for providing such a great scientific and human environment.

The  authors benefited from the support of the                  
French government
\emph{investissements d’avenir} program
   ANR-11-LABX-0020-01 and from 
   the
   \emph{défi de tous les savoirs} EMARKS research grant ANR-14-CE25-0010.

\tableofcontents
\newpage

\section{Dreaming of the smooth setting}
\label{sec:dream}
The main results of this work (for instance
Theorem~\ref{theo:mainquad}), are
 of  discrete geometric nature.
Yet the main ideas of our proof  were 
obtained via an  analogy with the moment map geometry
of the space of maps, from the tori endowed with a volume form, into
$\RR^{2n}$.
This section is independent of the others, but we think it is important to show how
smooth and discrete geometry analogies can be used to unravel unexpected ideas.

\subsection{Donaldson's moment map geometry}
The  moment map geometry presented here was coined by Donaldson,
although our
specific setting is not emphasized in~\cite{Don99}.
All the notions of moduli spaces shall be  discussed from a  purely
\emph{formal
perspective}.
With some additional effort, it may be possible to define
 infinite dimensional manifolds structures on  moduli spaces of
 interest,  by using suitable Sobolev or H\"older spaces. 

Let $M$ be a smooth manifold endowed with a K\"ahler
structure~$(M,J,\omega,g)$. The K\"ahler structure of $M$ is given by
an integrable almost complex structure~$J$, a K\"ahler form~$\omega$
and the corresponding K\"ahler metric~$g$. Recall that the K\"ahler
form is related to the metric via   
the usual formula
$$
\omega(v_1,v_2)=g(Jv_1,v_2), \mbox { for all $v_1,v_2 \in T_m M$}.
$$
 The reader may keep in mind the simplest example provided by $M=\RR^{2n}\simeq \CC^n$
 with its induced Euclidean K\"ahler structure. In this case,
 $\omega=d\lambda$, where $\lambda$ is the Liouville form, which
 implies that the cohomology class $[\omega]\in H^2(M,\RR)$ vanishes.

Let $\Sigma$ be a closed surface with orientation induced by a volume
form~$\sigma$. In real dimension $2$, a volume form is also a
symplectic form. Thus, the symplectic surface $(\Sigma,\sigma)$
 admits an infinite dimensional Lie group of Hamiltonian transformations denoted
 $$\cG=\Ham(\Sigma,\sigma).$$
 One can consider the moduli space of smooth
 maps
 $$
\scrM=\{ f : \Sigma \to M \quad | \quad f^*[\omega] = 0\};
$$
notice that in the case of $M=\RR^{2n}$ endowed with the standard
symplectic form, the condition $f^*[\omega]=0$
is satisfied by every smooth map.

The tangent space $T_f \cM$ is identified with the space of tangent vector
fields $V$ along~$f$, which is the space of smooth map $V:\Sigma\to TM$
such that $V(x)\in T_{f(x)}M$. 
 There is an obvious right-action of $\Ham(\Sigma,\sigma)$ on $\scrM$
 by precomposition.  

As pointed out by Donaldson, the geometry of the target space induces a formal Kähler structure on $\scrM$ denoted
$(\cM,\fg,\Omega,\cJ)$  given by
$$
(\fJ V)|_x = JV_x, \quad \fg(V,V') = \int_\Sigma g(V,V')\sigma , \quad \Omega(V,V') = \int_\Sigma \omega(V,V')\sigma
$$
for any pair of tangent vector fields $V,V'$ along $f:\Sigma\to\RR^4$.
By definition, the action of $\Ham(\Sigma,\sigma)$ preserves the
Kähler structure of $\cM$.

The canonical $L^2$-inner product on $\Sigma$, given by
$$\ipp{h,h'}=\int_\Sigma hh'\sigma ,$$
allows to define the space of smooth functions orthogonal to constants $C^\infty_0(\Sigma)$, which in turn, 
be identified to the Lie algebra $\Lie(\cG)$ of $\cG=\Ham(\Sigma,\sigma)$ via the map
$h\mapsto X_h$. Here $X_h$ is the Hamiltonian vector field with
respect to the symplectic form $\sigma$, which satisfies
$$
\iota_{X_h}\sigma = dh.
$$
The $L^2$-inner product $\ipp{h,h'}$ also provides an isomorphism between the
Lie algebra of $\Ham(\Sigma,\sigma)$ and its dual. The Lie algebra and
its dual will be identified throughout this section without any further warning.
Since $\Ham(\Sigma,\sigma)$ acts on $\scrM$, any element of the Lie algebra
$h\in \Lie(\cG)\simeq C^\infty_0(\Sigma)$ induces a \emph{fundamental
 vector field} $Y_h$  on
$\scrM$ defined by
$$
Y_h(f)=f_* X_h \in T_f\scrM.
$$

For $f\in\scrM$, we have  $f^*[\omega]=0$, hence
$$
\int_\Sigma
\frac{f^*\omega}\sigma \; \sigma= \int_\Sigma
f^*\omega =0,
$$
so that  we may consider the map
\begin{equation}
  \mu:\left\{\begin{array}{rcll}
    \scrM & \longrightarrow &  C^\infty_0(\Sigma) &\\
f& \longmapsto & \mu(f) = & \frac{f^*\omega}\sigma
  \end{array}\right.
\end{equation}
By definition, we have the obvious property
$$
\mu(f)=0 \Leftrightarrow f^*\omega = 0 \Leftrightarrow \mbox { $f$ is an  isotropic map}.
$$
But we have much more than an equation:
\begin{prop}[Donaldson]
  \label{prop:donaldson}
  The action of $\Ham(\Sigma,\sigma)$ on $\scrM$ is formally
  Hamiltonian and admits $\mu$ as a moment map. More precisely:
\begin{enumerate}
\item $\mu:\scrM\to C^\infty_0(\Sigma)$ is
  $\Ham(\Sigma,\sigma)$-equariant in the sense that for every
  $f\in\scrM$ and $u\in\Ham(\Sigma,\sigma)$
  $$
\ipp{\mu(f \circ u),h\circ u}  = \ipp{\mu(f ),h} ;
$$
\item for every variation $V$ of $f$
  $$
  \ipp{D\mu|_f\cdot V, h} = -\iota_{Y_h(f)} \Omega(V),
  $$
  where $D$ denotes the differentiation of functions on $\scrM$.
\end{enumerate}
\end{prop}
\begin{proof}
  Only the second property requires some explanation.
  We pick a smooth family of maps $f_t:\Sigma\to M$ such that
  $\frac{\del}{\del t} f_t|_{t=0}=V$ and $f_0=f$. The family is
  undersood as a smooth map
  $$
F:I\times\Sigma\to M
$$
where $I$ is a neighborhood of $0$ in $\RR$ and
$F(t,x)=f_t(x)$. We denote by $j_0:\Sigma \hookrightarrow
I\times \Sigma$ the canonical embedding given by $j_0(x)=(0,x)$. Notice
that by definition $F\circ j_0 =f$.
  Then
  \begin{align*}
  \ipp{D\mu|_f\cdot V, h} & = \left . \frac \del{\del t}\right |_{t=0} \int_\Sigma
  hf_t^*\omega\\
  &= \int_\Sigma h j_0^*\cL_{\del_t}\cdot F^*\omega \\
  &= \int_\Sigma h j_0^*(d\iota_{\del_t}F^*\omega +
  \iota_{\del_t}dF^*\omega)
  \end{align*}
where the last line comes from the Cartan formula. 
The symplectic form is closed, hence $dF^*\omega = F^*d\omega=0$. In
addition $F^*\del_t$ agrees with $V$ along $\{0\}\times\Sigma$, so
that $j_0^*\iota_{\del_t}dF^*\omega = df^*\iota_V\omega$.
It
follows that
\begin{align*}
  \ipp{D\mu|_f \cdot V, h} & = \int_\Sigma hd f^*\iota_V\omega\\
  &= - \int_\Sigma dh\wedge f^*\iota_V\omega\\
  &= - \int_\Sigma \iota_{X_h}\sigma \wedge f^*\iota_V\omega
\end{align*}
The interior product is an antiderivation. In particular
$$
\iota_{X_h}(\sigma \wedge f^*\iota_V\omega)= (\iota_{X_h}\sigma)
\wedge f^*\iota_V\omega +  (\iota_{X_h} f^*\iota_V\omega) \sigma.
$$
The LHS of the above identities must vanish since this is the case for
a $3$-form over a surface, and we obtain the identity
\begin{align*}
  \ipp{D\mu|_f\cdot V, h}  &= \int_\Sigma  (\iota_{X_h}
  f^*\iota_V\omega) \sigma\\
  &= \int_\Sigma  \omega(V,Y_h(f))\sigma\\
&=\Omega(V,Y_h(f)) 
  \end{align*}
which proves the proposition.
\end{proof}

\subsection{A moment map flow}
\label{sec:mmf}
From this point, gauge theorists may dream of generalizations of  the
Kempf-Ness Theorem, which is only known to hold in the finite
dimensional setting. The Kempf-Ness theorem asserts that the existence of a zero of the moment map in a given
complexified orbit of the group action is equivalent to an algebraic property of
stability, in the sense of geometric invariant theory. Under the
hypothesis of stability, the zeroes of the moment map are unique up to
the action of the real group. Unfortunately the Kempf-Ness Theorem
does not apply immediately  in the infinite
dimensional setting and the conjectural isomorphism
$$
 \scrM /\!\!/ {\cG}^{\CC} \simeq \mu^{-1}(0)/\cG,
 $$
 where the LHS is some type of GIT quotient,
is out of reach for the moment. To start with, the complexification of
$\cG$ is 
not even well defined and the quotient  $\scrM /\!\!/ {\cG}^{\CC}$ does not make sense. 
Nevertheless, a significant number of this thread of ideas may be implemented.
For instance, we may define a natural \emph{moment map flow}.
\begin{dfn}
  \label{dfn:mmf}
Let $f_t\in\cM$ be a family of  maps, for
$t$ in an open interval of $\RR$. We say that the family $f_t$ is solution of the
moment map flow if
\begin{equation}
  \label{eq:mmf}
\frac {df}{dt} = \fJ Y_{\mu(f)}(f).  
\end{equation}
\end{dfn}
\begin{rmk}
In the finite dimensional setting, such a moment map flow preserves the complexified
orbits and converges to a zero of the moment map under a suitable
assumption of stability. It is natural to conjecture that this flow
should converges generically to an isotropic  map in a prescribed
complexified orbit. We shall not tackle this problem here and only
prove some very down to earth properties of the flow.   
\end{rmk}
By construction, we have
\begin{align*}
  \fg(\fJ Y_{\mu(f)}(f),V) &= \Omega(Y_{\mu(f)},V) \\
  &= -\ipp{D\mu|_f\cdot V,\mu(f)}\\
& =-\frac 12 D(\|\mu\|^2)|_f\cdot V  
\end{align*}
so that
$$
\fJ Y_{\mu(\cdot)} = -\frac 12 \grad \|\mu \|^2,
$$
which proves the following lemma:
\begin{lemma}
  \label{lemma:flowdown}
Smooth maps $f:\Sigma\to M$ are zeroes of the moment map if, and
only if they are isotropic.
In addition, the moment map flow is a downward gradient flow of the functional
$f\mapsto \|\mu(f)\|^2$ on $\scrM$. More precisely, the evolution equation of the moment map flow can be written
 $$
\frac{df}{dt} = -\frac 12  \grad \|\mu(f)\|^2,
$$
where $\grad$ is the gradient of a function on $\scrM$ endowed with its Riemannian metric $\fg$.
\end{lemma}
As a corollary, we see that the functional should decrease along flow lines:
\begin{cor}
  If $f_t$ is a smooth family of maps solution of the moment map flow, then
  $$
\frac {d}{dt}\|\mu(f_t)\|^2 <0
$$
unless $f_t$ is isotropic, in which case $\frac {d}{dt}\|\mu(f_t)\|^2=0$ and the flow is stationary.
\end{cor}
\begin{proof}
  Assume that $f_t$ is not isotropic. In particular there exists
  $x\in \Sigma$ such that the differential of $\mu(f_t)$ does not
  vanish at $x$. Otherwise $\mu(f_t)$ would be constant. But the
  fact that $\omega$ is exact would force $\mu(f_t)=0$. By
  definition $X_{\mu(f_t)}$ is a non vanishing vector field at $x$
  since it is the symplectic dual of $d\mu(f_t)$. It follows that
  $Y_{\mu(f_t)}$ does not vanish hence
 \begin{align*}
    \frac{d}{dt}\|\mu(f_t)\|^2 
    &= -2\fg(\fJ Y_{\mu(f_t)}, \fJ Y_{\mu(f_t)}) \\
    &= -2\fg( Y_{\mu(f_t)}, Y_{\mu(f_t)}) <0.
  \end{align*}  
\end{proof}

\subsubsection{Laplacian and related operators}
\label{sec:laplrel}
For each  vector field $V$ tangent to $f$, we define the operator
$$\delta_f:T_f\scrM\to C^\infty_0(\Sigma)
$$
by
\begin{equation}
  \label{eq:deltaf}
\delta_f V = - D\mu|_f\cdot JV.  
\end{equation}
We see that that the adjoint $\delta_f^\star$ of $\delta_f$ satisfies
\begin{align*}
  \fg(\delta^*_fh,V)   & =  \ipp{\delta_f V, h}\\
  &= -\ipp{D\mu|_f\cdot JV,h} \\
  &= \Omega(Y_h(f), JV)\\
  & = \fg(Y_h(f),V) 
\end{align*}
so that
\begin{equation}
  \label{eq:hamdstar}
\delta^\star_f h = Y_h(f).  
\end{equation}
For each $f\in\scrM$, we may define a natural Laplacian
\begin{equation}
  \label{eq:Deltaf}
\Delta_f = \delta_f\delta^\star_f  
\end{equation}
acting on smooth functions on $\Sigma$.

\begin{rmk}
It seems likely that the moment map flow of Definition~\ref{dfn:mmf}
can be interpreted as a parabolic flow, once a suitable 
analytical framework
and gauge condition have been setup. Although we shall not prove
anything about short time existence of the moment map flow in this work, we provide at least a heursitic
 evidence. In the next section, 
we compute the variation of the moment map and
show that the variation of $\mu(f)$,
when
  $f$ is deformed in the direction of the complexified action $\fJ Y_h$, is expressed as a
  Laplacian of $h$. However, the
  systematic study of the moment map flow in the smooth setting is not our purpose here, and
  we shall return to this question in a sequel to this paper~\cite{JRT}.  
\end{rmk}

\subsection{Variations of the moment map}
The operator
$f\mapsto \mu(f)$ is a first order differential operator.
 This section is devoted to calculate its linearization.

\subsubsection{General variations}
Let $f_s:\Sigma\to M$ be a smooth family of maps, with
parameter $s\in I$, where $I$ is an open interval of $\RR$.
We use the notation,
$$
V_s=\frac{\del f_s}{\del s},
$$
for the infinitesimal variation  $V_s\in T_{f_s}\scrM$
 of the
family $f_s$.

We consider the map $F:I\times \Sigma\to M$ given by $F(s,x)=f_s(x)$ and the canonical injection
$j_{s_0}:\Sigma \hookrightarrow I\times\Sigma$, defined by $j_{s_0}(x)=(s_0,x)$ for some $s_0\in I$.
We compute, using the Cartan formula
\begin{align*}
  \left .\frac{\del f_s^*\omega}{\del s}\right |_{s=s_0} & = j_{s_0}^* \frac{\del}{\del s}\cdot F^*\omega \\
  & =  j_{s_0}^*(d\circ \iota_{\del_s} +\iota_{\del_s}\circ d)F^*\omega \\
  & = j_{s_0}^* d\circ \iota_{\del_s} F^*\omega \\
  & = j_{s_0}^* d F^*\iota_{V_s}\omega \\
  & = df_{s_0}^*\iota_{V_{s_0}}\omega 
\end{align*}
where we have used the fact that $d\omega=0$, that $d$ commutes with
pullbacks and that $F\circ j_{s_0} = f_{s_0}$.

The form
$$\alpha_{s_0} = f^*_s\iota_{V_{s_0}}\omega
$$
is called the
\emph{Maslov form} of the deformation $f_s$ at $s=s_0$.
The above computation shows that
$$
 \frac{\del f_s^*\omega}{\del s} = d\alpha_s,
 $$
which reads
  $$
 \frac{\del \mu(f_s)}{\del s} = \delta \alpha_s
 $$
 where $\delta$ is the operator  given by
 \begin{equation}
   \label{eq:deltadef}
   \delta \gamma = \frac{d\gamma}\sigma,
 \end{equation}
 for every $1$-form $\gamma$ on $\Sigma$.

Thus we have proved the following result:
\begin{lemma}
  \label{lemma:genvarmmp}
  Let $f:\Sigma\to M$ be a smooth map and $V\in T_f\scrM$ be an infinitesimal  variation of $f$.
  Then
  $$
D\mu|_f\cdot V = \delta \alpha_V
$$
where $\alpha_V= f^*\iota_V\omega$ is the Maslov form of the deformation and $\delta$ is the operator defined by~\eqref{eq:deltadef}.
\end{lemma}


\subsubsection{Variations at an immersion}
We assume now that $f:\Sigma\to M$ is a smooth immersion. In
particular the pullback $g_\Sigma=f^*g$ is a Riemannian metric on
$\Sigma$. The volume form $\vol_\Sigma$ may not agree with $\sigma$,
but the $2$-forms are related by a conformal factor
$$
\vol_\Sigma = \theta\sigma
$$
where $\theta:\Sigma\to \RR$ is a positive smooth function.
We introduce a conformal metric $g_\sigma$ that satisfies the equation
$$
g_\Sigma =\theta g_\sigma,
$$
and  the Hodge operator acting on forms of
$\Sigma$, associated to the metric $g_\sigma$ is denoted $*_\sigma$.
\begin{lemma}
  \label{lemma:varmu}
  Assume that  $f:\Sigma\to M$ is a smooth immersion. Then
  $\Sigma$ has an induced Riemannian metric $g_\Sigma$. Let
  $g_\sigma$ be the Riemannian metric with volume form $\sigma$, conformal to $g_\Sigma$.
  Let
  $Y_h$ be the fundamental vector field on $\scrM$ associated to the Hamiltonian function $h$.
   Then
  $$
\Delta_fh=\delta_f\delta^\star_f h= -D\mu|_{f}\cdot \fJ Y_h(f) =   d^{*_\sigma}\theta d h = \theta\Delta_\sigma
h - g_{\sigma}(d\theta,dh) .
  $$
  where $\Delta_\sigma$ is the Laplacian associated to the Riemannian
  metric $g_\sigma$ and $\theta$ is the conformal factor such that
  $g_\Sigma=\theta g_\sigma$.
\end{lemma}
\begin{rmk}
  In particular, if $\vol_\Sigma$ agrees with $\sigma$, then
  $\theta=1$, $g_\sigma=g_\Sigma$ and the above formal says that
  $$
 \Delta_f h = \Delta_\Sigma h.
  $$
\end{rmk}
\begin{proof}
  Let $f_s\in \scrM$, be a smooth family of maps for $s\in
  I=(-\epsilon,\epsilon)$, with the
  property that $f_0=f$ and $V_0 = JY_h = Jf_*X_h$.
    Then
  $\frac{\del\mu(f_s)}{\del_s} = \delta \alpha_s$ by Lemma~\ref{lemma:genvarmmp}.
  But $\alpha_0(U)= f ^*\omega(V_0,U)=
  \omega(Jf_*X_h,f_* U)= -g( f_*X_h,f_*
  U)=-g_\Sigma(X_h,U)$.
  It follows that
  $\alpha_0(U)=  - \theta
  g_\sigma(X_h,U) = -\theta \sigma(X_h,*_\sigma U) = \theta *_\sigma dh.$
Then we conclude that
$$
\left .\frac{\del \mu(f_s)}{\del s}\right |_{s=0} = *_\sigma
d\theta *_\sigma dh = - d^{*_\sigma}\theta d h = -\theta\Delta_\sigma
h + g_{\sigma}(d\theta,dh) ,
$$
which proves the Lemma.
\end{proof}
The next lemma shows that $\Delta_f$ is essentially an isomorphism.
\begin{lemma}
  \label{lemma:isodelta}
  The operator $h\mapsto d^{*_\sigma}\theta dh$ is an elliptic
  operator of order $2$, which is an isomorphism modulo constants.
\end{lemma}
\begin{proof}
  The fact that the operator is elliptic of order $2$ follows from the
  formula
  $$
d^{*_\sigma}\theta dh = \theta\Delta_\sigma h -g_\sigma(d\theta,dh).
$$
The operator is selfadjoint since
$$
\ipp{d^{*_\sigma}\theta dh , h'}= \ipp{\theta dh,d h'} = \ipp{h,d^{*_\sigma}\theta dh'}.
$$
If $h$ belongs to the  kernel of the operator, then
$$
0=\ipp{d^{*_\sigma}\theta dh,h} = \ipp{\theta dh,dh}
$$
which implies that $h$ is constant. Because the operator is
selfajoint, the orthogonal of its image is identified to the
kernel. So the operator is an isomophism when restricted to functions
which are $L^2$-orthogonal to constants.
\end{proof}

\subsection{Application}
We know that $\scrM$ is acted on by $\cG=\Ham(\Sigma,\sigma)$. The
$\cG$-orbit of $f\in\scrM$ is denoted $\cO_f$. The group of Hamiltonian transformations does not
admit a natural complexification. Nevertheless, it is possible to make sense of its complexified orbits.

The space of vector fields $Y_u$ defined by
$Y_u(f)=f_*X_u$ over $\scrM$ defines an integrable distribution
$\cD\subset T\cM$ which is the tangent space to $\cG$-orbits.
We can consider the complexified distribution of the tangent bundle to $\scrM$
$$
\cD^\CC = \cD +\fJ \cD 
$$
given by vector fields of the form $Y_u+\fJ Y_v$.
The fact
that $\cG$ preserves the complex structure $\fJ$ of $\scrM$ implies
that the distribution is formally integrable
into a holomorphic foliation. A leaf of the foliation, obtained by
integrating the distribution, is refered to as a
\emph{complexified orbit of $\cG$}. The complexified orbit of a
element $f\in\scrM$ is denoted $\cO_f^\CC$. 

We are now assuming for simplicity that $M$ is the K\"ahler manifolds
$\RR^{2n}$ identified to $\CC^{n}$. In this case, given $f\in \cM$, we may consider a type of exponential map given by
$$
\exp_f (u+iv)= f + Y_u(f) +\fJ Y_v(f),
$$
for $u, v\in C^\infty_0(\Sigma)$.
This type of exponential map does not come from a Lie group exponential
map. However, $\exp_f(u+iv)$
provides perturbations of $f$ in directions tangent to the
complexified orbit $\cO_f^\CC$ at $f$.

We have now all the tools necessary to prove the following result:
\begin{theo}
  \label{theo:perttoy}
  We choose $M=\RR^{2n}$ for the construction of $\scrM$.
  Let $\ell\in\cM$ be a smooth isotropic immersion. If $f\in\scrM$
  is sufficiently close to $\ell$ in $\cC^{1,\alpha}$-norm, there exists a nearby 
  perturbation of the form $\tilde \ell=\exp_f(ih)$, where $h$ is a
  $\cC^{2, \alpha}$ function on $\Sigma$, such that
  $\tilde \ell$ is an
  isotropic immersion.
\end{theo}

\begin{proof}
  We denote by $\cC^{k,\alpha}$, for some Hölde parameter $\alpha>0$, the
  usual Hölder spaces. The moduli space $\scrM$ is  replaced with
  $\scrM^{k,\alpha}$ which consists of $\cC^{k,\alpha}$-maps $f:\Sigma\to\RR^{2n}$.
Since $\scrM^{k,\alpha}$ is an affine space  modeled on
$\cC^{k,\alpha}$, it is naturally endowed with an infinite dimensional
manifold structure. In particular, the map $\exp_f$ defines a smooth
map
$$
\exp : \cC^{k+1,\alpha}(\Sigma,\CC) \times
\scrM^{k,\alpha}\longrightarrow \scrM^{k,\alpha}.
$$
given by $(h,f)\mapsto \exp_f(h)$.

We denote by $\cC^{k,\alpha}_0(\Sigma)$ the subspace of
$\cC^{k,\alpha}(\Sigma)$ that consists of real valued functions
$h:\Sigma\to \RR$
such that $\int_{\Sigma}h\sigma =0$ (i.e. functions orthogonal to
constants for the inner product $\ipp{\cdot,\cdot}$).
We consider the map
  $$
  Z:\left\{\begin{array}{ccl}
    \cC^{2,\alpha}_0(\Sigma)\times \scrM^{1,\alpha} &\longrightarrow & \scrM^{1,\alpha} \\
    (h,f) &\longmapsto &\exp_{f}(-ih)
  \end{array}\right.
  $$
whose  differential  at $(0,\ell)$ satisfies
$$
\frac{\del Z}{\del h}|_{(0,\ell)} \cdot \dot h = -\fJ Y_{\dot h}(\ell)
$$
by definition of the exponential map.
In particular
\begin{equation}
  \label{eq:implicit}
\frac{\del(\mu\circ Z)}{\del h}|_{(0,\ell)}\cdot \dot h=
- D\mu|_{\ell}\cdot\fJ Y_{\dot h}(\ell) = \delta_\ell
\delta^\star_\ell h = \Delta_\ell h  
\end{equation}
by Lemma~\ref{lemma:varmu}. This operator is an isomorphism modulo
constants by Lemma~\ref{lemma:isodelta}.
We consider the map
$$
F:\cC^{2,\alpha}_0(\Sigma)\times\scrM^{1,\alpha}\to \cC^{0,\alpha}_0(\Sigma)
$$
given by $F=\mu\circ Z$. We have proved that the differential
$$
\frac{\del F}{\del h}|_{(0,\ell)}: \cC^{2,\alpha}_0(\Sigma)\longrightarrow \cC^{0,\alpha}_0(\Sigma)
$$
is an isomorphism.
The rest of the proof follows from the implicit function theorem: for
every $f\in \scrM^{1,\alpha}$ sufficiently close to $\ell$ in
$\cC^{1,\alpha}$-norm, there exists a unique $\tilde h=h(f)\in \cC^{2,\alpha}(\Sigma)$
in a small neighborhood of the origin, such that
$$
F(\tilde h, f)=0.
$$
By definition $\exp_{f}(i\tilde h)=\tilde \ell$
 satisfies $\mu(\tilde\ell)=0$. By assumption $\ell$ is smooth. If $f$
 is also smooth, elliptic regularity and standard bootstrapping
 argument shows that $\tilde h$, and in turn $\tilde \ell$,  must be smooth as well. This proves the theorem.  
\end{proof}
\begin{rmk}
  In section~\ref{sec:pert}, we will develop a perturbation theory on
  the space of quadrangular meshes $\scrM_N$ that mimics
  Theorem~\ref{theo:perttoy}.
  We shall define an analogue $\delta_\tau$ of the operator $\delta_f$
  in the context of discrete geometry (cf. Formula~\eqref{eq:deltatau}).
  The operator $\delta_\tau$, and more precisely its adjoint
  $\delta^\star_\tau$, could be used to define an analogue of
  Hamiltonian vector fields in the context of discrete differential geometry, in
  view of Formula~\eqref{eq:hamdstar}. This could be relied upon to define a
  discrete analogue of the gauge group action
  $\cG=\Ham(\Sigma,\sigma)$. This idea will be explored in a sequel
  to this work~\cite{JRT}.
\end{rmk}

\subsubsection{Outreach}
In \S\ref{sec:pert} we shall define finite dimensional analogues of the
infinite dimensional moment map picture presented in the current section,
provided .
This will provide the incomplete dictionary below, where the right column, is
conjecturally a finite dimensional approximation of the left column:
\begin{center}
  \begin{tabular}{|c|c|}
    \hline
    Infinite dimensional case & finite dimensional case \\
    \hline
Area form $\sigma$ on $\Sigma$ & Quadrangulation $\cQ_N(\Sigma)$  \\
    \hline
    $\scrM=\{f:\Sigma\to\RR^{2n}\}$ &  $\scrM_N =
    C^0(\cQ_N(\Sigma)\otimes\RR^{2n}$ \\
    \hline
Canonical K\"ahler structure &   Canonical K\"ahler structure\\
    \hline
    $\Ham(\Sigma,\sigma)$-action & ???\\
    \hline
    Fundamental V.F $Y_h(f)=\delta^\star_f h$  & $\delta^\star_\tau \phi$\\
    \hline
    A moment map $\mu:\scrM\to C^\infty(\Sigma)$ &  $\mu_N^r:\scrM_N
    \to C^2(\cQ_N(\Sigma))$\\
    \hline
    The moment map flow \eqref{eq:mmf} & The discrete flow
    \eqref{eq:discrf} \\
    \hline
  \end{tabular}
\end{center}
Many aspects of the above dictionary remain unclear. First, the finite
dimensional picture does not come with a Lie group action that would,
in some sense, approximate $\Ham(\Sigma,\sigma)$. In particular
$\mu_N^r$ is not a moment map and $C^2(\cQ_N(\Sigma))$ is not
interpreted as a Lie algebra.  The flows are defined on both sides and
we would like to compare them as $N$ goes to infinity. Unfortunately,  we do not
even know whether the infinite dimensional flow exists for short time. The
discrete flow is an ODE, but it is not completely understood at this stage. For $N$ fixed, does
the flow converge, or does it blowup  ? Does a sequence
of flow converge to the moment map flow as $N$ goes to infinity ? Can we use the above sketch
of correspondence to make sense of some type of Kempf-Ness theorem in
the infinite dimensional setting ? 

All these gripping questions are  postponed
  to a later work. In this paper, we focus on the discrete flow on $\scrM_N$,
  for a given $N$, and merely provide a computer simulation of the
  discrete flow at \S\ref{sec:dmmf}.

\section{Discrete analysis}
\label{sec:anal}
In this section, we consider a real surface $\Sigma$, diffeomorphic to a
 torus. We denote by $g$ the canonical Euclidean metric of
$\RR^{2n}$ and $J$ the standard complex structure deduced from the
identification $\RR^{2n}\simeq \CC^n$. The standard symplectic form of $\RR^{2n}$ is
given by $\omega(\cdot,\cdot)=g(J\cdot,\cdot)$ and  $\ell:\Sigma\to \RR^{2n}$ is an
\emph{isotropic immersion}.

\subsection{Conformally flat metric}
\label{sec:confflat}
Every Riemannian metric on a surface diffeomorphic to a  torus is \emph{conformally flat}.
In particular,  $\Sigma$ carries a pullback Riemannian metric
$$
g_\Sigma=\ell^*g,
$$
which must be  conformally flat. In other words, there exists a
covering map
\begin{equation}
  \label{eq:p}
  p:\RR^2\to\Sigma
\end{equation}
with deck transformations given by a lattice $\Gamma\subset
\RR^2$. The Euclidean metric $\geuc$ of $\RR^2$ descends as a flat
metric $g_\sigma$ on $\Sigma$. In addition there exists   a positive
smooth function $\theta  :\Sigma\to (0,+\infty)$,
 known as the \emph{conformal factor}, 
such that
$$g_\Sigma = \theta g_\sigma.$$
The projection $p$, which descends to the quotient $\RR^ 2/\Gamma $, provides a preferred diffeomorphism
\begin{equation}
  \label{eq:ptilde}
  \Phi : \RR^ 2/\Gamma \to \Sigma,
\end{equation}
which is also an isometry from $(\RR^ 2/\Gamma,\geuc)$ to $(\Sigma,g_\sigma)$.

\subsection{Square lattice and checkers board}
\label{sec:sqlat}
Let  $e_1=(1,0)$ and $e_2=(0,1)$ be the canonical basis of
$\RR^2$. The basis $(e_1,e_2)$ is orthonormal with respect to the
canonical Euclidean metric $\geuc$ of $\RR^2$ and it is
positively oriented, by convention. 

For every positive integer $N$, we introduce the  lattice
$\Lambda_N\subset \RR^2$ spanned by
$e_1/N$ and $e_2/N$:
$$
\Lambda_N = \ZZ\cdot \frac{e_1}N\oplus \ZZ\cdot \frac{e_2}N  \subset  \RR^2.
$$
The lattice $\Lambda_N$ provides the familiar picture of a square grid
in $\RR^2$ with step size~$N^{-1}$. 
The lattice $\Gamma$, introduced at \S\ref{sec:confflat}, admits a basis
 $(\gamma_1,\gamma_2)$, compatible with the canonical orientation of $\RR^ 2$. The lattice $\Gamma$ is generally not a sublattice
of $\Lambda_N$. Indeed, the components of the vectors $\gamma_1, \gamma_2\in\RR^ 2$ may not be rational. This fact will cause a technical catch
for constructing quadrangulations of $\Sigma$. Luckily this difficulty is easily overcome
as we shall explain below.
The \emph{checkers board} sublattice
$\Lambda^{ch}_N\subset \Lambda_N$ 
is spanned by the vectors $\frac
{e_1+e_2}N$ and  $\frac
{e_2-e_1}N$:
$$
\Lambda_N^{ch}= \ZZ\cdot \frac{e_1+e_2}N \oplus \ZZ\cdot \frac{e_2-e_1}N \subset\Lambda_N.
$$
The elements of $\Lambda_N\subset \RR^ 2$ may be  thought of as the
positions of a standard
checkers board game. Then $\Lambda_N^{ch}$ acts on $\Lambda_N$ by translations. These translations are spanned by 
diagonal motions, as in some kind of checkers game. One can easily see that the quotient
$\Lambda_N/ \Lambda^{ch}_N$ is isomorphic to $\ZZ_2$ which is
isomorphic to the equivalence classes
of the usual black and white positions of the
checkers board game.

For each $N>0$ and $i=1,2$, we choose   $\gamma_i^ N\in\Lambda^{ch}_N$ which is a best approximation of 
$\gamma_i$ in~$\Lambda_{N}^{ch}$, for the Euclidean distance in
$\RR^2$. By definition, $\gamma_1^N$ and $\gamma_2^N$ are linearly
independent for all sufficiently large $N$. We define the lattice $\Gamma_N$, at least for sufficiently large $N$,  as
$$
\Gamma_N=\ZZ \cdot \gamma_1^ N \oplus \ZZ\cdot \gamma_2^N  \subset \Lambda_N^ {ch}\subset \Lambda_N.
$$
We summarize our construction in Figure~\ref{figure:gammaN}. The red and blue bullets 
represent the elements of $\Lambda_N$, where  the red bullets are in
 $\Lambda_N^{ch}$. We draw the generators $\gamma_i$ of
$\Gamma$ and their best approximations, in red, by elements
$\gamma_i^N$ of~$\Lambda_N^{ch}$:
  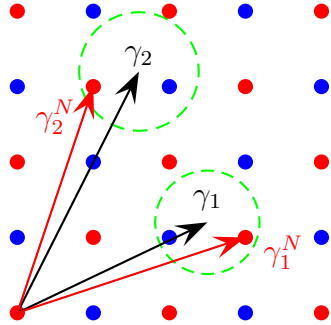
\begin{figure}[H]
  \begin{pspicture}[showgrid=false](0,0)(4,4)
     \psset{linecolor=green, linestyle=dashed}
         \pscircle(1.6,3.2){0.8}
         \pscircle(2.5,1.2){0.7}

         \color{black}
         \rput(2.5,1.5){$\gamma_1$}
         \rput(1.6,3.4){$\gamma_2$}
         \color{red}
         \rput(3.5,.8){$\gamma^N_1$}
         \rput(0.5,2.6){$\gamma^N_2$}

    \psset{linecolor=red, fillstyle=solid , fillcolor=red, linestyle=solid}
     \pscircle (0,0){.1}
     \pscircle (0,2){.1}
     \pscircle (0,4){.1}
     \pscircle (1,1){.1}
     \pscircle (1,3){.1}
     \pscircle (2,0){.1}
     \pscircle (2,2){.1}
     \pscircle (2,4){.1}
     \pscircle (3,1){.1}
     \pscircle (3,3){.1}
     \pscircle (4,0){.1}
     \pscircle (4,2){.1}
     \pscircle (4,4){.1}

       \psset{linecolor=blue, fillstyle=solid , fillcolor=blue, linestyle=solid}
     \pscircle(0,1){.1}
     \pscircle (0,3){.1}
     \pscircle (1,0){.1}
     \pscircle (1,2){.1}
     \pscircle (1,4){.1}
     \pscircle (2,1){.1}
     \pscircle (2,3){.1}
     \pscircle (3,0){.1}
     \pscircle (3,2){.1}
     \pscircle (3,4){.1}
     \pscircle (4,1){.1}
     \pscircle (4,3){.1}

         \psset{linecolor=black, arrows=->,arrowsize=.3,linestyle=solid}
     \psline(0,0)(1.6,3.2)
         \psline(0,0)(2.5,1.2)

         \psset{linecolor=red}
         \psline(0,0)(1,3)
         \psline(0,0)(3,1)

  \end{pspicture}
  \caption{Construction of $\Gamma_N$}
  \label{figure:gammaN}
  \end{figure}
  By construction,  $\Gamma_N$ is  a sublattice of
  $\Lambda_N^{ch}$; this choice has been designed so that the checkers
  graph splits into two connected components precisely~(cf. \S\ref{sec:check}).
  Furthermore, the lattices $\Gamma_N$ converge towards $\Gamma$,
  in a sense to be made more precise now:
  the linear transformation $U_N$ of
$\RR^2$ defined by
$$
U_N(\gamma_i^N)= \gamma_i
$$
identifies  the lattices
$\Gamma_N$ and $\Gamma$ by an automorphism of $\RR^2$.
Using an operator norm for linear transformations of $\RR^2$, we have 
\begin{equation}
  \label{eq:unasympt}
\|U_N - \id|_{\RR^ 2}\| = \cO\left (N^ {-1}\right ).
\end{equation}
In conclusion  $U_N$ converges towards the identity and
$U_N(\Gamma_N)=\Gamma$, which is understood as $\Gamma_N$ converges
towards $\Gamma$.

By construction,  $U_N$   descends to the
quotient as a (locally linear) diffeomorphism
$$
u_N:\RR^2/\Gamma_N\to \RR^2/\Gamma.
$$
The linear transformation $U_N$ may not belong to the orthogonal
group. Therefore neither $U_N$ nor $u_N$ are isometries. But,
they are isometries in the limit, since $U_N$ converges to the
identity. This fact will be sufficient for our purpose.
The quotients $\RR^2/\Gamma$ and $\RR^2/\Gamma_N$ are canonically identified to $\Sigma$ via the diffeomorphisms
$$
  \RR^2/\Gamma_N
\stackrel{u_N}\longrightarrow  \RR^2/\Gamma  \stackrel{\Phi}\longrightarrow \Sigma.
$$
There are now several competing covering maps: we defined $p:\RR^ 2\to
\Sigma$ at~\eqref{eq:p}, but we may also consider the covering maps
\begin{equation}
  \label{eq:pN}
  p_N= p\circ U_N:\RR^ 2\to \Sigma.
\end{equation}
The group of  deck transformation of $p$ is $\Gamma$, whereas the group of deck transformations of $p_N$ is $\Gamma_N$.
There are also several
 flat metrics  descending on $\Sigma$ via $p$ and $p_N$. The first
$g_\sigma$ is induced by the Euclidean metric and the diffeomorphism
$\Phi : \RR^2/\Gamma \to \Sigma$. The other flat metrics $g_{\sigma}^ N$ are induced by the
 Euclidean metric and the diffeomorphisms
 \begin{equation}
   \label{eq:ptildeN}
\Phi_N :\RR^2/\Gamma_N
\to \Sigma
 \end{equation}
induces by~\eqref{eq:pN}.
According to \eqref{eq:unasympt} we have
$$
g_{\sigma}^ N= g_\sigma + \cO\left (N^ {-1}\right ).
$$

\subsection{Quadrangulations}
\label{sec:quadconv}
Instead
of  linear triangulations, we shall work with  particular
\emph{linear quadrangulations} $\cQ_N(\Sigma)$ of $\Sigma$. The
current section is devoted to the definition of these CW-complexes.

\subsubsection{Quadrangulations of the plane}
\label{sec:quadnot}
For $k,l\in\ZZ$, the points of $\RR^2$ given by
$$
\vert_{kl} = \frac kN e_1 + \frac lN e_2 ,
$$
are  the
elements of the lattice $\Lambda_N\subset\RR^2$.
 The elements of the lattice $\Lambda_N$ are also the vertices of a
 nice quadrangulation $\cQ_N(\RR^2)$ of the plane $\RR^2$, pictured as
 the  usual square grid with step $N^{-1}$.
 More precisely, the   quadrangulation $\cQ_N(\RR^2)$ is a particular
 CW-complex decomposition of $\RR^2$,
characterized by the following properties:
\begin{itemize}
\item The edges $\edge_{1,kl}$ and $\edge_{2,kl}$ of the quadrangulation are
  the oriented line segments of $\RR^2$ with
  oriented boundary
  $$\del \edge_{1,kl}= \vert_{k+1,l}-\vert_{kl} \mbox{   and  } \del
  \edge_{2,kl}=\vert_{k,l+1}-\vert_{kl}.
  $$
\item 
 The faces $\face_{kl}$ of the quadrangulation are oriented squares of
 $\RR^2$ with oriented boundary 
 $$ \del \face_{kl}=\edge_{1,kl}+\edge_{2,k+1,l} - \edge_{1,k,l+1} - \edge_{2,kl}.$$
\end{itemize}
Figure~\ref{figure:tiling} shows the familiar picture of the plane
tiled  by  squares together with the notations introduced above.
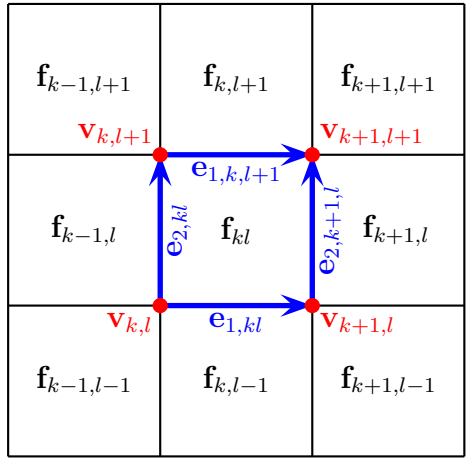
\begin{figure}[H]
\begin{pspicture}[showgrid=false](-3,-3)(3,3)
                          \psscalebox{1.0}
           {
             \psline (-3,3) (3,3)
\psline (-3,1) (3,1)
\psline (-3,-1) (3,-1)
\psline (-3,-3) (3,-3)

\psline (-3,3) (-3,-3)
\psline (-1,3) (-1,-3)
\psline (1,3) (1,-3)
\psline (3,3) (3,-3)

\psset{linecolor=blue, arrows=->,arrowsize=.3, linewidth=2pt,linestyle=solid}
\psline{->}(-1,-1) (-1,1)
\psline{->}(-1,-1) (1,-1)
\psline{->}(-1,1) (1,1)
\psline{->}(1,-1) (1,1)

\psset{linecolor=red, linewidth=2pt,linestyle=solid, fillstyle=solid, fillcolor=red}
\rput(-2,-2){$\face_{k-1,l-1}$}
\rput(0,-2){$\face_{k,l-1}$}
\rput(2,-2){$\face_{k+1,l-1}$}

\rput(-2,0){$\face_{k-1,l}$}
\rput(0,0){$\face_{kl}$}
\rput(2.1,0){$\face_{k+1,l}$}

\rput(-2,+2){$\face_{k-1,l+1}$}
\rput(0,+2){$\face_{k,l+1}$}
\rput(2,+2){$\face_{k+1,l+1}$}

\color{red}
\pscircle(-1,-1){.1}
\rput[tr](-1.1,-1.1){$\vert_{k,l}$}
\pscircle(1,-1){.1}
\rput[tl](1.1,-1.1){$\vert_{k+1,l}$}

\pscircle(-1,+1){.1}
\rput[br](-1.1,1.1){$\vert_{k,l+1}$}
\pscircle(1,1){.1}
\rput[bl](1.1,1.1){$\vert_{k+1,l+1}$}

\color{blue}
\rput[t](0.0,-1.1){$\edge_{1,kl}$}
\rput[t](0.0,+0.9){$\edge_{1, k,l+1}$}

\rput[t]{90}(-0.9,0.0){$\edge_{2,kl}$}
\rput[t]{90}(1.1,0.0){$\edge_{2, k+1,l}$}

}
\end{pspicture}
\caption{Quadrangulation $\cQ_N(\RR^2)$}
\label{figure:tiling}
\end{figure}

\subsubsection{Quadrangulations of the torus}
The lattice $\Lambda_N$ acts on itself, by translation. It follows that $\Lambda_N$ also acts naturally on the vertices, on the edges and on the faces of the quadrangulation $\cQ_N(\RR^2)$ by translation.
Since $\Gamma_N\subset\Lambda_N^{ch}\subset \Lambda_N$, the lattice
$\Gamma_N$ acts on $\cQ_N(\RR^2)$ as well. Thus, the quadrangulation descends
to a quadrangulation  $\cQ_N(\Sigma)$  of the quotient $\Sigma$, via the covering map $p_N:\RR^ 2\to\Sigma$.
When this is clear from the context, the  vertices, edges and faces of $\cQ_N(\Sigma)$ will  still be denoted $\vert_{kl}$, $\edge_{1,kl}$,
$\edge_{2,kl}$ and $\face_{kl}$.

\subsubsection{Alternate quadrangulation of the plane}
\label{sec:altquad}
Our  construction involves the various diffeomorphisms
 $\Phi:\RR^ 2/\Gamma\to\Sigma$ and $\Phi_N:\RR^ 2/\Gamma_N\to\Sigma$.
For the purpose of analysis and, more specifically, the notion of convergence, it is convenient to identify $\Sigma$ with a single reference quotient, say $\RR^ 2/\Gamma$ using $\Phi$. 

Lifting $\cQ(\Sigma)$ via the covering map $p:\RR^2\to \Sigma$ provides a quadrangulation different from $\cQ_N(\RR^2)$. We denote by $\hat\cQ_N(\RR^2)$ the quadrangulation obtained as the image of $\cQ_N(\RR^ 2)$ by the isomorphism $U_N:\RR^ 2\to \RR^ 2$. We also denote by $\hat\Lambda_N$ and $\hat\Lambda^{ch}_N$ 
the images of $ \Lambda_N$ and $\Lambda^{ch}_N$  by $U_N$. 
By definition,  $\Gamma$ is a sublattice of $\hat\Lambda^{ch}_N$ and we have a sequence of canonical inclusions
$$
\Gamma\subset \hat\Lambda^{ch}_N  \subset \hat\Lambda_N
$$
which is nothing else but the image of the inclusions
$$
\Gamma_N\subset \Lambda^{ch}_N  \subset \Lambda_N
$$
by $U_N$.
By construction, the quadrangulation $\hat\cQ_N(\RR^ 2)$ has vertices given by the elements of the lattice $\hat\Lambda_N$.
Furthermore $\hat\cQ_N(\RR^2)$  descends to the quotient via the covering map $p:\RR^2\to\Sigma$ into a quadrangulation that coincides with~$\cQ_N(\Sigma)$.

\subsection{Checkers graph}
\label{sec:check}
We  associate a graph $\cG_N(\RR^2)$ to the quadrangulation $\cQ_N(\RR^2)$,
called the \emph{checkers graph} of $\cQ_N(\RR^2)$. Combinatorially, the vertices $\zert_{kl}$ of $\cG_N(\RR^2)$ correspond to
faces $\face_{kl}$ of $\cQ_N(\RR^2)$. However a vertex $\zert_{kl}$ of the graph $\cG_N(\RR^2)$ shall
be though of as  the  
barycenter  of the face $\face_{kl}$ of $\cQ_N(\RR^2)$, understood as a square of
$\RR^2$. The fact that vertices of the graph correspond to points in
$\RR^2$ will be most helpful for defining the notion of convergence at \S\ref{sec:conv}.
Two barycenters are connected by an edge if, and only if, they
belong to faces having exactly one vertex in common. For instance the
faces $\face_{kl}$ and $\face_{k+1,l+1}$ of $\cQ_N(\RR^ 2)$ have exactly one vertex in
common. An edge between two vertices of $\cG_N(\RR^ 2)$ is the segment of straight
line of $\RR^2$ between the two vertices.

Figure~\ref{fig:CG} shows
the quadrangulation $\cQ_N(\RR^2)$ using dashed lines  and
the corresponding checkers graph $\cG_N(\RR^2)$. The graph has two connected components painted
with  colors red and blue. The bullets correspond to vertices of the
graph.
\begin{figure}[H]
  \begin{pspicture}[showgrid=false](-2,-2)(2,2)
    \psgrid[griddots=5,gridlabels=0, subgriddiv=1]
    \psset{linecolor=blue }
    \psline (-2,1) (-1,2)
    \psline (-2,-1) (1,2)
    \psline (-1,-2) (2,1)
    \psline (1,-2) (2,-1)
    \psline (2,-2) (-2,2)
    \psline (0,2) (2,0)
    \psline (-2,0) (0,-2)

    \psset{linecolor=blue, fillstyle=solid , fillcolor=blue, linestyle=solid}
    \pscircle (-1.5,1.5){.12}
    \pscircle (0.5,1.5){.12}
    \pscircle (-0.5,0.5){.12}
    \pscircle (1.5,0.5){.12}
    \pscircle (-1.5,-0.5){.12}
     \pscircle(0.5,-0.5){.12}
     \pscircle (0.5,-0.5){.12}
     \pscircle(-0.5,-1.5){.12}
     \pscircle(1.5,-1.5){.12}

         \psset{linecolor=red}
         \psline(-2,0)(0,2)
         \psline(-2,-2)(2,2)
         \psline(0,-2)(2,0)
         \psline(-2,-1)(-1,-2)
         \psline(-2,1)(1,-2)
         \psline(-1,2)(2,-1)
         \psline(1,2)(2,1)
         
    \psset{linecolor=red, fillstyle=solid , fillcolor=red,
      linestyle=solid}
    \pscircle (-0.5,1.5){.12}
    \pscircle (1.5,1.5){.12}
    \pscircle (-1.5,0.5){.12}
    \pscircle (0.5,0.5){.12}
    \pscircle (-0.5,-0.5){.12}
     \pscircle (1.5,-0.5){.12}
     \pscircle (-0.5,-0.5){.12}
     \pscircle (-1.5,-1.5){.12}
     \pscircle (0.5,-1.5){.12}
  \end{pspicture}
  \caption{Graph $\cG_N(\RR^2)$}
    \label{fig:CG}
\end{figure}

 \subsubsection{Splitting of the ckeckers graph} 
The graph $\cG_N(\RR^2)$ splits  into two connected components denoted
$$
\cG_N(\RR^2) = \cG^+_N(\RR^2) \cup \cG_N^-(\RR^2),
$$
where $\cG_N^+(\RR^2)$ contains the vertex $\zert_{00}$ corresponding to the face
$\face_{00}$, by convention.

 The lattice  $\Lambda_N$ acts by translation on $\cQ_N(\RR^2)$ and on
$\cG_N(\RR^2)$. The action on the vertices of
$\cG_N(\RR^2)$
(or, equivalently the faces of $\cQ_N(\RR^2)$) is transitive.
However the sublattice $\Lambda^{ch}_N$ does not act transitively: in fact it
preserves the connected components of the
graph $\cG_N(\RR^2)$  and acts  transitively on  each component.
The quotient  $\Lambda_N/\Lambda^{ch}_N
\simeq \ZZ_2$ is the residual action of the lattice $\Lambda_N$ on the connected
components of $\cG_N(\RR^2)$.

\subsubsection{Checkers graph of the quotient}
By construction $\Gamma_N\subset \Lambda^{ch}_N$, so that the action of $\Lambda_N$
preserves the connected components of $\cG_N(\RR^2)$. It follows that the
graphs $\cG_N(\RR^2)$, $\cG_N^+(\RR^2)$ and $\cG_N^-(\RR^2)$ descend
as graphs $\cG_N(\Sigma)$, $\cG_N^+(\Sigma)$ and $\cG_N^-(\Sigma)$ on the
quotient $\Sigma\simeq \RR^2/\Gamma_N$ via the covering map $p_N:\RR^ 2\to \Sigma$. Furthermore, the graph
$\cG_N(\Sigma)$ splits into two connected components 
$\cG_N^+(\Sigma)$ and $\cG_N^-(\Sigma)$:
$$
\cG_N(\Sigma) = \cG_N^+(\Sigma)\cup \cG_N^-(\Sigma).
$$

\subsubsection{Alternate checkers graph on the plane}
A discussion similar to the case of the quadrangulations $\cQ_N(\RR^ 2)$ and $\hat\cQ_N(\RR^ 2)$ occurs here (cf. \S\ref{sec:altquad}). We introduce the checkers graphs
$\hat\cG_N(\RR^2)$, $\hat\cG^+_N(\RR^2)$ and $\hat\cG_N^-(\RR^2)$ obtained as the image of $\cG_N(\RR^2)$, $\cG^+_N(\RR^2)$ and $\cG_N^-(\RR^2)$ by $U_N$. Similarly to the non-hat version, these graphs can be also understood  as the checkers graphs of $\hat\cQ_N(\RR^ 2)$. They descend via the covering map $p:\RR^ 2\to\Sigma$ where we recover  $\cG_N(\Sigma)$, $\cG_N^+(\Sigma)$ and $\cG_N^-(\Sigma)$. 
If the vertices of the checkers graph $\cG_N(\RR^2)$ are the barycenters  $\zert_{kl}$ of the faces $f_{kl}$, their images by $U_N$, denoted $\hat \zert_{kl}$, are the vertices of $\hat\cG_N(\RR^2)$.

\subsection{Examples}
\label{sec:examples}
The lattices of $\RR^2$ defined at \S\ref{sec:sqlat} come with canonical inclusions
$$
\Lambda_1\subset \cdots\Lambda_N\subset\Lambda_{N+1}\subset\cdots
$$
and
$$
\Lambda_1^{ch}\subset \cdots\Lambda_N^{ch}\subset\Lambda_{N+1}^{ch}\subset\cdots
$$
If $\Gamma$ is a sublattice of $\Lambda_1^{ch}$, then its appromixations $\Gamma_N$ coincide with $\Gamma$, which makes the construction of $\cQ_N(\Sigma)$ somewhat simpler.
For example, we may consider the lattice
$$
\Gamma' = \ZZ (e_1+e_2)\oplus \ZZ(e_2-e_1) \subset \Lambda_1^{ch},
$$
or the lattice 
$$
\Gamma'' = \ZZ e_1\oplus \ZZ e_2 = \Lambda_1.
$$
In the latter case, $\Gamma''\subset \Lambda_N^{ch}$ if and only if $N$ is even and we shall only consider $\cQ_N(\Sigma'')$ when $N$ is even.
The quotients $\Sigma'=\RR^2/\Gamma'$ and $\Sigma''=\RR^2/\Gamma''$ are conformally isomorphic but the quadrangulations $\cQ_N(\Sigma')$ and $\cQ_N(\Sigma'')$ are not isomorphic through a conformal mapping. 

Let $\ell_1:\SS^1\to\CC $ and $\ell_2:\SS^1\to\CC$ be two smooth embeddings of the circle into the complex plane $\CC$. This provides an embedding of the torus
\begin{align*}
\ell:\SS^1\times\SS^1 &\longrightarrow \CC^2 \simeq \RR^4\\
(\varphi_1,\varphi_2)&\longmapsto (\ell_1(\varphi_1),\ell_2(\varphi_2))
\end{align*}
which is isotropic since both maps $\ell_i$ are. The image of $\ell$
is usually called a \emph{product Lagrangian torus of $\RR^4$}.

The map $\ell$ can be approximated by a piecewise linear maps. The idea
is to approximate the two embedded circles by polygons of $\CC$. We
obtain a product of two polygons approximating the product torus.
More precisely, we define
$$
\ell_i^N:(N^{-1}\ZZ)/\ZZ \to \CC
$$
by $\ell_i^N(v)=\ell_i(v)$. The map $\ell_i^N$ can be extended as a
piecewise linear map denoted
$$
\ell_i^N:\RR/\ZZ \to \CC
$$
as well. If $N$ is sufficiently large, the maps $\ell^N_i$ are
piecewise linear embeddings.
For the same reasons as before, the product embedding
$$
\ell_N:\SS^1\times\SS^1\to \RR^4
$$
defined by $\ell_N(\varphi_1,\varphi_2)=(\ell_1^N(\varphi_1),\ell_2^N(\varphi_2))$
is isotropic and it is a piecewise linear isotropic approximation of
$\ell$.
Notice that the maps $\ell_N$ can be recovered only from the
$\RR^4$-coordinates of the vertices of the points in
$\Lambda_N/\ZZ^2$. These vertices are by definition the vertices of
the quadrangulation  $\cQ_N(\Sigma'')$,
 modulo the isomorphism
$$
\SS^1\times\SS^1\simeq \RR^2/\Gamma'' =\Sigma'',
$$
where $\Gamma''=\Lambda_1$ is the standard lattice described above.
Notice that each face of the quadrangulation is mapped to a
quadrialteral of $\RR^4$ contained in a Lagrangian plane.

\begin{rmk}
  The piecewise linear isotropic
  embeddings $\ell_N$ of
  the torus described above were essentially the only examples
  known at the begining of this research project. 
   If $\ell$ is any smooth isotropic map, one can
   construct samples $\ell_N$ as above (cf. \S\ref{sec:quadsamp}).
   Strictly speaking, these samples
  are quadrangular meshes. In general these samples are not exactly isotropic. 
  From this point of view, the product examples described above are very special, because in
  this case the samples are isotropic.  In general,  one needs a suitable perturbation theory so that they
  become isotropic, which is the  technical task  of this paper. 
\end{rmk}

\subsection{A splitting for discrete functions}
In this paper, the space of cycles  $C_j(\cW)$ of a CW-complex $\cW$, is the
real vector space freely generated by $j$-cells of $\cW$. The space of
cocycles
$C^j(\cW)$ is the dual of $C_j(\cW)$. The family of $j$-cells form a
canonical basis of $C_j(\cW)$. But there is no simple way to associate a
canonical dual basis to $C^j(\cW)$, unless $C_j(\cW)$ is finite
dimensional. However $C^j(\cW)$ may be thought of as the space of
constant functions on each $j$-cell called \emph{discrete functions}.

Cocycles $\phi\in C^2(\cQ_N(\RR^2))$ are thought of as discrete
functions, taking constant values
$$
\phi_{kl}=\phi(\face_{kl})= \ip{\phi,\face_{kl}},
$$
on faces.
Here $\ip{\cdot,\cdot}$ denotes the duality bracket, but we shall use the notation  $\phi(\face_{kl})$ as well.
 By construction there is a canonical identification between
the faces $\face_{kl}$ of $\cQ_N(\RR^ 2)$ and the vertices
$\zert_{kl}$ of $\cG_N(\RR^2)$. In other words, there is an identification 
$$
\fC_2(\cQ_N(\RR^2))\simeq \fC_0(\cG_N(\RR^2)) = \fC_0(\cG^+_N(\RR^2)) \cup \fC_0(\cG^-_N(\RR^2)),
$$
where $\fC_i$ denotes the set of vertices, edges, faces for $i=0,1,2$
of the relevant $CW$-complexes. 
Therefore, a discrete function $\phi$ can be understood, either as a function on
faces $\face_{kl}$ of $\cQ_N(\RR^2)$, or as a function on vertices $\zert_{kl}$ of $\cG_N(\RR^2)$. 
The above identification leads to an  isomorphism of cocycles
\begin{equation}
\label{eq:splitA}
  C^2(\cQ_N(\RR^2)) \simeq C^0(\cG_N(\RR^2)) = C^0(\cG^+_N(\RR^2)) \oplus C^0(\cG^-_N(\RR^2)).  
\end{equation}
The same decomposition holds for the \emph{hat} version of theses
objects and 
we have a canonical isomorphism 
\begin{equation}
  \label{eq:splitB}
C^2(\hat\cQ_N(\RR^2)) \simeq C^0(\hat\cG_N(\RR^2)) =
C^0(\hat\cG^+_N(\RR^2)) \oplus C^0(\hat\cG^-_N(\RR^2)).
\end{equation}
The isomorphism \eqref{eq:splitA} descends to the quotient $\Sigma$ via $p_N$ and may be expressed as an
isomorphism
\begin{equation}
  \label{eq:splitC}
  C^2(\cQ_N(\Sigma)) \simeq C^0(\cG_N(\Sigma)) = C^0(\cG^+_N(\Sigma)) \oplus C^0(\cG^-_N(\Sigma)).  
\end{equation}
Any discrete function $\phi$, in one of the
three kind of spaces $C^0(\cG_N(\cdot))$
as above,  admits a unique decomposition according to the
splittings~\eqref{eq:splitA}, \eqref{eq:splitB} or \eqref{eq:splitC}
$$
\phi = \phi^+ +\phi ^-
$$
where $\phi^\pm \in C^0(\cG_N^\pm(\cdot))$.

The induced splitting of
$C^2(\cQ_N(\cdot))$ via the isomorphisms~\eqref{eq:splitA},  \eqref{eq:splitB} or  \eqref{eq:splitC} is also denoted 
\begin{equation}
  \label{eq:splitD}
  C^2(\cQ_N(\cdot)) = C^2_+(\cQ_N(\cdot))\oplus C^2_-(\cQ_N(\cdot)).
\end{equation}
When the discrete function $\phi$ is regarded as a constant function on faces of the quadrangulation, we also write $\phi = \phi^+ +\phi^-$ according to the above splitting.
\begin{convention}
  In the sequel we shall use a shorthand
  in order to make statements that hold  either for cocycles
  of the graph $\cG_N^+(\Sigma)$, or for cocycles
  of the graph $\cG_N^-(\Sigma)$. For this purpose, we will use
  the notation $\cG_N^\pm(\Sigma)$ and the convention below:

  For every statement using the symbols  $\pm$ and $\mp$, the reader should either
  \begin{itemize}
  \item replace all symbols $\pm$  (resp $\mp$) consistently with  $+$ (resp. $-$), or
\item replace all symbols $\pm$  (resp $\mp$) consistently with $-$ (resp. $+$). 
  \end{itemize}
\end{convention}

\subsection{Discrete Hölder norms}
\label{sec:dhn}
In this section we define particular norms on the space
$C^2(\cQ_N(\RR^2))$ (or equivalently, on the space $C^0(\cG_N(\RR^2))$,  which is
a discrete analogue of the Hölder norm. The norms are defined first
on each component of the splitting~\eqref{eq:splitD} (or~\eqref{eq:splitA}).

\subsubsection{$\cC^ 0$-norm}
\label{sec:findif}
Given $\phi \in C^0(\cG_N^+(\RR^2))$ we
define its $\cC^0$-norm by
\begin{equation}
  \label{eq:C0}
\|\phi \|_{\cC^0} = \sup_{\zert \in \fC_0(\cG_N^+(\RR^2))}
|\ip{\phi,\zert}|.
\end{equation}
We define a similar norm on $C^0(\cG_N^-(\RR^2)$ (resp. $C^0(\cG_N(\RR^2)$) by taking the $\sup$
on vertices of $\cG_N^-(\RR^2)$ (resp. $\cG_N(\RR^2)$). We deduce  a norm, with the same notation~$\|\cdot\|_{\cC^0}$ on
$C^2_\pm(\cQ_N(\RR^2))$ via the isomorphisms~\eqref{eq:splitA} and
\eqref{eq:splitD}.
These quantities may be infinite. Later we shall restrict to periodic
functions, which are bounded and have a well defined $\cC^0$-norm.

\subsubsection{Finite differences}
The canonical basis $(e_1,e_2)$ with canonical coordinates $(x,y)$ of $\RR^2$ is
 not the best for
our situation.
Most of the times, we shall rotate the plane $\RR^2$ by an angle
$\pi/4$. For this purpose we introduce the rotated orthonormal basis
$(e'_1,e'_2)$ of $\RR^ 2$ 
given by 
\begin{equation}
e'_1 = \frac {e_1+e_2}{\sqrt 2},\quad e'_2 = \frac {e_2-e_1}{\sqrt 2}.
\end{equation}
The coordinates $(u,v)$ with respect to the basis $(e'_1,e'_2)$ are
deduced from the canonical coordinates $(x,y)\in\RR^2$ by the formula
\begin{equation}
  \label{eq:uv}
u = \frac {x+y}{\sqrt 2},\quad v = \frac {y-x}{\sqrt 2}.
\end{equation}

We  define finite differences of $\phi \in C^0(\cG_N^\pm(\RR^2))$,
which are discrete analogues of the partial derivatives of a function
on $\RR^2$, with respect to $u$ or $v$. These differences are denoted 
$$
\frac{\del\phi}{\del \cev u},\quad \frac{\del\phi}{\del \vec u},\quad
\frac{\del\phi}{\del \cev v}\quad \mbox{ and }\quad \frac{\del\phi}{\del \vec v} \in C^0(\cG_N^\pm(\RR^2)),
$$
where the forward or retrograde arrows indicate forward or retrograde differences,
defined as follows: for $\phi \in C^0(\cG^+(\RR^2))$, we write
$\phi_{kl} = \ip{\phi,\zert_{kl}}$ for $\zert_{kl}\in
\fC_0(\cG_N^+(\RR^2))$ and put
\begin{align}
\ip{\frac{\del\phi}{\del \cev u},\zert_{kl}}  &=  \frac{N}{\sqrt 2}(\phi_{kl}-
\phi_{k-1,l-1}) , \\
\ip{\frac{\del\phi}{\del \vec u},\zert_{kl}} & =  \frac{N}{\sqrt 2}(\phi_{k+1,l+1}-
\phi_{kl}), \\
\ip{\frac{\del\phi}{\del \cev v},\zert_{kl}} & =  \frac{N}{\sqrt 2}(\phi_{kl}-
\phi_{k+1,l-1})  \quad \mbox { and }\\
\ip{\frac{\del\phi}{\del \vec v},\zert_{kl}} & =  \frac{N}{\sqrt 2}(\phi_{k-1,l+1}-
\phi_{kl}).
\end{align}
The finite differences are defined with the same formulae if  $\phi
\in C^0(\cG^-(\RR^2))$. Since all the indices involved in the above formulae correpond to
vertices in connected component of $\zert_{kl}$ in $\cG_N(\RR^2)$, the
finite differences $\frac{\del}{\del \cev u}$, $\frac{\del}{\del \vec u}$,
$\frac{\del}{\del \cev v}$  and $\frac{\del}{\del \vec v}$ define endomorphisms 
$$
C^0(\cG_N^+(\RR^2))\oplus  C^0(\cG_N^-(\RR^2))\longrightarrow
C^0(\cG_N^+(\RR^2))\oplus  C^0(\cG_N^-(\RR^2)),
$$
which respect the above splitting.

Finite differences can also be expressed using the translations of
$\Lambda_N^{ch}$ acting on functions. If $T_u$, $T_v$ are the
translations acting on  $\cG_N(\RR^2)$, given respectively by the
vectors
$\frac {e_1+e_2}N$ and $\frac {e_2-e_1}N$, then
\begin{equation}
  \label{eq:fd1}
\frac{\del\phi}{\del\vec u} = \frac N{\sqrt 2} ( \phi\circ T_u -\phi ), \quad  \frac{\del\phi}{\del\cev u} = \frac N{\sqrt 2} ( \phi -\phi \circ T_u^ {-1} )  
\end{equation}
and
\begin{equation}
  \label{eq:fd2}
\frac{\del\phi}{\del\vec v} = \frac N{\sqrt 2} ( \phi\circ T_v -\phi ), \quad  \frac{\del\phi}{\del\cev v} = \frac N{\sqrt 2} ( \phi -\phi \circ T_v^ {-1} ).
\end{equation}
As an immediate consequence of \eqref{eq:fd1}, we have
\begin{equation}
  \label{eq:fd3}
  \frac{\del\phi}{\del\vec u} = \frac{\del\phi}{\del \cev u}\circ T_u,
\end{equation}
so that  the functions $\frac{\del\phi}{\del\vec u}$ and $\frac{\del\phi}{\del \cev u}$
have the same $\cC^0$-norm.
The same holds for the $v$-coordinate since by~\eqref{eq:fd2} 
\begin{equation}
  \label{eq:fd4}
  \frac{\del\phi}{\del\vec v} = \frac{\del\phi}{\del \cev v}\circ T_v,
\end{equation}
so that  finite differences $\frac{\del\phi}{\del \cev v}$ and $\frac{\del\phi}{\del \vec v}$
 have the same $\cC^ 0$-norm.
\begin{notation}
 As far as we are concerned with the $\cC^0$-norms of finite
 differences, we could drop the arrow notation over $u$ or $v$, since
 the forward of retrograde differences have  the  same norms. 
\end{notation}

\subsubsection{Definition of H\"older norms}
For $\phi\in C^0(\cG_N^+(\RR^2))$, we may define its $\cC^1$-norm as
$$
\|\phi\|_{\cC^1} = \|\phi\|_{\cC^0} + \left \|\frac{\del \phi}{\del
  u}\right \|_{\cC^0} + \left \|\frac{\del \phi}{\del v}\right
\|_{\cC^0}
$$
and its  $\cC^2$-norm by
$$
\|\phi\|_{\cC^2} = \|\phi\|_{\cC^1} + \left \|\frac{\del^2 \phi}{\del
  u^2}\right \|_{\cC^0} + \left \|\frac{\del^2 \phi}{\del
  v^2}\right  \|_{\cC^0} + \left \|\frac{\del^2 \phi}{\del
  u\del v}\right \|_{\cC^0}.
$$
More generally we can define a $\cC^k$-norm on $C^0(\cG^+(\RR^2))$ by
induction. Similarly we define a $\cC^k$-norm on $C^0(\cG^-(\RR^2))$.

For a positive H\"older constant $\alpha\in (0,1)$, we define the
$\cC^{0,\alpha}$-H\"older norm of~$\phi\in C^0(\cG^+(\RR^2))$ by 
\begin{equation}
  \label{eq:C0alpha}
\|\phi\|_{\cC^{0,\alpha}}=  \|\phi\|_{\cC^{0}} +
\sup_{
  \substack{{\zert_{kl}, \zert_{mn} \in
    \fC_0(\cG_N^+(\RR^2))} \\
    {\zert_{kl}\neq\zert_{mn} }
  }q
}\frac {|\phi_{kl}-\phi_{mn}|}{\|\zert_{kl}-\zert_{mn}\|^\alpha},
\end{equation}
where $\|\zert_{kl}-\zert_{mn}\|$ is the Euclidean distance between $\zert_{kl}$
and $\zert_{mn}$ in $\RR^2$.
The $\cC^{1,\alpha}$-H\"older norm of $\phi\in C^0(\cG^+(\RR^2))$ is defined by
$$
\|\phi\|_{\cC^{1,\alpha}} = \|\phi\|_{\cC^0} + \left \|\frac{\del \phi}{\del
  u}\right \|_{\cC^{0,\alpha}} + \left \|\frac{\del \phi}{\del
  v}\right \|_{\cC^{0,\alpha}} ,
$$
and its $\cC^{2,\alpha}$-H\"older norm is defined by
$$
\|\phi\|_{\cC^{2,\alpha}} = \|\phi\|_{\cC^1} + \left \|\frac{\del^2 \phi}{\del
  u^2}\right \|_{\cC^{0,\alpha}} + \left \|\frac{\del^2 \phi}{\del
  v^2}\right \|_{\cC^{0,\alpha}} + \left \|\frac{\del^2 \phi}{\del
  u \del v}\right \|_{\cC^{0,\alpha}}. 
$$
More generally, we can define a $\cC^{k,\alpha}$-H\"older norm by induction on
$C^0(\cG^+(\RR^2))$, in a obvious way. We define
a $\cC^k$ and a $\cC^{k,\alpha}$-H\"older norm on  $C^0(\cG^-(\RR^2))$ by taking the
$\sup$ of the above formulae on vertices
of $\cG_N^-(\RR^2)$ instead.

\subsubsection{Weak H\"older norms}
\label{sec:won}
 For
 $\phi \in C^2(\cQ_N(\RR^2))\simeq C^0(\cG_N(\RR^2))$ we use the direct sum decomposition
 $\phi=\phi^++\phi^-$ of \eqref{eq:splitA} or \eqref{eq:splitD}. We define the \emph{weak} $\cC^{k,\alpha}_w$-norm of $\phi$ by
$$
\|\phi\|_{\cC^{k,\alpha}_w} = \left \|\phi^+\right \|_{\cC^{k,\alpha}}
+\left \|\phi^-\right \|_{\cC^{k,\alpha}},
$$
where the Hölder norms of each components $\phi^\pm$ are defined in
the previous section.
Similarly, the \emph{weak} $\cC^k_w$-norm of $\phi $ is defined by
$$
\|\phi\|_{\cC^{k}_w} = \left \|\phi^+\right \|_{\cC^{k}}
+\left \|\phi^-\right \|_{\cC^{k}}.
$$
\begin{rmk}
As you may have noticed, the discrete $\cC^{k,\alpha}_w$-H\"older norms or $\cC^k_w$-norms defined above on
  $C^0(\cG_N(\RR^2))$ are called \emph{weak}. Indeed, only the variations of $\phi$ in the \emph{diagonal directions} spanned by the vectors $\frac{e_1+e_2}2$ and $\frac{e_2-e_1}2$
are taken into account. It turns out that these weak norms are the one appropriate to
set up the fixed point principle, as explained in~\S\ref{sec:fpt}.

In the sequel, we shall drop the term \emph{weak} for the sake of brevity. However, the reader should bear in mind that these norms may allow some unexpected behavior when $N$ goes to infinity (cf. Example~\ref{example:comb}).
\end{rmk}

\subsubsection{Quotient and alternate quadrangulations}
The alternate versions of the quadrangulation $\hat\cQ_N(\RR^2)$ 
and checkers graph $\hat\cG_N(\RR^2)$ are canonically isomorphic to the non-hat
versions  $\cQ_N(\RR^2)$ 
and $\cG_N(\RR^2)$. Thus, we have an isomorphism
$$C^2(\cQ_N(\RR^ 2))\simeq C^2(\hat\cQ_N(\RR^ 2)).$$
This isomorphism allows to define   $\cC^k_w$ and $\cC^ {k,\alpha}_w$-norms   on $C^2(\hat\cQ_N(\RR^ 2))$.
A function $\phi\in C^2(\cQ_N(\Sigma))$ admits a lift
$\phi_N = \phi \circ p_N\in C^ 2(\cQ_N(\RR^2))$. We define the norms of $\phi$ as the norms of its lift:
$$
\|\phi\|_{\cC^{k,\alpha}_w} = \|\phi_N\|_{\cC^{k,\alpha}_w},\quad \|\phi\|_{\cC^{k}_w} = \|\phi_N\|_{\cC^{k}_w}.
$$

\begin{rmk}
  The discrete functions on $\Sigma$ have finite H\"older norm since they are bounded, and so are their finite differences. 
\end{rmk}

\subsection{Convergence of discrete functions}
\label{sec:conv}
In this section,  a suitable notion of convergence for a sequence of discrete functions is introduced. This concept will be the cornerstone of a version of the Ascoli-Arzela compactness theorem. It will be an essential tool to obtain spectral gap results at~\S\ref{sec:specgap}.

\subsubsection{Definition of converging sequences}
\begin{dfn}
\label{def:conv}
  Let $(N_k)_{k\in \NN}$ be an  increasing sequence of positive integers. Let $\psi_{N_k}\in C^0(\hat\cG^\pm_{N_k}(\RR^2))$ be a sequence of discrete functions and
  $\phi:\RR^2\to \RR$ be  a function defined on the plane.

Assume that
  for every point  $w \in \RR^2$ and $\epsilon >0$, there exists
  $ \delta >0$  and an integer $ k_0 >0$, such that
for every integer $k\geq k_0$ and
vertex $\zert\in\fC_0(\hat\cG_{N_k}^\pm(\RR^2))$
with the property that $\| w-\zert \|\leq \delta$, we have
$  |\phi(w)-\psi_{N_k}(\zert)|\leq \epsilon$.

Then we say that the sequence of discrete functions $(\psi_{N_k})$
converges toward the function  $\phi:\RR^2\to \RR$. This property is denoted by
$$
\psi_{N_k} \to \phi \quad \mbox{ or } \quad \lim \psi_{N_k} = \phi.
$$  

If $\psi_{N_k}\in C^0(\hat\cG_{N_k}(\RR^2))$ is a sequence of discrete
functions with associated decomposition $ \psi_{N_k}= \psi^+_{N_k}+
\psi^-_{N_k}$ and with the property that the components converge to
functions $\phi^+$ and $\phi^-$, in the sense of the above  
definition, i.e.
$$
\psi_{N_k}^+ \to \phi^+ \quad \mbox{ and }\quad \psi_{N_k}^- \to \phi^-,
$$
we say that $\psi_{N_k}$ converges toward the pair of functions $(\phi^+,\phi^-)$. This property is  denoted by
$$
\psi_{N_k} \to (\phi^+,\phi^ -) \quad \mbox{ or } \quad \lim \psi_{N_k} =  (\phi^+,\phi^ -).
$$
\end{dfn}
\begin{rmk}
  The above definition may also be stated in a somewhat slicker way:
  we say that a sequence $\psi_{N_k} \in
  C^0(\hat\cG^+_{N_k}(\RR^2))$ converges toward a function
  $\phi:\RR^2\to \RR$ if, at every point $w$ of the plane,
   $\psi_{N_k}$ takes  arbitrarily close values to $\phi(w)$, for
  every $k$ sufficiently large and for all
  vertices of $\hat\cG^+_{N_k}(\RR^2)$ in a sufficiently
  small neighborhood of $w$.
\end{rmk}
\begin{example}
  \label{example:comb}
  The splitting of discrete functions into their positive and negative
  components leads to some unusual type  converging sequences
  in the sense of  Definition~\ref{def:conv}. 
For example, we may define a sequence of \emph{discrete comb functions} as follows. We 
define $\psi_N^\pm\in C^0(\hat\cG^\pm_N(\RR^2))$ as a constant
function each connected component of the graph, equal to $\pm 1$ at
each vertex of $\cG^\pm_N(\RR^2)$. Let  
$\I:\RR^ 2\to\RR$ be the constant  function equal to $1$ at every
point of the plane. Then $\lim \psi_N^+ =\I$ whereas $\lim \psi_N^-
=-\I$. If $\psi_N:=\psi^+_N+\psi^-_N$, then $\psi_N$ converges and 
$$
\lim\psi_N=(\I,-\I).
$$
Typically, the sequence $\psi_N$ is uniformly bounded in weak
$\cC^{0,\alpha}_w$-norm. Our notion of convergence is designed to
state a version of the Ascoli-Arzela theorem in this setting.
\end{example}

The notion of convergence of discrete functions is extended to $C^2(\cQ_N(\Sigma))$ as follows:
\begin{dfn}
\label{def:convsigma}
  Let $\psi_{N_k}\in C^0(\cG^\pm_{N_k}(\Sigma))$ be a sequence of discrete functions and
  $\phi:\Sigma\to \RR$ be a function defined on $\Sigma$.
Let $\hat \psi_{N_k} =\psi_{N_k}\circ p \in C^0(\hat\cG^\pm_{N_k}(\RR^2))$ be the lift of $\psi_{N_k}$
via the canonical projection $p$ and $\hat
\phi=\phi\circ p:\RR^2\to \RR$ be the lift of $\phi$. We say that $(\psi_{N_k})$ converges to
$\phi$ if $(\hat\psi_{N_k}) \in C^0(\cG^\pm_{N_k}(\Sigma))$ converges to $\hat \phi:\RR^2\to \RR$ in the sense Definition~\ref{def:conv}.
 This property is denoted by
$$
\psi_N \to \phi \mbox{ or } \lim \psi_N = \psi.
$$  

If $\psi_{N_k}\in C^0(\cG_{N_k}(\Sigma))$ is a sequence of discrete functions with associated decomposition $ \psi_{N_k}= \psi^+_{N_k}+ \psi^-_{N_k}$ and with the property that both components converge to some functions $\phi^+:\Sigma\to\RR$ and $\phi^-:\Sigma\to\RR$ in the sense of the above 
definition, we say that $\psi_{N_k}$ converges toward the pair of functions $(\phi^+,\phi^-)$ and denote this by
$$
\psi_{N_k} \to (\phi^+,\phi^ -) \mbox{ or } \lim \psi_{N_k} =  (\phi^+,\phi^ -).
$$
\end{dfn}

\subsection{Continuity and limits of discrete functions}
Our notion of convergence for discrete function is intimately
 related  to the uniform convergence, in the case of continuous functions.
Indeed, we have the following result:
\begin{prop}
\label{lemma:cont}
  Let $\psi_{N_k}\in C^0(\hat\cG^\pm_{N_k}(\RR^2))$ be a sequence of discrete
  functions converging toward 
  $\phi:\RR^2\to \RR$. Then $\phi$ must be continuous.
\end{prop}
\begin{proof}
The proof goes by contradiction:
assume that $\psi_{N_k}\in C^0(\hat\cG^\pm_N(\RR^2))$ is a sequence
converging toward a discontinuous function $\phi$.
Then there exists $\epsilon_0>0$, $w\in \RR^2$ and a
  sequence of points $w_k\in \RR^2$ such that $\lim_{k\to \infty} w_k = w$ and $|\phi(w_k) -
  \phi(w)|\geq \epsilon_0$ for all $k$.

  From the definition of convergence
  of discrete functions, we can extract a sequence $N'_k$ from $N_k$
  and vertices $\zert_k$ of $\hat\cG^\pm_{N'_k}(\RR^2)$
    such that $|w_k - \zert_{k}|\to
  0 $ and $|\psi_{N'_k}(\zert_k) -  \phi(w_{k})| \to 0$ as
  $k\to \infty$.

  By construction $\lim \zert_k=w$. Furthermore
  $$
     |\phi(w_{k})-\phi( w)|\leq
  |\phi(w_{k})- \psi_{N'_k}(\zert_k)| +   |\psi_{N'_k}(\zert_k)- \phi( w)|.
  $$
  The LHS is bounded below by $\epsilon_0>0$. The first term of the
  RHS converges to $0$ by definition of the sequences. The second term
  of the RHS converges to zero, by definition of the convergence of a
  sequence of discrete functions. This is a contradiction, hence
  $\phi:\RR^2\to \RR$ is continuous.
\end{proof}
\begin{cor}
  \label{cor:cont}
  Let $\psi_{N_k}\in C^0(\cG^\pm_N(\Sigma))$ be a sequence of discrete
  functions converging toward 
  $\phi:\Sigma\to \RR$. Then $\phi$ is continuous.  
\end{cor}
\begin{proof}
We use the covering map $p:\RR^2\to\Sigma$ and apply Proposition~\ref{lemma:cont} to the lift of the functions.
\end{proof}

\subsection{Samples and convergence of discrete functions}
\begin{dfn}
If  $\phi:\RR^2\to \RR$ is any real function, we  define its
\emph{samples} $\phi^\pm_N\in C^0( \hat\cG^\pm_N(\RR^2))$ by
$$
\ip{\phi_N^\pm,
  \hat \zert_{kl}}:= \phi(\hat \zert_{kl})
$$
for every $\hat \zert_{kl}\in\fC_0(\hat\cG^\pm_N(\RR^2))$.
We define similarly the samples  $\phi_N^\pm\in C^0(\cG_N^\pm(\Sigma))$ of
a real function $\phi:\Sigma\to\RR$.
Let $\hat \phi =\phi \circ p :\RR^2\to \RR $
be the  lift of $\phi$ via the projection $p$.
Its samples $\hat\phi_N^\pm\in C^0(\hat\cG_N^\pm(\RR^2))$, as defined above, descend to
discrete functions $\phi_N^\pm\in  C^0(\cG_N^\pm(\Sigma))$ on the
quotient, referred to as the samples of $\phi$.
\end{dfn}

The convergence defined in Definition \ref{def:conv} is uniform in the sense of the following lemma:
\begin{prop}
  \label{prop:c0nec}
  Let $\psi_{N_k}^\pm\in C^0(\cG^\pm_{N_k}(\Sigma))$ be a sequence of discrete
  functions converging to 
  $\phi:\Sigma\to \RR$ and
  $\phi_N^\pm \in C^0(\cG^\pm_{N_k}(\Sigma))$ be the samples of $\phi$. Then
  $$
\lim_{k\to \infty}\left \|\phi^\pm_{N_k} - \psi^\pm_{N_k} \right \|_{\cC^0} = 0.
$$
\end{prop}
\begin{proof}
   Since $\phi:\Sigma\to \RR$
is a limit of a sequence of discrete functions, it  is continuous by
Corollary~\ref{cor:cont}. The surface $\Sigma$ is compact, hence $\phi$ is uniformly 
  continuous by Heine theorem.
  We denote by $\hat\psi^\pm_{N_k}$ and $\hat \phi$ the canonical lifts of
  $\psi^\pm_{N_k}\in C^0(\hat\cG^\pm_N(\RR^2))$ and $\phi$ via the projection $p:\RR^2\to \Sigma$.
    Since $\phi$ is uniformly continuous, so is $\hat\phi$.  Let $\epsilon$, be a positive real number. By uniform continuity, there exists $\delta>0$
such that for every $w, w' \in\RR^2$
\begin{equation}
  \label{eq:unifcontcond}
  \|w- w'\|\leq \delta \Rightarrow |\hat\phi(w) -
\hat\phi(w')|\leq \epsilon.  
\end{equation}

By definition of the convergence of discrete functions, for each $w\in\RR^2$, we may choose an integer $k(w)\geq 0$ and
and a real number $\eta(w)>0$ such that for all $k\geq k(w)$ and $\hat
\zert\in
\fC_0(\hat \cG^\pm_{N_k}(\RR^2))$ we have

\begin{equation}
  \label{eq:convcond}
  \|\hat \zert- w\| \leq \eta(w) \Rightarrow |\hat
\psi^\pm_{N_k}(\hat \zert)-
\hat\phi(w)|\leq \epsilon.
\end{equation}

For each $w\in\RR^2$, put
$$
\delta(w) = \min(\delta,\eta(w)).
$$
The family of open Euclidean balls $B(w,\delta(w))$,  centered at
$w\in\RR^2$ with radius $\delta(w)$, provides an open cover of $\RR^2$.
Their images $U_w=p(B(w,\delta(w)))$, by the canonical projection
$p:\RR^2\to\Sigma$, 
provide an open cover of the compact surface $\Sigma$. Hence we can
extract a finite cover  $U_i=U_{w_i}$ of $\Sigma$, for a finite collection of
points  $\{w_i\in \RR^2, 1\leq i\leq d\}$. We put $k_0 = \max \{
k(w_i)_{1\leq i \leq d} \}$ and consider $k\geq k_0$.

Every  $\zert\in \fC_0(\cG^\pm_{N_k}(\Sigma))$ is an
element of one of the open sets $U_i$. Hence $\zert$ admits a lift
$\hat \zert \in \fC_0(\hat\cG^\pm_{N_k}(\RR^2))$ contained in one of a the balls $B(w_i,\delta(w_i))$. In
particular
$$
| \psi^\pm_{N_k}(\zert) -  \phi^\pm_{N_k}(\zert)| =
| \hat \psi^\pm_{N_k}(\hat \zert) - \hat \phi^\pm_{N_k}(\hat \zert)|
\leq |\hat \psi^\pm_{N_k}(\hat \zert)  -
\hat \phi(w_i)| + |\hat \phi (w_i) - \hat \phi^\pm_{N_k}(\hat \zert)|.
$$
The first term of the RHS is bounded above by $\epsilon$ by
\eqref{eq:convcond}. By definition $\hat \phi^\pm_{N_k}(\hat \zert) = \hat
\phi(\hat \zert)$, hence the second term of the RHS is bounded above by
$\epsilon$ thanks to \eqref{eq:unifcontcond}.
In conclusion
$$
| \psi^\pm_{N_k}( \zert) - \phi^\pm_{N_k}(\zert)|\leq 2\epsilon,
$$
which  shows that
$$
\|\psi^\pm_N - \phi^\pm_N\|_{\cC^0}\leq 2\epsilon
$$
for $k\geq k_0$.
\end{proof}
We also have a sort of converse for Proposition~\ref{prop:c0nec}:
\begin{prop}
    \label{prop:c0suf}
  Let $\psi^\pm_{N_k}\in C^0(\cG^\pm_{N_k}(\Sigma))$ be a sequence of discrete
  functions and 
  $\phi:\Sigma\to \RR$ a continuous function such that
$$
\lim_{k\to \infty}\left \|\phi_{N_k}^\pm - \psi_{N_k}^\pm \right\|_{\cC^0} = 0,
$$
where $\phi^\pm_{N_k}\in C^0(\cG^\pm_{N_k}(\Sigma))$ are the samples of $\phi$.
Then
$$
 \lim \psi_{N_k}^\pm = \phi.
$$
\end{prop}
\begin{proof}
The compactness of $\Sigma$ implies the uniform continuity of
$\phi$, which  is a key argument in a proof closely related to the one of Proposition~\ref{prop:c0nec}. The details are left to the
  interested reader.
\end{proof}

Proposition~\ref{prop:c0suf} has the following immediate corollary, which shows that samples of a function are natural approximations:
\begin{cor}
\label{cor:sample}    Let  
  $\phi:\Sigma\to \RR$ be a continuous function, and
    $\phi^\pm_{N} \in C^0(\cG^\pm_{N}(\Sigma))$ its samples.
Then
$$
 \lim \phi_{N}^\pm = \phi.
$$
\end{cor}

\subsection{Precompactness}
We denote by $\|\cdot \|_{\cC^{0,\alpha}}$ the usual Hölder norm on
the space of function $\phi:\Sigma\to \RR$, defined with respect to the
Riemannian metric $g_\sigma$, for instance. The corresponding Hölder space is
denoted
$\cC^{0,\alpha}(\Sigma)$.
We may now state a version of the Ascoli-Arzela theorem adapted to our setting:
\begin{theo}[Ascoli-Arzela, first version]
  \label{theo:AA1}
  Let $\psi^\pm_{N_k}$ be a sequence of discrete function in
  $C^0(\cG^\pm_{N_k}(\Sigma))$,   
  which are uniformly bounded in $\cC^{0,\alpha}$-norm.
  In other words, 
  there exists a constant $c>0$ with the property that
  $$
\|\psi^\pm_{N_k}\|_{{\cC}^{0,\alpha}}\leq c
$$
for all $k\in\NN$.
Then there exists a subsequence $N_k'$ of $N_k$ and a function $\phi:\Sigma\to
\RR$ in $\cC^{0,\alpha}(\Sigma)$,  such that
$$
\lim \psi^\pm_{N_k'} = \phi.
$$
\end{theo}
\begin{proof}
Let 
$\psi_{N_k}\in C^0(\cG^+_{N_k}(\Sigma))$ be a sequence of discrete functions bounded in H\"older norm, as in the theorem.

We start by choosing a countable dense set $Q=\{q_n\in
  \Sigma, n\in\NN \}$  of
  $\Sigma$; for instance the projection by $p:\RR^2\to \Sigma$ of the points of rational
  coordinates in $\RR^2$ is a possible choice. For each $q_n$, we choose a lift $\hat q_n$ such that $p(\hat q_n)=q_n$.
  For each $n$ we choose a sequence $\hat \zert^n_{N_k}\in
  \fC_0(\hat\cG^+_{N_k}(\RR^2))$ such that
  that
  $$
  \lim_{k\to\infty} \hat \zert_{N_k}^n = \hat q_n.
  $$

We denote by $\hat\psi_{N_k}=\psi_{N_k}\circ p\in C^ 0(\hat\cG_{N_k}^+(\RR^2))$ the
canonical lift of $\psi_{N_k}$. By assumption, the uniform estimate on the H\"older norms provides a 
 uniform bound  $|\hat \psi_{N_k}(\hat \zert_{N_k}^n)|\leq c$. Hence we can choose a
  subsequence $N^0_k$ of integers such that
  $\hat \psi_{N^0_k}(\hat \zert_{N^0_k}^0)$ converges as $k\to \infty$.
  
  By extracting a subsequence $N^1_k$ of $N^0_k$, we may assume that
  $\hat \psi_{N^1_k}(\hat \zert_{N^1_{k}}^n)$ converges for $n=0$ or $1$, as $k\to
  \infty$.
  Extracting subsequences inductively provides 
  family of  subsequences $N^m_k$, indexed by $m$, such that
  $\hat \psi_{N^m_k}(\hat \zert^n_{N^m_{k}})$ converges for fixed $0\leq n\leq m$ as
  $k\to \infty$.
 Finally, using the diagonal subsequence $M_k=N^k_k$, we find a
 subsequence $\psi_{M_k}$ such that $\hat\psi_{M_k}(\hat \zert^n_{M_k})$
 converges for every $n\in \NN$, as $k\to \infty$.

 The function 
 $$
\phi:Q\to \RR
 $$
is defined on the countable dense subset $Q\subset\Sigma$ by
$$\phi(q_n) = \lim_{k\to\infty} \hat \psi_{M_k}(\hat z^n_{M_k}).
$$
Since the $\psi_{N_k}$ are uniformly bounded with respect to the discrete $\cC^{0,\alpha}$-norms, it follows that
 the function $\phi:Q\to \RR$ is bounded with respect to the usual $\cC^{0,\alpha}$-norm. In
 particular $\phi$ is uniformly continuous on $Q$, hence it admits a unique
 continuous extension $\phi:\Sigma\to \RR$ which turns out to be in
 $\cC^{0,\alpha}(\Sigma)$ as well.
 One can readily check, using the uniform Hölder-norm estimates, that
 the construction of the function $\phi$ is independent of the choice
 of sequences $\hat \zert^n_N$.
Furthermore the uniform H\"older estimates imply that
$$\lim _{k\to \infty}\|\psi_{M_k}-\phi_{M_k}\|_{\cC^0}=0,$$ where
$\phi_{M_k}\in C^0(\cG^+_{M_k}(\Sigma))$ are the samples of $\phi$.
This implies by Proposition~\ref{prop:c0suf} that
$$
\lim \psi_{M_k} = \phi.
$$

\end{proof}

\subsection{Higher order convergence}
We are interested in stronger  convergence of discrete
functions, taking into account higher 
order  finite differences. We start by stating the
following  elementary results:
\begin{lemma}
  \label{lemma:convc1}
  Let $\psi_{N_k}\in C^0(\cG^\pm_{N_k}(\Sigma))$ be a sequence of discrete
  functions. The finite differences
   $\frac{\del \psi_{N_k}}{\del \vec
    u}$ (resp. $\frac{\del \psi_{N_k}}{\del \vec
    v}$) converge if, and only if,
the finite differences
   $\frac{\del \psi_{N_k}}{\del \cev
    u}$ (resp.  $\frac{\del \psi_{N_k}}{\del \cev
    v}$)
converge. It they converge, they have the same limits:
$$
\lim \frac{\del \psi_{N_k}}{\del \vec
    u} =  \lim \frac{\del \psi_{N_k}}{\del \cev
    u}, \quad \lim \frac{\del \psi_{N_k}}{\del \vec
    v} =  \lim \frac{\del \psi_{N_k}}{\del \cev
    v}.
$$
\end{lemma}
\begin{proof}
  This follows
   from the fact that finite
  differences in the forward and backward directions are related by the
  translations $T_u$, or $T_v$, spanning the lattice $\Lambda^{ch}_N$,
  thanks to Formulae~\eqref{eq:fd3} and \eqref{eq:fd4}.
\end{proof}
\begin{rmk}
According to the above lemma, one can talk about the convergence of
the finite differences of a sequence of discrete functions without
specifying on the forward or retrograde directions.  
\end{rmk}

\begin{prop}
\label{prop:c1conv}
  Let $\psi_{N_k}\in C^0(\cG^\pm_{N_k}(\Sigma))$ be a converging sequence of discrete
  functions such that its first order  finite differences
  converge as well  towards the limits
    $$ \phi=\lim\psi_{N_k}, \phi_u=\lim \frac{\del
    \psi_{N_k}}{\del \vec u}
    \mbox{ and }  \phi_v= \lim
  \frac{\del \psi_{N_k}}{\del \vec v}.
  $$
  Then, the limit $\phi:\Sigma\to \RR$ is of class $\cC^1$ with partial derivatives
  given by
  $$
  \frac{\del\phi}{\del u}=\phi_u,\quad \frac{\del\phi}{\del v}=\phi_v.
  $$
\end{prop}
\begin{proof}
One can readily show that $\phi$ is a primitive function of $\phi_u$
(resp. $\phi_v$) in the
$u$-direction (resp. $v$-direction) using Riemann sums.
The limits $\phi_u$ and $\phi_v$ are continuous by
Lemma~\ref{lemma:cont} and it follows that $\phi$ is continuously differentiable.
\end{proof}

Lemma~\ref{lemma:convc1} and Proposition~\ref{prop:c1conv} motivate the following definition:
\begin{dfn}
  If a sequence of discrete functions $\psi_{N_j}\in
  C^0(\cG^\pm_{N_j}(\Sigma))$ converges together with its
  finite differences, up to order $k$, we say that the sequence  $(\psi_{N_j})$
  converges in the $\cC^k$-sense toward the function $\phi=\lim
  \psi_{N_j}$. We denote
  this property by
  $$
  \psi _{N_j}\stackrel{\cC^k}\longrightarrow \phi.
  $$

  If  $\psi_{N_j}\in
  C^0(\cG_{N_j}(\Sigma))$ is a sequence of discrete functions with
  decompositions $\psi_{N_j} =\psi_{N_j}^+ + \psi_{N_j}^-$ and
  $\phi^+,\phi^-:\Sigma\to \RR$ are functions such that
  $$
  \psi _{N_j}^+ \stackrel{\cC^k}\longrightarrow \phi^+,
  \mbox{ and }  \psi _{N_j}^-\stackrel{\cC^k} \longrightarrow \phi^-,
  $$
  we say that $\psi_{N_j}$ converges in the weak $\cC^k$-sense toward
  the pair of functions $(\phi^+,\phi^-)$. This property is denoted
  $$
  \psi _{N_j}\stackrel{\cC^k_w}\longrightarrow (\phi^+,\phi^-).
  $$
\end{dfn}
This  definition and Propositions~\ref{prop:c1conv} leads to the
following corollary:
\begin{prop}
  If  $\psi_{N_j}\in
  C^0(\cG^\pm_{N_j}(\Sigma))$ converges in the $\cC^k$ sense, the limit $\phi=\lim \psi_{N_j}$ is of class $\cC^k$. Furthermore the finite differences of $\psi_{N_j}$ converge, up to order $k$ toward the corresponding partial derivatives of $\phi$.
\end{prop}

We may now state an improved version of the Ascoli-Arzela theorem in the
$\cC^k$ setting:
\begin{theo}[Ascoli-Arzela, second version]
  \label{theo:ascoli}
  Let $\psi_{N_j}$ be a sequence of discrete function in $C^0(\cG^\pm_{N_j}(\Sigma))$,
  which are uniformly bounded in $\cC^{k,\alpha}$-norm for some $k\geq 0$, in the sense  that
  there exists  a constant $c>0$ with the property that
  $$
\|\psi_{N_j}\|_{\cC^{k,\alpha}}\leq c \quad \mbox{ for all $j\geq 0$.}
$$
 Then there exists a subsequence $N_j'$ of $N_j$ and a function $\phi:\Sigma\to
\RR$ with $\phi\in\cC^{k,\alpha}(\Sigma)$, such that
$$
\psi_{N'_j}\stackrel{\cC^k}{\longrightarrow} \phi.
$$
\end{theo}
\begin{proof}
  We give a sketch of proof in the case  $k=1$. By assumption, the $\psi_{N_j}$ are
uniformly bounded in $\cC^{1,\alpha}$-norms. Thus the finite
differences of order $1$ are bounded in $\cC^{0,\alpha}$-norm:
$$
\left \|\frac{\del\psi_{N_j}}{\del \vec u}\right \|_{\cC^{0,\alpha}}\leq c, \quad
\left \|\frac{\del\psi_{N_j}}{\del \vec v}\right \|_{\cC^{0,\alpha}}\leq c.
$$
and we may apply  Theorem~\ref{theo:AA1} to
the first order finite differences. After passing to suitable
subsequences, we may assume that
$$
\frac{\del\psi_{N_j}}{\del \vec u} \stackrel{\cC^0}{\longrightarrow}
     {\phi_u}, \frac{\del\psi_{N_j}}{\del \vec v} \stackrel{\cC^0}{\longrightarrow}{\phi_v}
$$
where $\phi_u, \phi_v \in \cC^{0,\alpha}(\Sigma)$.
Since $\psi_N$ is bounded in $\cC^{1,1}$-norm, we may apply
Ascoli-Arzela again and assume, up to further extraction, that
$$
\lim \psi_{N_j} = \phi
$$
for some continuous function $\phi$.
The rest of the proof follows from Proposition~\ref{prop:c1conv}.
The general case is proved by induction on $k$.
\end{proof}

\subsection{Examples of discrete convergence}
We present two examples of converging sequences of discrete functions
that will turn out to be useful.
\subsubsection{Samples of continuously differentiable functions}
Corollary~\ref{cor:sample} extends to stronger $\cC^k$-convergence as
follows:
\begin{prop}
  Let  $\phi:\Sigma\to \RR$ be a function of class $\cC^k$, and
    $\phi^\pm_{N} \in C^0(\cG^\pm_{N}(\Sigma))$ its samples.
Then
$$
 \phi_{N}^\pm \stackrel{\cC^k}\longrightarrow \phi.
$$
\end{prop}
\begin{proof}
  The Taylor formula insures that finite differences of $\phi^\pm_N$ converge
  uniformly to the corresponding partial derivative of $\phi$. It
  follows by Proposition~\ref{prop:c0suf} that, up to order $k$, the finite differences
  of $\phi_N^\pm$ converge in the sense of
  Definition~\ref{def:convsigma}, which proves the proposition.
\end{proof}

\subsubsection{Discrete tangent vector fields}
We may consider discrete functions with values in $\RR^m$, or more
precisely $\RR^{2n}$, rather than real valued functions.
It is an easy exercise to check
that all the notions of convergence of discrete functions, H\"older norms, introduced
before trivially extend to this setting.

Given a smooth  immersion
$\ell:\Sigma\to\RR^{2n}$, we shall define a sample $\tau_N\in 
 C^0(\cQ_N(\Sigma)) \otimes \RR^{2n}$ of $\ell$ at \S\ref{sec:lagquad}.
We will show that the discrete tangent vector fields associated to the
diagonals of the
sample $\tau_N$ converge in Proposition~\ref{prop:convell}.

\section{Perturbation theory for  isotropic meshes}
\label{sec:pert}
We keep on using the notations of the previous section. Recall that
$\ell:\Sigma\to\RR^{2n}$ is a smooth isotropic immersion and  $\Sigma$ a
surface diffeomorphic to a torus. The surface is endowed with the
pullback metric $g_\Sigma$ and the flat metric $g_\sigma$ related by a
conformal factor $g_\Sigma=\theta g_\sigma$. There is also a family of
flat metrics $g_{\sigma}^N$ induced by the diffeomorphism
$\Phi_N: \RR^2/\Gamma_N \to \Sigma$. We construct
the various versions of quadrangulations and the checkers
graphs as in~\S\ref{sec:anal}.

\subsection{Isotropic quadrangular meshes}
\label{sec:lagquad}
A quadrangular mesh
$$
\tau\in \cM_N = C^0(\cQ_N(\Sigma))\otimes\RR^{2n}
$$
associates  $\RR^{2n}$-coordinates to each vertex
of $\cQ_N(\Sigma)$. One can define a unique piecewise linear map
$$
\ell_\tau:\Sigma^1_N\to \RR^{2n}
$$
from the $1$-skeleton $\Sigma_N^1$ of the quadrangulation
$\cQ_N(\Sigma)$ into $\RR^{2n}$, which agrees with  $\tau$ at vertices.
 Contrarily to the case of a triangulation, there
 is generally no piecewise linear extension to the $2$-skeleton, that is $\Sigma$. Indeed,
  quadrilaterals of $\RR^{2n}$ that may not be planar.
There are several options to construct extensions of $\ell_\tau
:\Sigma^1_N\to \RR^{2N}$ to 
$\Sigma$, but this is not a fundamental issue as we shall see.
\begin{dfn}
An Euclidean quadrilateral of $\RR^{2n}$ is said to be isotropic if
the integral of the Liouville form $\lambda$ along the quadrilateral
vanishes.
Similarly, a mesh $\tau\in\cM_N$ is called isotropic
if the quadrilaterals of $\RR^n$ associated to each face of $\cQ_N(\Sigma)$ via
$\tau$ are isotropic in the above sense.
 The space $\scrL_N\subset \scrM_N$ is the set of all isotropic
 quadrangular meshes  $\tau\in\cM_N$.
\end{dfn}
\subsubsection{Equation for isotropic quadrilaterals}
An oriented quadrilateral  of $\RR^{2n}$ can be given by $4$ ordered vertices
$(A_0,$  $A_1,$ $A_2,$ $A_3)$. We  introduce the diagonals of the
quadrilateral
\begin{equation}
\label{eq:notdiag}
  D_0=\overrightarrow{A_0A_2},\quad D_1=\overrightarrow{A_1A_3}.  
\end{equation}
Then we have the following result, which shows that the equation for
an isotropic quadrilateral is quadratic:
\begin{lemma}
  \label{lemma:quadiso}
  The integral of the Liouville form $\lambda$ along an oriented
  quadrilateral $(A_0,\cdots,A_3)$ of $\RR^ {2n}$ is given by
  $$
\frac 12 \omega(D_0,D_1),
$$
where $D_i$ are the diagonals of the quadrilateral defined by~\eqref{eq:notdiag}.
\end{lemma}
\begin{proof}
  We construct a pyramid $\cP$ with base the quadrilateral $\cQ$ and with apex located at the
  origin $O\in\RR^{2n}$, for instance.
  By Stokes Theorem
  $$
\int_{\cQ}\lambda
=\int_{\cP}\omega.  $$
The integral of the RHS is the sum of the symplectic areas of the four
triangles $(OA_iA_{i+1})$, for $i$ considered as an index modulo $4$.
Hence the integral of the Liouville form is given by
$$
\frac
12 \sum_{i=0}^3\omega(\overrightarrow{OA_i},\overrightarrow{OA_{i+1}}) = \frac 12 \omega(D_0,D_1).
$$
\end{proof}

\subsubsection{Diagonals notation}
\label{sec:diag1}
For $\tau\in\cM_N$, we consider the lifts $\tilde \tau = \tau \circ
p_N \in C^0(\cQ_N(\RR^2))$.
Using the notations of \S\ref{sec:quadnot}, we define the diagonals
$$
D^u_\tau,D^v_\tau \in C^2(\cQ_N(\RR^2))\otimes \RR^{2n}
$$
by
$$
D^{u}_\tau(\face_{kl})=  \tilde \tau (\vert_{k+1,l+1}) - \tilde \tau(\vert_{kl}) 
$$
and
$$
D^{v}_\tau (\face_{kl})=  \tilde \tau(\vert_{k,l+1}) - \tilde \tau( \vert_{k+1,l}).
$$
Then, $D^u_\tau$ and $D^v_\tau$ descend to the quotient $\Sigma$ and provide
discrete vector fields denoted in the same way
$$
D^u_\tau,D^v_\tau\in C^2(\cQ_N(\Sigma))\otimes \RR^{2n} \simeq C^0(\cG_N(\Sigma))\otimes \RR^{2n}.
$$
By definition $D^u_\tau$ and $D^v_\tau$ represent certain diagonals of each face
of the quadrangular mesh $\tau$. It is also convenient to introduce
the renormalized discrete vector fields
$$
\scrU_\tau=\frac N{\sqrt 2}D^u_\tau \quad \mbox {and } \quad  \scrV_\tau =\frac N{\sqrt 2}D^v_\tau.
$$
\subsubsection{Equation for isotropic mesh}
The problem of finding isotropic meshes can be formulated
using a suitable equation.
Each $ \tau\in \cM_N$ and each
face $\face\in \fC_2(\cQ_N)$ defines  Euclidean quadrilateral in
$\RR^{2n}$, given by the $\RR^{2n}$-coordinates of ordered vertices of
$\face$. Such a
quadrilateral has a symplectic area defined by  the
integral of the Liouville form $\lambda$ along the quadrilateral.
We can pack this data
into a map 
$$
\mu_N : \cM_N  \longrightarrow C^2(\cQ_N(\Sigma))
$$
such that $\ip{\mu_N(\tau),\face}$ is the symplectic area of the
corresponding quadrilateral.
The space  of isotropic meshes $\scrL_N$
is by definition the set of solutions of the
equation $\mu_N=0$. In other words
$$
\cM_N\supset \scrL_N=\mu_N^{-1}(0).
$$
For analytical reasons, it will be convenient to introduce a
renormalized version of $\mu_N$, defined by
$$
\mu_N^r = N^2 \mu_N.
$$
\begin{rmk}
Given $\face\in\fC_2(\cQ_N(\Sigma))$, the real number $\mu^r_N(\face)$ is the
ratio between the symplectic area $\ip{\mu_N(\tau),\face}$ and the Euclidean area
of $\face$ with respect to the metric $g_\sigma^N$, which is
$$
\area(\face,g_\sigma^ N)=\frac 1{N^ 2}.
$$
In this sense $\mu_N^r$ can be regarded as a discrete version of the
moment map $\mu(\ell)=\frac{\ell^*\omega}{\sigma}$ introduced at
\S\ref{sec:dream} and $\ip{\mu_N^r(\tau),\face}$ as the symplectic density
of the face~$\face$ with respect to $\tau$.   
\end{rmk}

The space of isotropic meshes $\scrL_N$ is the zero set of
$\mu_N^r$. This subspace is defined by a system of quadratic
polynomials as shown by the following lemma.
\begin{lemma}
  \label{lemma:quadratic}
  The map $\mu_N : \cM_N \to C^2(\cQ_N(\Sigma))$ is quadratic.
More precisely, we have
\begin{equation}
  \label{eq:sympareaquad}
  \ip{\mu_N(\tau),\face} = \frac 12 \omega(D^u_\tau(\face),D^v_\tau(\face))
\end{equation}
and 
\begin{equation}
  \label{eq:notdiagrenq}
  \ip{\mu_N^r(\tau),\face} = \omega(\scrU_\tau(\face),\scrV_\tau(\face)).
\end{equation}
\end{lemma}
\begin{proof}
This is an immediated consequence of Lemma~\ref{lemma:quadiso}.
\end{proof}

\begin{dfn}
  \label{dfn:quadratic}
Since $\mu_N:\cM_N\to C^2(\cQ_N(\Sigma))$ is a quadratic map, it is
associated to a unique symmetric bilinear map
$$\Psi_N : \cM_N \times  \cM_N \to C^2(\cQ_N(\Sigma)).
$$
  Similarly, $\Psi^r_N$ is the symmetric bilinear map associated to
  the quadratic map~$\mu^r_N$.
\end{dfn}

\subsection{Shear action on meshes}
\label{sec:shear}
The space  $\cM_N$ admits an obvious action induced by the
translations of $\RR^{2n}$, which  preserves the subspace of
isotropic meshes $\scrL_N$.
However, translations belong to a larger group acting on $\scrM_N$, defined below, preserving
isotropic meshes.

The space of vertices of $\cQ_N(\RR^2)$ admits a splitting similar to
faces. Indeed, $\Lambda_N^{ch}$ acts on the vertices, with exactly two
orbits denoted
$\fC^+_0(\RR^2)$ and $\fC_0^-(\RR^2)$, with the convention that $\vert_{00}\in \fC^+_0(\RR^2)$.
This splitting descends to the quotient via $p_N:\RR^2\to\Sigma$,
where we have two sets of vertices (cf. Figure~\ref{figure:shear}  for
a picture)
$$
\fC_0(\Sigma) = \fC^+_0(\Sigma)\cup \fC_0^-(\Sigma).
$$
For any mesh $\tau\in\scrM_N$ and vector $T=(T_+,T_-)\in\RR^{2n}\times \RR^{2n}$, we define the action of $T$ on $\tau$ by
$$
\ip {T\cdot\tau , \vert}=
\left \{
\begin{array}{ll}
  \ip{\tau,\vert}+T_+ &\mbox { if } \vert\in \fC^+_0(\Sigma) \\
    \ip{\tau,\vert}+T_- &\mbox { if } \vert\in \fC^-_0(\Sigma) 
\end{array}
\right .
$$
The above action of $\RR^{2n}\times\RR^{2n}$ on $\scrM_N$ is called
the \emph{shear action}. If $T_+=T_-$ the action of $T$ is the usal
action by translations mentionned earlier. However, the shear action
$\tau\mapsto T\cdot\tau$
by a vector $T=(T_+,0)$  pulls apart positive and negative
vertices of $\tau$.
But the shear action preserves isotropic meshes:
\begin{prop}
  The space of isotropic meshes $\cL_N\subset \cM_N$ is invariant under the shear action.
\end{prop}
\begin{proof}
The diagonals of the quadrilaterals associated to some mesh $\tau$ are
invariant under the shear action. In particular, any isotropic mesh
remains isotropic under the shear action by Lemma~\ref{lemma:quadiso}.
\end{proof}

\begin{rmk}
The shear symmetry for $\scrL_N$ seems to rule out all prospects of
the solutions of the equation $\mu^r_N=0$ getting more regular as
$N\to\infty$. Intuitively, if $\tau$ is isotropic and close to a
smooth immersed surface (in some sense), the isotropic mesh
$(T_+,0)\cdot\tau$ now  looks wild
(cf. Figure~\ref{figure:shear}), even more so as the step size of the
quadrangulation goes to $0$.
This partly explains why  we cannot expect to use much stronger norms than the
weak Hölder norms introduced at \S\ref{sec:dhn}, and why
Theorem~\ref{theo:maindiscr} and Theorem~\ref{theo:mainquad} are only
stated with $\cC^0$-norms.
\end{rmk}

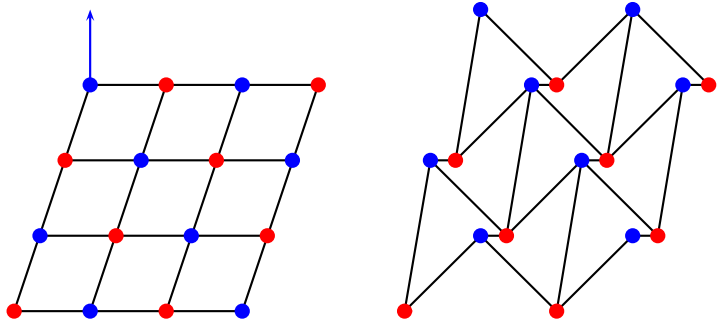
\begin{figure}[H]
\begin{pspicture}[showgrid=false](-2,-2)(3,2)
\psline(-2,-2) (-1,1)
\psline(-1,-2) (0,1)
\psline(0,-2) (1,1)
\psline(1,-2) (2,1)

\psline(-2,-2)(1,-2)
\psline(-1.66,-1)(1.33,-1)
\psline(-1.33,0)(1.66,0)
\psline(-1,1)(2,1)

    \psset{linecolor=red, fillstyle=solid , fillcolor=red, linestyle=solid}
     \pscircle (-2,-2){.1}
     \pscircle (0,-2){.1}
     \pscircle (-0.66,-1){.1}
     \pscircle (1.33,-1){.1}
     \pscircle (0.66,0){.1}
     \pscircle (-1.33,0){.1}
     \pscircle (0,1){.1}
     \pscircle (2,1){.1}

         \psset{linecolor=blue, fillstyle=solid , fillcolor=blue, linestyle=solid}
         \pscircle (-1,-2){.1}
     \pscircle (1,-2){.1}
     \pscircle (-1.66,-1){.1}
     \pscircle (0.33,-1){.1}
         \pscircle (-0.33,0){.1}
     \pscircle (1.66,0){.1}
         \pscircle (-1,1){.1}
     \pscircle (1,1){.1}
\psline{->}(-1,1)(-1,2)
     
\end{pspicture}
\begin{pspicture}[showgrid=false](-2,-2)(3,2)
\psline(-2,-2)(-1.66,0)(-1.33,0) (-1,2)
\psline(-1,-1) (-0.66,-1)(-0.33,1)(0,1)
\psline(0,-2) (0.33,0)(0.66,0) (1,2)
\psline(1,-1)(1.33,-1)(1.66,1) (2,1)

\psline(-2,-2)(-1,-1)(0,-2)(1,-1)
\psline(-1.66,0)(-.66,-1)(.33,0)(1.33,-1)
\psline(-1.33,0)(-0.33,1)(0.66,0)(1.66,1)
\psline(-1,2)(-0,1)(1,2)(2,1)

    \psset{linecolor=red, fillstyle=solid , fillcolor=red, linestyle=solid}
     \pscircle (-2,-2){.1}
     \pscircle (0,-2){.1}
     \pscircle (-0.66,-1){.1}
     \pscircle (1.33,-1){.1}
     \pscircle (0.66,0){.1}
     \pscircle (-1.33,0){.1}
     \pscircle (0,1){.1}
     \pscircle (2,1){.1}

         \psset{linecolor=blue, fillstyle=solid , fillcolor=blue, linestyle=solid}
         \pscircle (-1,-1){.1}
     \pscircle (1,-1){.1}
     \pscircle (-1.66,0){.1}
     \pscircle (0.33,0){.1}
         \pscircle (-0.33,1){.1}
     \pscircle (1.66,1){.1}
         \pscircle (-1,2){.1}
     \pscircle (1,2){.1}


\end{pspicture}
\caption{Shear action on the blue vertices of a mesh}
\label{figure:shear}
\end{figure}

\begin{rmk}
  We will make seldom mention of the shear action. But this action
  will be crucial  at
  \S\ref{sec:quadtri} to get more  generic isotropic
  quadrangular meshes. 
\end{rmk}

\subsection{Meshes obtained by sampling}
\label{sec:quadsamp}
Given a smooth immersion $\ell:\Sigma\to \RR^{2n}$, we construct a
canonical sequence of approximations of $\ell$ by quadrangular meshes
$$\tau_N\in
\scrM_N.
$$
The  map $\ell:\Sigma \to \RR^{2n}$   can be restricted to the
vertices of $\cQ_N(\Sigma)$.  Hence we may define  an element
$\tau_N\in \cM_N$, called \emph{a sample 
of $\ell$}, by
$$
\tau_N(\vert)=\ell(\vert)
$$
for each $\vert\in \fC_0(\cQ_N(\Sigma))$.
 
We would like to discuss more precisely the nature of the convergence
of $\tau_N$ towards $\ell$ in the spirit of \S\ref{sec:anal}. This is
possible at the cost of extending all the analysis introduced at
\S\ref{sec:anal} for discrete functions on faces of $\cQ_N(\Sigma)$ to
the case of functions defined at vertices. Instead of carrying this
uncomplicated but lengthy work, we will adopt a more straightforward
approach here.

For the special case $\tau=\tau_N$, where the meshes $\tau_N$ are the samples of
an immersion $\ell:\Sigma\to\RR^{2n}$, the diagonals
$D_\tau^u$, $D_\tau^v$, $\scrU_\tau$ and $\scrV_\tau$ are denoted $D_N^u,
D_N^v, \scrU_N$ and $\scrV_N$ instead.
Then we have the following result
\begin{prop}
\label{prop:convell}
  The sequence of discrete vector fields $\scrU_N^\pm$ and $\scrV_N^\pm
\in C^0(\cG_N^\pm(\Sigma))\otimes\RR^{2n}$
converge in the $\cC^k$-sense, for every $k$. Furthermore
$$
\scrU_N^\pm\stackrel {\cC^k}\longrightarrow \frac{\del\ell}{\del u}
\quad \mbox{ and } \quad
\scrV_N^\pm\stackrel {\cC^k}\longrightarrow \frac{\del\ell}{\del v}.
$$
More precisely, if we denote by $\scrU_N'$ (resp. $\scrV_N'$) the
samples of $\frac{\del\ell}{\del u}$ (resp. $\frac{\del\ell}{\del u}$)
then
$$
\|\scrU_N-\scrU_N'\|_{C^{k}_w}=\cO(N^{-1}) \mbox{ and } \|\scrV_N-\scrV_N'\|_{C^{k}_w}=\cO(N^{-1}) .
$$
\end{prop}

\subsection{Almost isotropic samples}
The defect of the samples $\tau_N$ to be isotropic is given by the
sequence of discrete functions 
\begin{equation}
  \label{eq:etaN}
\eta_N=\mu_N^r(\tau_N) \in C^2(\cQ_N(\Sigma)).  
\end{equation}
 The error
$\eta_N$ is small as $N$ goes to infinity in the sense of the
following proposition:
\begin{prop}
  Let $\ell:\Sigma\to\RR^{}$ be a smooth isotropic immersion and
 $\tau_N\in\scrM_N$ be the sequence of samples of $\ell$ with respect to the quadrangulations
  $\cQ_N(\Sigma)$.
Let $\eta_N = \mu^r(\tau_N) \in C^2(\cQ_N(\Sigma))$ be the isotropic defect
of $\tau_N$.
  Then for every integer $k\geq 0$, we have
  $$
\|\eta_N \|_{\cC_w^{k}}=\cO(N^{-1}).
$$
\end{prop}
\begin{proof}
  For each face $f\in \fC_2(\cQ_N(\Sigma))$, the quantity $\eta_N(f)$
  is given by
$$
\eta_N(f)=\frac {N^2}2 \omega(D^u_N(f),D^v_N(f))  = \omega(\scrU_N(f), \scrV_N(f)).
$$
The formula for discrete differences of a quadratic form and the
$\cC^k$-convergence of  Proposition~\ref{prop:convell} proves the proposition.
\end{proof}

\subsection{Inner products}
\label{sec:ip}
The tangent vectors to the space of meshes $\scrM_N$ and the space of
discrete functions come equiped with canonical inner products, which
are crucial for the analysis.
\subsubsection{The case of function}
The space $C^2(\cQ_N(\Sigma))$ of discrete functions comes equiped with an Euclidean inner
product which is a  discrete version of the $L^2$-inner product
for smooth functions.
The space $C_2(\cQ_N(\Sigma))$ admits a canonical basis,
given by the set of
faces $\face \in \fC_2(\cQ_N(\Sigma))$. Thus, we have a  corresponding
dual basis $\face^*$ of 
$C^2(\cQ_N(\Sigma))$ defined by
$$\ip{\face^*,\face'}=\left \{
\begin{array}{l}
  1 \mbox{ if } \face=\face'\\
0 \mbox{ otherwise}
\end{array}
\right .$$
where $\ip{\cdot,\cdot}$ is the duality bracket.

Recall that the area of a face  $\face$ of $\cQ_N(\Sigma)$, with respect to the
Riemannian metric $g_{\sigma}^N$,  is equal to $N^{-2}$.
The $1$-form $\face^*$ is understood as a constant function equal to
$1$ on the face $\face$ and $0$ on other faces. This intuition gives
an interpretation of the duality bracket
$$
\ip{\cdot,\cdot} : C^2(\cQ_N(\Sigma))\times C_2(\cQ_N(\Sigma)) \to \RR
$$
as the \emph{pointwise} evaluation of functions on face.
This leads to a discrete analogue
$$\ipp{\cdot,\cdot }  : C^2(\cQ_N(\Sigma))\times C^2(\cQ_N(\Sigma))\to \RR$$
of the $L^2$-inner
product defined by
$$\ipp{\face^*_{1},\face^*_{2}} =
\left \{
\begin{array}{ll}
  0 &\mbox{ if $\face_1\neq \face_2$} \\
\frac 1{N^2} & \mbox{ if $\face_1=\face_2$}
\end{array}
\right . .
$$
The corresponding Euclidean norm on $C^2(\cQ_N(\Sigma))$  is
simply denoted $\|\cdot\|$.
Notice that the splitting of $$C^2(\cQ_N(\Sigma))\simeq
C^0(\cG_N(\Sigma))= C^0(\cG_N^+(\Sigma)) \oplus C^0(\cG^-_N(\Sigma))$$
is  orthogonal for $\ipp{\cdot,\cdot }$.
By construction, we have the following result:
\begin{prop}
  Let $\psi_{N_k}^\pm \in C^0(\cG_{N_k}^\pm(\Sigma))$ be a converging sequence
  of discrete functions with $\lim\psi^\pm_{N_k}=\phi ^\pm$. Then
$$
\lim \|\psi_{N_k}^\pm\|^2 = \frac 12 \|\phi^\pm\|^2_{L^2}
$$
where $\|\phi^\pm\|_{L^2}$ is the $L^2$-norm of $\phi^\pm$ with respect to
the Riemannian \emph{flat} metric $g_\sigma$. In particular if both sequences
converge and $\phi^+=\phi^-=\phi$, then $\lim \|\psi_{N_k} \|^2 = \|\phi
\|^2_{L^2}$.
\end{prop}
\begin{proof}
  Let $\phi^\pm_N$ be the sequence of samples of $\phi^\pm$. Then
  $\|\phi_N^\pm \|^2$ is understood as a Riemann sum for the integral
  $\|\phi^\pm\|^2_{L^2}$. Compared to a usual Riemann sum, we are
  throwing away half of the faces of the subdivision, and we have
$$
\lim \|\phi_N^\pm\|^2 = \frac 12 \|\phi^\pm\|^2_{L^2}.
$$
Using the $\cC^0$-convergence of $\psi_N^\pm$ and Proposition~\ref{prop:c0nec}, we
deduce that 
$$\lim \|\phi_N^\pm - \psi_N^\pm\|^2 = 0,$$
and the proposition follows by the triangle inequality.
\end{proof}

\subsubsection{The case of vector fields}
The space $T_\tau\cM_N$ consists of tangent vectors $V\in
C^0(\cQ_N(\Sigma))\otimes \RR^{2n}$. Here $V$ is understood as a
family of vectors,
given at each vertex $\vert$ of $\cQ_N(\Sigma)$ by
$V(\vert)=\ip{V,\vert}\in \RR^{2n}$. 
We deduce an Euclidean inner product on $C^0(\cQ_N)\otimes\RR^4$,
defined by
$$
\ipp{V,V'} =  \frac 1{N^2} \sum_{\vert \in \fC_0(\cQ_N(\Sigma))} g(V(\vert),V'(\vert)).
$$
The corresponding Euclidean norm is also denoted $\|\cdot\|$.

\subsection{Linearized equations}
Recall that the moduli space of quadrangular meshes $\cM_N$ is in fact
the vector  space
$C^0(\cQ_N(\Sigma))\otimes \RR^{2n}$.
So for $\tau\in \cM_N$, the tangent space at
$\tau$ is identified to
$$
T_\tau\cM_N = C^0(\cQ_N(\Sigma))\otimes \RR^{2n} = \scrM_N.
$$
Hence a tangent vector at $\tau$ is identified to a familly of vectors
of $\RR^{2n}$ defined
at each vertex of the quadrangulation.

The differential of $\mu_N^r:\cM_N\to C^2(\cQ_N(\Sigma))$ at $\tau$,
which is a linear map denoted
$$
D\mu^r_N|_\tau : T_\tau \cM_N \to C^2(\cQ_N(\Sigma)),
$$
is readily computed. Formally, we have
$$
D\mu^r_N|_\tau \cdot V = 2\Psi^r_N(\tau,V),
$$
where $\Psi^r_N$ is the symmetric bilinear map associated to the
quadratic map $\mu_N^r$.
For a more explicit formula, we merely need to compute the variation of the
symplectic area of a
quadrilateral in $\RR^{2n}$, which is being deformed by moving its
vertices. Let $V\in T_\tau\cM_N$ be a discrete vector field. We define a path of quadrangulations by
$$
\tau_t = \tau+ tV, \mbox{ for } t\in \RR.
$$
We would like to express the variation of $\mu_N^r$ along $\tau_t$. In order to state a result, we need some additional notations.

\subsubsection{Other diagonal notations}
We introduced the diagonals $D^u_\tau$ and $D^v_\tau$ at
\S\ref{sec:diag1}. We need now a slightly different indexing in order to
have a streamless expression of the differential of $\mu_N^r$.
We denote by $\face_{kl}$ for $k,l\in\ZZ$, the faces of  $\cQ_N(\RR^2)$. Their image under
the projection $p_N$ are still denoted $\face_{kl} \in
\fC_2(\cQ_N(\Sigma))$. Similarly, we denote by $\vert_{kl}$ the vertices
of $\cQ_N(\RR^2)$ and their image by $p_N$ as vertices of
$\cQ_N(\Sigma)$.
Let
$V\in T_\tau\cM_N = C^0(\cQ_N(\Sigma))\otimes\RR^{2n}$ be a vector  given
as family of vectors
$$V_{kl} = \ip{V,\vert_{kl}}\in
\RR^{2n}.$$
We define a deformation of the quadrangulation
$\tau$  by $\tau_t =\tau+tV$, or in coordinates
$$\ip{\tau_t,\vert_{kl}} = \ip{\tau,\vert_{kl}}+
t V_{kl}.$$

Let  $\tau\in\scrM_N$,  $\face\in \fC_2(\cQ_N(\RR^2))$ and
$\vert\in\fC_0(\cQ_N(\RR^2))$ be one of the vertices of $\face$. We enumerate
the vertices $(\vert_0,\vert_1,\vert_2,\vert_3)$ of $\face$  consistently with
the orientation and such that $\vert_0=\vert$.
The diagonals are defined by
$$D^\tau_{\vert,\face} = \tau(\vert_3) - \tau(\vert_1) \in \RR^{2n}
$$
and if $\vert$ is not a vertex of $\face$, we put
$D^\tau_{\vert,\face} =0$. Figure~\ref{figure:diagonal} shows a
diagrammatic representation of the above construction, with
orientation conventions.
\begin{figure}[H]
  \begin{pspicture}[showgrid=false](-2,-1)(2,1.3)
    \psset{linestyle=none, fillstyle=solid,
      fillcolor=lightgray}
    \psline(-1,1) (1,-1) (1,-1) (1,1) (1,1) (-1,1)

    \psset{linecolor=black,linestyle=solid, fillstyle=none, arrowsize=.3}
  \psline{->}(-1,1) (-1,-1)
  \psline{->}(-1,-1) (1,-1)
  \psline{->}(1,-1) (1,1)
  \psline{->}(1,1) (-1,1)
  \psset{fillstyle=solid,fillcolor=black}
  \pscircle(1,1){.1}
  
  \psset{linecolor=red}
  \psline{->}(-1,1)(1,-1)

  \rput(1.3,1.3){$\tau(\vert)$}
  \color{red}
  \rput(-.1,-.4){$D^\tau_{\vert,\face}$}

\end{pspicture}
\caption{A face $\face$ of a mesh $\tau$ with one diagonal and orientations}
\label{figure:diagonal}
\end{figure}
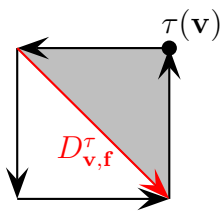

\begin{notation}
The vector $D^\tau_{\vert,\face}\in \RR^{2n}$ is called the diagonal opposite
to $\vert$ of the face $\face$ with respect to $\tau$.
\end{notation}
With these notations, we have the following expression for the
variation of the symplectic area:
\begin{lemma}
  \label{lemma:diffphi}
  $$
\frac d{dt}\ip {\mu_N(\tau_t),\face}|_{t=0} = - \frac 12\sum_{\vert\in \fC_0(\cQ_N(\Sigma))}\omega(V(\vert),D^\tau_{\vert,\face}) .
$$
\end{lemma}
\begin{proof}
  We use the ordered vertices $(A_0,A_1,A_2,A_3)$ of an oriented quadrilateral in
  $\RR^{2n}$ and   consider
  a variation $(A^t_0,A^t_1,A^t_2,A^t_3)= (A_0,A_1,A_2,A_3)+ t
  (V_0,V_1,V_2,V_3)$. We denote by $D_0^t$ and $D_1^t$ the diagonals
  of the deformed quadrilateral.
  By Lemma~\ref{lemma:quadiso}, its symplectic area is
  $$
\frac 12\omega(D_0^t,D_1^t).
  $$
Hence, the variation of
  symplectic area at $t=0$ is given by
  $$
\frac 12 \left ( - \omega (V_0,D_1) + \omega (V_1,D_0) + \omega
(V_2,D_1) - \omega (V_3,D_0) \right ).
$$
Using our conventions for the diagonals of quadrilaterals, this proves
the lemma.
\end{proof}
\subsubsection{Computation of the discrete Laplacian}
Any discrete vector field $V\in T_{\tau}\scrM_N$ is given by a
family of vectors
$$
V_v = \ip{V,v} \in\RR^{2n}.
$$
The almost complex structure $J$ of $\RR^{2n}\simeq\CC^n$ induces a canonical action on
$T_\tau\cM_N=C^2(\cQ_N(\Sigma))\otimes\RR^{2n}$ that can be expressed as
$$(JV)_v:=J(V_v).$$
Recall that the Euclidean metric $g$ and the symplectic form $\omega$
of $\RR^{2n}$ are related by the formula
$$
\forall u_1,u_2\in\RR^{2n}, \quad \omega(u_1,u_2) = g(Ju_1,u_2).
$$
According to the Lemma~\ref{lemma:diffphi}, the differential of
$\mu_N$ at $\tau_N$ satisfies
$$
\ip {D\mu_N|_{\tau}\cdot V,\face} =-\frac 12 \sum_{\vert} \omega (V(\vert),D^\tau_{\vert,\face}) .
$$
hence
$$
\ip {D\mu_N|_{\tau}\cdot JV,\face} =\frac 12 \sum_{\vert} g(V(\vert),D^\tau_{\vert,\face}) .
$$
In turn, we have
\begin{equation}
  \label{eq:DJ}
\ip {D\mu^r_{N}|_{\tau}\cdot JV,\face} = \frac {N^2}{2}\sum_{\vert}
g(V(\vert),D^\tau_{\vert,\face}).
\end{equation}
We introduce the operator (notice the analogy with Formula~\ref{eq:deltaf})
  \begin{equation}
  \label{eq:deltatau}
\boxed{\delta_\tau= -D\mu^r_{N}|_{\tau}\circ J,}
  \end{equation}
so that Formula~\eqref{eq:DJ} reads
\begin{equation}
  \label{eq:deltatauB}
  \ip {\delta_\tau V , \sum_{\face} \phi(\face)\face}
  = \frac {N^2}2\sum_{\vert,\face}
  \phi(\face) g(V(\vert),D^\tau_{\vert,\face})  .
\end{equation}
or, equivalently
\begin{equation}
  \label{eq:deltatauB2}
  \delta_\tau V = 
   \frac {N^2}2\sum_{\vert,\face}
   g(V(\vert),D^\tau_{\vert,\face}) \face^* .
\end{equation}
With the above conventions
\begin{equation}
  \label{eq:deltatauC}
  \ip {\delta_\tau V , \sum_{\face} \phi(\face)\face} =
  N^2 \ipp {\delta_\tau V , \sum_{\face} \phi(\face) \face^*}  
\end{equation}
and it follows from Formulae~\eqref{eq:deltatauB} and \eqref{eq:deltatauC} that
$$
\ipp {\delta_\tau V , \sum_{\face} \phi(\face) \face^* }= \frac 12\sum_{\vert,\face}
\phi(\face) g(V(\vert),D_{\vert,\face}^\tau)  .
$$
We deduce that the ajoint $\delta_\tau^\star$ of $\delta_\tau$ for the inner product
$\ipp{\cdot,\cdot}$ satisfies
\begin{align*}
\ipp {V ,\delta_\tau^\star \sum_{\face} \phi(\face) \face^*} &= \frac 12 \sum_{\vert}
\frac 1{N^2}
g\left (V(\vert), \sum_{\face} N^2\phi(\face)
D^\tau_{\vert,\face}\right ) \\
&= \ipp{V ,\frac{N^2}2\sum_{\vert,\face}\phi(\face) D^\tau_{\vert,\face} \vert^*}.  
\end{align*}
which proves the following lemma
\begin{lemma}
  \label{lemma:operators}
The operator
   $$\delta_\tau :  T_{\tau_N}\cM_N \to  C^2(\cQ_N(\Sigma))$$
 is given by
   \begin{equation}
     \label{eq:dN}
     \boxed{
       \delta_\tau V =  \frac {N^2}2 \sum_{\vert,\face} g(V(\vert),
       D^\tau_{\vert,\face}) \face^*,
       }
   \end{equation}
whereas its adjoint 
   $$\delta_\tau^\star :  C^2(\cQ_N(\Sigma))\to T_{\tau_N}\cM_N$$
 is given by
   \begin{equation}
     \label{eq:dNstar}
     \boxed{
       \delta^\star_\tau \phi =  \frac {N^2}2 \sum_{\vert,\face} \phi(\face)
       D^\tau_{\vert,\face} \otimes \vert^*.
       }
   \end{equation}
 \end{lemma}
 \begin{rmk}
   The operator $\delta_\tau =-D\mu^r_N|_\tau\circ J$ is the finite dimensional version of $\delta_f=-D\mu|_f\circ\fJ$ considered in the smooth setting (cf. \S\ref{sec:laplrel}).
   In the smooth setting, the adjoint $\delta_f^\star$ allows to recover the Hamiltonian infinitesimal action of the gauge group $\cG=\Ham(\Sigma,\sigma)$ on $\scrM$ according to the identity \eqref{eq:hamdstar}. In the finite dimensional approximation, there is no clear group action on $\scrM_N$  for which $\mu_N^r$ would be the corresponding moment map. However, the vector fields $V_\phi(\tau)=\delta_\tau^\star \phi$ define infinitesimal isometric Hamiltonian action which should play the role of finite dimensional approximations of $\Ham(\Sigma,\sigma)$.
 \end{rmk}

 The kernel of $\delta^\star_\tau$ contains the constants discrete
 functions, but might contain other function as well.
 This is not the case generically, according to the proposition below
 \begin{prop}
   \label{prop:generickernel}
   Let $\tau\in\scrM_N$ be a generic quadrangular mesh in the
   following sense: for every vertex $\vert$ of the quadrangulation
   $\cQ_N(\Sigma)$, the four possibly non vanishing diagonals
   $D^\tau_{\vert,\face}$, where $\face$ is a face that contains the vertex $\vert$ span a
   $3$-dimensional subspace of $\RR^{2n}$.
   Then the kernel of $\delta_\tau^\star$ reduces to constant discrete functions.
 \end{prop}
 \begin{proof}
The equation $\delta^\star_\tau\phi=0$ provides  a linear system of rank $3$
with four variables associated   to each vertex. This imply that
$\phi$ must be locally constant around each vertex and it follows that
$\phi$ is constant.
 \end{proof}

\begin{dfn}
Given a quadrangular mesh $\tau\in\scrM_N$, we define the discrete
Laplacian $\Delta_\tau : C^2(\cQ_N(\Sigma)) \to C^2(\cQ_N(\Sigma))$ associated to the mesh $\tau$ by
$$
\boxed{
  \Delta_\tau = \delta_\tau\delta_\tau ^\star.
  }
$$
Given a smooth isotropic immersion $\ell:\Sigma\to \RR^{2n}$ and its samples $\tau_N\in\scrM_N$,
  the associated  operators to $\delta_{\tau_N}$,
  $\delta_{\tau_N}^\star$ and $\Delta_{\tau_N}$ are denoted  $\delta_{N}$,
  $\delta_{N}^\star$ and $\Delta_{N}$, for simplicity.
\end{dfn}
\begin{rmk}
  Notice the analogy between the operator $\Delta_f$ defined by
  Formula~\eqref{eq:Deltaf} and $\Delta_\tau$. The operator
  $\Delta_\tau$ will play a central role in the perturbation theory of
  quadrangular meshes, as $\Delta_f$ did for smooth isotropic immersions. The reader should already be aware that $\Delta_\tau$ is \emph{not} the classical Laplacian associated to the mesh $\tau$, as will become clear from the sequel.
\end{rmk}

By
Formula~\eqref{eq:dNstar}
$$
\delta_\tau^\star \face^* =  \frac {N^2}2 \sum_{\vert} D_{\vert,\face} \vert^*,
$$
and by Formula~\eqref{eq:dN}
$$
\Delta_\tau \face^* = \frac {N^4}4 \sum_{\vert,\face_2}    g(D^\tau_{\vert,\face}, D^\tau_{\vert,\face_2}) \face_{2}^*.
   $$
We obtain the following result
\begin{prop}
  \label{prop:laplform}
For $\phi \in C^2(\cQ_N(\Sigma))$, we have
   $$
\boxed{
  \Delta_\tau\phi = \frac {N^4}4 \sum_{\vert,\face,\face_2}
  \phi(\face) g(D^\tau_{\vert,\face}, D^\tau_{\vert,\face_2})
  \face_{2}^*.
  }
   $$
\end{prop}

\subsection{Coefficients of the discrete Laplacian}
The discrete Laplacian $\Delta_N$ is an endomorphism of $C^2(\cQ_N(\Sigma))$ whose
coefficients are explicitely given by Proposition~\ref{prop:laplform}.
When dealing with $\tau_N$, we use the notation
$D_{\vert,\face}:=D^{\tau_N}_{\vert,\face}$ for simplicity.
We introduce the coefficients
$$
\beta_{\face_1\face_2} = \frac {N^4}4 \sum_{\vert}   g(D_{\vert,\face_1}, D_{\vert,\face_2}).
$$
By Proposition~\ref{prop:laplform}
   $$
\Delta_N \face_1^* =  \sum_{\face_1 ,\face_2}    \beta_{\face_1\face_2}  \face_{2}^*.
   $$
\subsubsection{Splitting of the Laplacian}
\label{sec:splitlapl}
The matrix $(\beta_{\face\face'})$ is obviously symmetric in $\face$
and $\face'$, which is not
surprising since $\Delta_N$ is
selfadjoint by definition.
The matrix is sparse in the sense that most of the coefficients
$\beta_{\face\face'}$ vanish.
There are three types of possibly non vanishing coefficients:
\begin{enumerate}
\item $\face=\face'$.
\item $\face$ and  $ \face'$ have only one vertex in common.
  \item $\face$ and  $ \face'$ have exactly one edge (and two vertices) in common.
\end{enumerate}
Using the above observation, we may write the operator $\Delta_N$ as a
sum
$$
\Delta_N = \Delta_N^E + \Delta_N^I.
$$
Here 
$$
\Delta^E_N \face^* = \frac {N^4}4 \sum_{
     \vert,\face_2 \in E_{12}(\face) }
g(D_{\vert,\face}, D_{\vert,\face_2}) \face_{2}^*
$$
where $E_{12}(\face)$ is the set of faces $\face_2$ such that the pair $(\face,\face_2)$
is of type (1) or (2), and
$$
\Delta^I_N \face^* = \frac {N^4}4 \sum_{
    \vert,\face_2\in E_3(\face)}
g(D_{\vert,\face}, D_{\vert,\face_2}) \face_{2}^*
$$
where $E_{3}(\face)$ is the set of faces $\face_2$ such that the pair $(\face,\face_2)$
is of type (3).
By definition, we have the following lemma
\begin{lemma}
  The operator $\Delta_N^E$ preserves the components of the direct
  sum decomposition 
  $C^2_+(\cQ_N(\Sigma))\oplus C^2_-(\cQ_N(\Sigma))$, whereas
  $\Delta_N^I$ exchanges the components. Accordingly, there have
  a block decomposition of the discrete Laplacian
  $$
  \Delta_N = \left (
  \begin{array}{c|c}
      \Delta_N^E 
    & \Delta_N^I \\
      \hline
      \Delta_N^I& \Delta_N^E
  \end{array}
  \right ).
  $$
\end{lemma}

\subsubsection{Finite difference operators and discrete Laplacian}
The smooth Laplacian~\eqref{eq:Deltaf} is related to a twisted
Riemannian Laplacian by Lemma~\ref{lemma:varmu}.
The goal of this section is to find a similar expression for the discrete
Laplacian $\Delta_N \phi$, using finite difference operators.

The strategy is to compute $\ip{\Delta_N\phi,\face}$ at some face $\face$ of
$\cQ_N(\Sigma)$. For this purpose, we will use the notations $\face_{kl}$
and $\vert_{ij}$ for faces and vertices of $\cQ_N(\RR^2)$, considered
as vertices and faces of $\cQ_N(\Sigma)$
(cf. \S\ref{sec:quadnot} and \S\ref{sec:diag1}).
The values of a discrete function are denoted
$$
\phi_{kl}=\ip{\phi,\face_{kl}},
$$
and the diagonals $D_{ijkl}$ are obtained as
$D_{\vert_{ij}\face_{kl}}$, with the convention that $D_{ijkl}=0$ if
$\vert_{ij}$ is not a vertex of the face $\face_{kl}$ in
$\cQ_N(\RR^2)$.
The coefficients $\beta_{klmn}$ are denoted $\beta_{\face_{kl}\face_{mn}}$
and we choose the integers $k,l$ so that $\face=\face_{kl}$. The 
coefficients $\beta_{\face\face'}$  vanishes unless $\face=\face'$ or $\face$ and $\face'$ are
contiguous faces. In such case we may  choose a unique pair of
integers $(m,n)$ such that
$\face'=\face_{mn}$ with $m\in \{k-1,k,k+1\}$ and $n\in\{l-1,l,l+1\}$.
Under these conditions (cf. \S\ref{sec:splitlapl})
\begin{enumerate}
  \item $\face_{kl}$ and $\face_{mn}$  are of type (1) if $(k,l)=(m,n)$,
\item  $\face_{kl}$ and $\face_{mn}$  are of 
type (2) if $(m,n) = (k\pm 1, l\pm 1)$ or $(k\pm 1, l\mp 1)$,
\item $\face_{kl}$ and $\face_{mn}$ are of
 type (3) if $(m,n)= (k\pm 1,l)$ or $(k,l\pm 1)$.
\end{enumerate}

For the first type of coefficients, we find
$$
\beta_{klkl}=  \frac {N^4}4 \sum_{ij}  \|D_{ijkl}\|^2,
$$
where we may take the sum over all pairs of indices $i,j\in\ZZ$.
For  the second type of cooefficients, we have
$$
\beta_{klmn}=  \frac {N^4}4 g(D_{ijkl},D_{ijmn})
$$
where $\vert_{ij}$ is the common vertex of $\face_{kl}$ and
$\face_{mn}$ in $\cQ_N(\RR^2)$.
In the third case there are two common vertices $\vert_{ij}$ and $\vert_{i'j'}$
which belong to the same edge. Then
$$
\beta_{klmn} =  \frac {N^4}4 g(D_{ijkl},D_{ijmn}) + \frac {N^4}4 g(D_{i'j'kl},D_{i'j'mn}).
$$

For simplicity of notations, we also use the notations $D_{kl}^u$ and
$D_{kl}^v$ for the diagonal (cf. \S\ref{sec:diag1}), which differ only by a sign. We start our
computations with the operator $\Delta^E_N$:
\begin{align*}
  4N^{-4} & \ip{\Delta^E_N \phi,\face_{kl}}  = - \phi_{k-1,l-1}
  g(D_{k-1,l-1}^v, D_{kl}^v)  - \phi_{k+1,l+1}
  g(D_{k+1,l+1}^v, D_{kl}^v)\\
  &  - \phi_{k-1,l+1}
  g(D_{k-1,l+1}^u, D_{kl}^u)  - \phi_{k+1,l-1}
  g(D_{k+1,l- 1}^u, D_{kl}^u) \\
  &+ 2\phi_{kl}(g(D_{kl}^u, D_{kl}^u) + g(D_{kl}^v, D_{kl}^v))
\end{align*}

One can write
\begin{align*}
  \phi_{k+1,l+1}g(D^v_{k+1,l+1}, D^v_{kl})= & \phi_{k+1,l+1}g(D^v_{k,l},
  D^v_{kl}) \\
  + & \phi_{kl}  g(D^v_{k+1,l+1}- D^v_{kl}, D^v_{kl}) \\
  + & (\phi_{k+1,l+1}-\phi_{k,l})g(D^v_{k+1,l+1}-D^v_{kl}, D^v_{kl})  
\end{align*}
and similar Leibnitz type decomposition for the other terms.
Thus, we obtain accordingly
\begin{align*}
  4N^{-4} \ip{\Delta^E_{N} \phi,\face_{kl}} & = \\
  &(- \phi_{k-1,l-1} -
  \phi_{k+1,l+1} +2 \phi_{kl})  g(D_{kl}^v, D_{kl}^v) \\
  +&  (- \phi_{k-1,l+1} -
  \phi_{k+1,l-1} +2 \phi_{kl})  g(D_{kl}^u, D_{kl}^u)\\
  - &  \phi_{kl}    g(D_{k-1,l-1}^v +  D_{k+1,l+1}^v - 2D_{kl}^v,
  D_{kl}^v) \\
  -&\phi_{kl}  g(D_{k+1,l-1}^v + D_{k-1,l+1}^u - 2D_{kl}^u, D_{kl}^u)
  \\
  -& (\phi_{k-1,l-1} - \phi_{kl}) g(D^v_{k-1,l-1} -D^v_{kl},D^v_{kl})\\
    -& (\phi_{k+1,l+1} - \phi_{kl}) g(D^v_{k+1,l+1}
    -D^v_{kl},D^v_{kl})   \\
  -& (\phi_{k-1,l+1} - \phi_{kl}) g(D^u_{k-1,l+1} -D^u_{kl},D^u_{kl})\\
    -& (\phi_{k+1,l-1} - \phi_{kl}) g(D^u_{k+1,l-1} -D^u_{kl},D^u_{kl})
\end{align*}

We gather the RHS into a sum of three operators:
first we define $\hat \Delta^E_N$. This operator will turn out to be a
discrete version of the Riemannian Laplace-Beltrami operator on $\Sigma$.
\begin{align*}
  4N^{-4} \ip{\hat \Delta^E_{N} \phi,\face_{kl}} & = \\
  &(- \phi_{k-1,l-1} -
  \phi_{k+1,l+1} +2 \phi_{kl})  g(D_{kl}^v, D_{kl}^v) \\
  +&  (- \phi_{k-1,l+1} -
  \phi_{k+1,l-1} +2 \phi_{kl})  g(D_{kl}^u, D_{kl}^u)
  \end{align*}
then we define the operator $K^E_{N}$, which is some kind of discrete
curvature operator by
\begin{align*}
  4N^{-4} \ip{K^E_{N}\phi,\face_{kl}} & = \\
   - &   \phi_{kl}   g(D_{k-1,l-1}^v +  D_{k+1,l+1}^v - 2D_{kl}^v,
   D_{kl}^v)   \\
   -&\phi_{kl} g(D_{k+1,l-1}^v + D_{k-1,l+1}^u - 2D_{kl}^u, D_{kl}^u)
\end{align*}
The last four lines can be rearranged into an operator $\Gamma^E_N$
given by
\begin{align*}
  4N^{-4} \ip{\Gamma^E_N \phi,\face_{kl}} & =  \\
  -& \frac 12(\phi_{k-1,l-1} - \phi_{kl}) (|D^v_{k-1,l-1}|_g^2 -|D^v_{kl}|_g^2)\\
    -& \frac 12(\phi_{k+1,l+1} - \phi_{kl}) (|D^v_{k+1,l+1}|_g^2
    -|D^v_{kl}|_g^2)\\
  -& \frac 12(\phi_{k-1,l+1} - \phi_{kl}) (|D^u_{k-1,l+1}|_g^2 -|D^u_{kl}|_g^2)\\
    -& \frac 12(\phi_{k+1,l-1} - \phi_{kl}) (|D^u_{k+1,l-1}|_g^2
    -|D^u_{kl}|_g^2)
\end{align*}
plus an operator
\begin{align*}
 4N^{-4} \ip{\cE^E_{N} \phi,\face_{kl}} & =  \\
  -& \frac 12(\phi_{k-1,l-1} - \phi_{kl}) |D^v_{k-1,l-1}-D^v_{kl}|_g^2\\
    -& \frac 12(\phi_{k+1,l+1} - \phi_{kl}) |D^v_{k+1,l+1}   - D^v_{kl}|_g^2\\
  -& \frac 12(\phi_{k-1,l+1} - \phi_{kl}) |D^u_{k-1,l+1} - D^u_{kl}|_g^2\\
    -& \frac 12(\phi_{k+1,l-1} - \phi_{kl}) |D^u_{k+1,l-1}  -D^u_{kl}|_g^2
\end{align*}
So, we have a decomposition
$$
\Delta_N^E = \hat\Delta_N^E + K^E_N + \Gamma_N^E + \cE^E_N.
$$

Similar computations can be carried out for $\Delta^I_N$.
\begin{align*}
  4N^{-4} & \ip{\Delta^I_N \phi,\face_{kl}}  = \\
  & \phi_{k+1,l}
  (-g(D_{k+1,l}^v, D_{kl}^u)  - g(D_{k+1,l}^u, D_{kl}^v))\\
   & \phi_{k,l+1}
  (g(D_{k,l+1}^v, D_{kl}^u)  + g(D_{k,l+1}^u, D_{kl}^v))\\
   & \phi_{k-1,l}
  (-g(D_{k-1,l}^v, D_{kl}^u)  - g(D_{k-1,l}^u, D_{kl}^v))\\
   & \phi_{k,l-1}
  (+g(D_{k,l-1}^v, D_{kl}^u)  + g(D_{k,l-1}^u, D_{kl}^v))  
\end{align*}

We introduce the averaging operator $\phi\mapsto \bar \phi$ defined by
$$
\bar \phi_{kl}=\ip{\bar\phi, \face_{kl}}= \frac 14 \left (\phi_{k+1,l} +
\phi_{k-1,l} + \phi_{k,l+1} + \phi_{k,l-1}\right )
$$
  
  and we write each term above under the form
  \begin{align*}
    \phi_{k+1,l}
    g(D_{k+1,l}^v, D_{kl}^u) & =\\ & (\bar\phi_{kl} +
    (\phi_{k+1,l}-\bar \phi_{kl})) g\left (D_{k,l}^v +
    (D_{k+1,l}^v-D_{k,l}^v), D_{kl}^u \right )  
  \end{align*}
  Expanding these expressions leads to
  \begin{align*}
    4N^{-4} & \ip{\Delta^I_N \phi,f_{kl}}  = \\
& \bar\phi_{kl} g\left ( (D^v_{k,l+1}-D^v_{k+1,l}) -
    (D^v_{k-1,l}-D^v_{k,l-1}),D^u_{kl} \right )
    \\
+ & \bar\phi_{kl} g\left ( (D^u_{k,l+1}-D^u_{k+1,l}) -
(D^u_{k-1,l}-D^u_{k,l-1}),D^v_{kl} \right )
\\
+&2 ((\phi_{k,l+1} - \phi_{k-1,l}) - (\phi_{k+1,l}-\phi_{k,l-1}))
g(D^v_{kl},D^u_{kl})  \\
 +& (\phi_{k+1,l}-\bar \phi_{kl})  (-g(D_{k+1,l}^v - D_{kl}^v,
 D_{kl}^u)  - g(D_{k+1,l}^u - D_{kl}^u, D_{kl}^v))\\
 +& (\phi_{k,l+1}-\bar \phi_{kl})  (g(D_{k,l+1}^v - D_{kl}^v,
 D_{kl}^u)  + g(D_{k,l+1}^u - D_{kl}^u, D_{kl}^v))\\
 +& (\phi_{k-1,l}-\bar \phi_{kl})  (-g(D_{k-1,l}^v - D_{kl}^v,
 D_{kl}^u)  - g(D_{k-1,l}^u - D_{kl}^u, D_{kl}^v))\\
 +& (\phi_{k,l-1}-\bar \phi_{kl})  (g(D_{k,l-1}^v - D_{kl}^v,
 D_{kl}^u)  + g(D_{k,l-1}^u - D_{kl}^u, D_{kl}^v))\\
\end{align*}
  The first two lines can be expressed using Chasles relation as an
  operator

  \begin{align*}
   4N^{-4} \ip{K^I_{N}\phi,\face_{kl}} = &\\
    &\bar\phi_{kl}g(D^u_{k-1,l+1}+D^u_{k+1,l-1}-2D^u_{kl},D^u_{kl})\\
    +&\bar\phi_{kl}g(D^v_{k+1,l+1}+D^v_{k-1,l-1}-2D^v_{kl},D^v_{kl})
  \end{align*}
  In particular, we see that if $\phi_{kl}=\bar\phi_{kl}$, then
  $\ip{(K^I_{\tau_0}+K^E_{\tau_0})\phi,\face_{kl}} = 0$. We decompose
  $\Delta^I$ as a sum
  $$
\Delta^I_N= K^I_N +\cE^I_N.
  $$

All the above operators may be expressend in terms of finite differences. First
we define  analogue $\theta^u_N, \theta_N^v\in C^2(\cG_N(\Sigma))$ of the conformal
factor $\theta$ by
$$
\theta^u_N = \|\scrU_N\|^2_g, \mbox { and } \theta^v_N = \|\scrV_N\|^2_g.
$$
and a discrete analogue of the Gau{\ss} curvature plus an energy term $\kappa_N\in C^2(\cG_N(\Sigma))$ given by
$$
\kappa_N =- g\left (\frac{\del^2}{\del\cev u \del\vec u}\scrV_N
,\scrV_N\right )  - g\left (\frac{\del^2}{\del\cev v \del\vec
  v}\scrU_N ,\scrU_N\right )
$$
\begin{prop}
  The operators introduced above satisfy the following identites for
  every discrete function $\phi$:
  \label{prop:formopl}
  $$
\hat\Delta ^E_N\phi= - \left (\theta_N^u \frac{\del^2}{\del\cev u \del\vec u}
+ \theta_N^v \frac{\del^2}{\del\cev v \del\vec v}\right ) \phi
$$
$$
K^E_N\phi = \kappa_N\cdot\phi
$$
\begin{align*}
\Gamma_N^E\phi =& -\frac 12 \left (\frac {\del\phi}{\del\cev u}  \frac {\del\theta_N^v}{\del\cev
  u}+\frac {\del\phi}{\del\vec u}  \frac {\del\theta_N^v}{\del\vec
  u}+\frac {\del\phi}{\del\cev v}  \frac {\del\theta_N^u}{\del\cev
  v}+\frac {\del\phi}{\del\vec v}  \frac {\del\theta_N^u}{\del\vec
  v} \right )\\
=&- \grad\phi \cdot \grad \theta_N
\end{align*}
and
$$
K_N^I\phi = -\kappa_N\cdot \bar\phi.
$$
\end{prop}
The operators $\cE^E_N$ and $\cE^I_N$ become negligible as $N$ goes to
infinity, in the sense of the following proposition:
\begin{prop}
  \label{prop:errors}
There exists a sequence $\epsilon_N=\cO(N^{-1})$ with $\epsilon_N>0$
such that for all $N$ and all
functions $\phi\in C^2(\cQ_N(\Sigma))$, we have
$$
\|\cE^I_N\phi \|_{\cC^{0,\alpha}}\leq \epsilon_N \|\phi
\|_{\cC^{2,\alpha}}
\quad \mbox { and } \quad
\|\cE^E_N\phi \|_{\cC^{0,\alpha}}\leq \epsilon_N \|\phi \|_{\cC^{2,\alpha}}
$$
and
$$
\|\cE^I_N\phi \|_{\cC^{0}} \leq \epsilon_N \|\phi
\|_{\cC^{2}}
\quad \mbox { and } \quad
\|\cE^E_N\phi \|_{\cC^{0}} \leq \epsilon_N \|\phi \|_{\cC^{2}}
$$

\end{prop}

\section{Limit operator}
\label{sec:limit}

\subsection{Computation of the limit operator}
We denote by $\Delta_\sigma$ (resp. $\Delta_\Sigma$  the Laplace-Beltrami operator associated to the Riemannian metric $g_\sigma$ (resp. $g_\Sigma)$ on $\Sigma$.
\begin{theo}
  \label{theo:limit}
  Let $k$ be an integer such that $k\geq 2$.
  For every sequence of discrete functions $\psi_{N_k}\in
  C^2(\cQ_{N_k}(\Sigma))$, converging in the $\cC^k_w$-sense toward a
  pair of functions 
  $(\phi^+,\phi^-)$, we have
  $$
  \Delta_{N_k}\psi_{N_k} \stackrel{\cC^{k-2}}{\longrightarrow } \opl(\phi^+,\phi^-)
  $$
  where $\opl(\phi^+,\phi^-)$ is the pair of functions defined by
    \begin{align*}
\opl (\phi^+,\phi^-)= \Big ( & \theta\Delta_{\sigma}\phi^+ - g_\sigma (d\phi^+,d\theta)
+ (K+E)(\phi^+ - \phi^-),\\
& \theta\Delta_{\sigma}\phi^- - g_\sigma (d\phi^-,d\theta)
+ (K+E)(\phi^- - \phi^+)    \Big ),
    \end{align*}
     $K$ is the Gau{\ss} curvature of $g_\Sigma$ and $E$ is a nonnegative function on $\Sigma$ defined at~\eqref{eq:defE}.
\end{theo}
\begin{proof}
  The result is a trivial consequence of
  Proposition~\ref{prop:formopl} and the convergence of the
  coefficients of the operator.
  The only non trivial fact that must be proved is the following
  lemma:
  \begin{lemma}
    \label{lemma:gauss}
    We have the identity
$$
K + E =-g\left (\frac{\del^3\ell}{\del u^2\del v}, \frac{\del\ell}{\del v}\right ) - g\left (\frac{\del^3\ell}{\del v^2\del u}, \frac{\del\ell}{\del u}\right ),
$$
where $K$ is the Gau{\ss} curvature of the metric $g^\Sigma$ and $E$ is the nonnegative function on $\Sigma$ defined via the second fundamental form $\II$ of $\ell:\Sigma\to \RR^ {2n}$
\begin{equation}
  \label{eq:defE}
E = 2 g\left (
\frac{\del^2 \ell}{\del u\del v}^\perp,\frac{\del^2 \ell}{\del u\del
  v}^\perp \right ) = 2g\left (\II \left (\frac\del{\del u},
\frac\del{\del v}\right ), \II \left (\frac\del{\del u},
\frac\del{\del v}\right )\right  ).  
\end{equation}
\end{lemma}
\begin{proof}
  Recall the standard formula for the Gau{\ss} curvature $K$ of the metric $g^\Sigma= \ell^* g = \theta g_\sigma$, conformal to the flat metric $g_\sigma$:
  $$
K= \frac 12 \theta\Delta_\sigma \log\theta.
$$
Using the classical identity
$$
\theta\Delta_\sigma \log\theta = \Delta_\sigma\theta + \theta^{-1}g_\sigma (d\theta,d\theta)
$$
and using the fact that
$$
\theta = g\left (\frac{\del \ell}{\del u}, \frac{\del \ell}{\del u}\right ) = g \left (\frac{\del \ell}{\del v}, \frac{\del \ell}{\del v}\right ),
$$
we compute
\begin{equation}
  \label{eq:K1}
\frac{\del\theta}{\del v} = 2g \left (\frac{\del^2 \ell}{\del u\del v},
\frac{\del \ell}{\del u}\right ),\quad \frac{\del\theta}{\del u} = 2g\left (\frac{\del^2 \ell}{\del u\del v}, \frac{\del \ell}{\del v}\right ),  
\end{equation}
hence
$$
\frac{\del^2\theta}{\del v^2} = 2g\left (\frac{\del^3 \ell}{\del u\del v^2}, \frac{\del \ell}{\del u}\right )
+  2g\left (\frac{\del^2 \ell}{\del u\del v}, \frac{\del^2 \ell}{\del u\del v} \right )
$$
and
$$
\frac{\del^2\theta}{\del u^2} = 2g \left (
\frac{\del^3 \ell}{\del u^2\del v}, \frac{\del \ell}{\del u}
\right )
+  2g \left (
\frac{\del^2 \ell}{\del u\del v}, \frac{\del^2 \ell}{\del u\del v}
\right ).
$$
In particular
\begin{align*}
  g_\sigma (d\theta,d\theta ) = & \left |\frac {\del \theta}{\del u} \right |^2 + \left |\frac {\del \theta}{\del v} \right |^2 \\
  = &  4g \left (\frac{\del^2 \ell}{\del u\del
  v},\frac{\del \ell}{\del u} \right ) ^2 +4g \left (\frac{\del^2 \ell}{\del u\del
    v}, \frac{\del \ell}{\del v} \right )^2
\end{align*}
thanks to Formula~\eqref{eq:K1}. The fact that $\frac{\del\ell}{\del u}$ and $\frac{\del\ell}{\del v}$ is an orthogonal family of vectors of $g$-norm $\sqrt\theta$ implies that any vector $V\in \RR^{2n}$ satisfies the identity
$$
g \left (V,\frac{\del \ell}{\del u} \right ) ^2 +g \left (V, \frac{\del \ell}{\del v} \right )^2 = \theta g
\left (V^T,V^T \right ),
$$
where $V^T$ is the $g$-orthogonal projection of $V$ onto the plane spaned by  $\frac{\del\ell}{\del u}$ and $\frac{\del\ell}{\del v}$. In other words, $V^T$ is the $g$-orthogonal projection onto the tangent plane to $\ell(\Sigma)$.
Therefore 
$$  \theta^{-1} g_\sigma (d\theta,d\theta ) =    4 g\left (
\frac{\del^2 \ell}{\del u\del v}^T,\frac{\del^2 \ell}{\del u\del v}^T \right )  
$$

and
$$
\Delta_\sigma\theta = -2g\left (\frac{\del^3 \ell}{\del u\del v^2},
\frac{\del \ell}{\del u}\right )-2g\left (\frac{\del^3 \ell}{\del u^2\del v},
\frac{\del \ell}{\del v}\right ) -
4g \left (\frac{\del^2 \ell}{\del u\del v},\frac{\del^2 \ell}{\del u\del v} \right ).
$$
In conclusion
\begin{align*}
  2K = & \theta \Delta_\sigma \log \theta \\
   = & \Delta_\sigma \theta + \theta^{-1}g_\sigma(d\theta,d\theta) \\
  = & -2g(\frac{\del^3 \ell}{\del u\del v^2}, \frac{\del \ell}{\del u})-2g(\frac{\del^3 \ell}{\del u^2\del v}, \frac{\del \ell}{\del v})  -  4 g\left (
\frac{\del^2 \ell}{\del u\del v}^\perp,\frac{\del^2 \ell}{\del u\del v}^\perp \right )  
\end{align*}
where $\perp$ denotes the component of a vector orthogonal to the tangent space to $\ell$ at a point.
\end{proof}
\begin{cor}
  \label{cor:kappaconv}
  For all integer $k\geq 0$
  $$
\kappa_N \stackrel{\cC^k}\longrightarrow K+E.
  $$
\end{cor}
The coefficients of $\Delta_N$ are now all understood asymptotically. This complete the proof of the theorem.
\end{proof}

\begin{dfn}
  The operator defined by
  \begin{align*}
\opl(\phi^+,\phi^-)=  \Big ( &\theta\Delta_{\sigma}\phi^+ - g_\sigma(d\phi^+,d\theta)
+ (K+E)(\phi^+ - \phi^-),\\
& \theta\Delta_{\sigma}\phi^- - g_\sigma (d\phi^-,d\theta)
+ (K+E)(\phi^- - \phi^+)    \Big )
  \end{align*}
is called the limit operator of $\Delta_{N}$.
\end{dfn}
\begin{rmk}
  In particular, the limit operator $\Xi$ is elliptic. This fact will be crucial to derive uniform discrete Schauder estimates for $\Delta_N$.
\end{rmk}
\subsection{Kernel of the limit operator}

\begin{prop}
\label{prop:kernel}
  A pair of smooth functions $(\phi^+,\phi^-)$ is an element of the kernel of the limit
  operator $\opl$ if, and only if, there exists some real constants
  $c_0$ and $c_1$ such that
  $$
\phi^+ = c_0 + c_1\theta^{-1},\quad  \phi^- = c_0 - c_1\theta^{-1}.
$$
with  $c_1=0$, unless the function $E$ vanishes identically on $\Sigma$.
In particular the kernel of $\opl$ has dimension $1$ or $2$ depending on the vanishing of $E$.
\end{prop}
\begin{proof}
The Proposition is proved by a straightforward argument using integration by part. A few formulae are needed in order to give a streamlined proof:  
\begin{lemma}
  \label{lemma:kernel}
For every smooth function $f:\Sigma\to \RR$, we have
$$
d^{*_\sigma}\theta df =\theta\Delta_{\sigma} f - g_\sigma(d f, d \theta ),
$$
where $d^ {*_\sigma}$ is the adjoint of $d$ with respect to the $L^2$-inner product induced by $g_\sigma$.
On the other hand, we have
$$
\theta d^{*_\sigma} \theta^{-1} d\theta f = \theta\Delta_\sigma f - g_\sigma (d f, d \theta)  + 2Kf
$$
where $K$ is the Gau{\ss} curvature of $g_\Sigma$.
\end{lemma}
\begin{proof}
  For every $1$-form $\beta$ 
  and every function $w$ on $\Sigma$, we have
  $d^ {*_\sigma}(w\beta) = -*_\sigma d*_\sigma  (w\beta)= -*_\sigma d(w*_\sigma \beta) = -*_\sigma (dw\wedge
*_\sigma \beta + wd*_\sigma \beta)= wd^{*_\sigma}\beta-*_\sigma  g_\sigma(dw,\beta)\vol_\sigma =
wd^{*_\sigma}\beta- g_\sigma (dw,\beta )$. In conclusion
  $$d^ {*_\sigma}(w\beta) = wd^{*_\sigma}\beta-
g_\sigma(dw,\beta).$$

The first formula of the lemma follows from the above identity.
For second identity, we have $\theta^{-1}d(\theta f) = fd\log\theta +
df$. Now, $d^{*_\sigma}\theta^{-1}d(\theta f) = fd^{*_\sigma}d\log \theta -
g_\sigma (df,d\log\theta) +d^{*_\sigma}df$. We use the fact that the Gau{\ss} curvature
of $g_\Sigma$ is given by the formula
$2K= \theta\Delta_\sigma \log\theta$ and deduce the second identity of the lemma.
\end{proof}

We may now complete the proof ot Proposition~\ref{prop:kernel}.
Let $\phi^\pm$ be a solution of the system
\begin{align*}
  \theta\Delta_{\sigma}\phi^+ - g_\sigma ({d \phi^+, d \theta}) + (K+E)(\phi^+
  - \phi^-) &= 0 \\
  \theta\Delta_{\sigma}\phi^- - g_\sigma({d \phi^-, d \theta})  + (K+E)(\phi^-
  - \phi^+)&= 0.
\end{align*}
Adding up the two equations gives the identity
$$
\theta\Delta_{\sigma}(\phi^++\phi^-) -  \ip{d (\phi^++\phi^-), d
  \theta}_\sigma =0 = d^{*_\sigma}\theta d(\phi^++\phi^-)
$$
by Lemma~\ref{lemma:kernel}. Integrating against $\phi^++\phi^-$ using the $L^2$-inner product induced by $g_\sigma$ gives $0=\ip{d^*\theta d(\phi^++\phi^-), \phi^++\phi^-}_{L^2} = \ip{\theta d(\phi^++\phi^-), d(\phi^++\phi^-)}_{L^2}$. Since $\theta$ is positive, this   forces
$$
\phi^+ + \phi^- = 2c_0,
$$
for some constant $c_0$.

On the other hand the difference of the two equations provides the identity
$$
\theta\Delta_{\sigma}(\phi^+-\phi^-) - g_\sigma (d (\phi^+-\phi^-),
  d\theta ) + 2(K+E)(\phi^+ - \phi^-)=0.
$$
by Lemma~\ref{lemma:kernel} we deduce that
$$
 \theta d^{*_\sigma}\theta^{-1} d\theta(\phi^+-\phi^-) +2E (\phi^+ - \phi^-) =0.
 $$
 Integrating the above equation against $\theta(\phi^+-\phi^-)$ provides the identity
 $$
\ip{\theta d(\phi^+-\phi^-),  d(\phi^+-\phi^-) }_{L^2} + \ip{2E (\phi^+-\phi^-),  \phi^+-\phi^-}_{L^2}=0.
$$
Now $\theta$ is positive and $E$ is nonnegative, so the two terms of the LHS are non negative: they must vanish both.
The vanishing of the first term forces
$$
\phi^+ -\phi^- = 2\theta^{-1} c_1 ,
$$
for somme real constant $c_1$. The vanishing of the second term implies that $c_1=0$ unless $E$ vanishes identically on $\Sigma$.
\end{proof}

\subsection{Degenerate families of quadrangulations}
Proposition~\ref{prop:kernel} leads us to distinguish two types of
constructions.

Recall that the construction of $\cQ_N(\Sigma)$ depends on the choice
of a
Riemannian universal cover $p:\RR^2\to \Sigma$ for the
  flat metric $g_\sigma$ on $\Sigma$.  Such cover are not unique. They may be, for instance, precomposed with a rotation of $\RR^2$. Equivalently, we may
  replace the canonical basis of $\RR^2$ by a rotated basis, which
  also provides rotated $(u,v)$-coordinates.
  
We introduce a definition of degeneracy, bearing on pairs
$(p,\ell)$,
that consists of an isotropic
  immersion $\ell:\Sigma\to\RR^{2n}$ and a choice of Riemannian cover
  $p:\RR^2\to\Sigma$ for a flat metric $g_\sigma$, in the conformal
  class of the induced metric $g_\Sigma$.
  \begin{dfn}
    \label{dfn:degenerate}
  We say that the pair $(p, \ell)$ is degenerate, if the function $E:\Sigma\to\RR$ defined
  by~\eqref{eq:defE} vanishes identically. Otherwise, we say that the
  pair $(p,\ell)$ is nondegenerate. 
\end{dfn}
\begin{example}
     An example of degenerate pair is provided by
   the map 
  $$\ell:\RR^2 \longrightarrow \CC\otimes \CC\simeq \RR^{4}$$ 
  defined by 
  $$\ell(x,y)=(\exp (2\pi iu)), \exp(2\pi iv))\in \CC^2,
  $$
  where $(x,y)$ are the canonical coordinates of $\RR^2$ and $(u,v)$ are the rotated coordinates  defined by~\eqref{eq:uv}.
  This map clearly satisfies
  \begin{equation}
    \label{eq:degex}
   \frac{\del^2\ell}{\del u\del v}=0 .
  \end{equation}
  Moreover, $\ell$ is invariant under the lattice $\Gamma$ spanned by $\frac{e_1+e_2}{\sqrt{2}}$  and $\frac{e_2-e_1}{\sqrt{2}}$. Hence $\ell$ descends to a quotient map denoted $\ell:\RR^2/\Gamma\to\CC^2$.
  We obtain a pair $(p,\ell)$, where $p:\RR^2\to\RR^2/\Gamma$ is the canonical projection,  which is degenerate in the sense of Definition~\ref{dfn:degenerate} by~\eqref{eq:degex}.
\end{example}

Degenerate pairs can create additional technical
difficulties. Nevertheless, they may be taken care of with some additional
caution (cf. \S\ref{sec:deg}). Or they can just be avoided according to the following proposition:
  \begin{prop}
    \label{prop:avoid}
    Given a pair $(p,\ell)$, there always exists a rotation $r$ of
    $\RR^2$ such that $(p\circ r,\ell)$ is non degenerate.
\end{prop}
\begin{proof}
   The $(u,v)$ coordinates of $\RR^2$ induce an
  orthonormal basis of tangent vectors of $\Sigma$ for the metric
  $g_\sigma$, denoted $\frac{\del}{\del u}$,  $\frac{\del}{\del
    v}$. If $(p,\ell)$ is degenerate,  $\II$ must vanishes identically on this pair of
  vector fields. If $(p\circ r,\ell)$ is degenerate for every rotation
  of $\RR^2$,  the second fundamental form must also
  vanish for every pair of tangent vectors obtained by rotating
  the basis $\frac{\del}{\del u}$,  $\frac{\del}{\del v}$. This means
  that $\II$ vanishes on every pair of orthogonal tangent vectors for
  $g_\sigma$. Since $g_\Sigma$ is conformal to $g_\sigma$, this means
  that $\II$ must vanish for every pair of orthogonal tangent vectors
  for $g_\Sigma$. This is a contradiction according to the following
  lemma:
  \begin{lemma}
    For any immersion $\ell:\Sigma\to\RR^{2n}$, where $\Sigma$ is a
    closed surface diffeomorphic to a torus, there exists a point
    $x\in\Sigma$ and an orthogonal basis of tangent vectors $U$, $V$
    for the induced metric $g_\Sigma$, such that the second
    fundamental form satisfies $\II(U,V)\neq 0$.
  \end{lemma}
  \begin{proof}
    We choose a point $x\in T_x\Sigma$. 
    Assume that $\II(U,V)$ vanishes for every orthonormal basis 
    $(U,V)$ of $T_x\Sigma$. Notice that in this case $(U+V,U-V)$ is an
    orthogonal basis, hence, by assumption $\II(U+V,U-V)=0$, and we have
    $$
\II(U,U)=\II(U+V,U-V) +\II(V,V) = \II(V,V).
$$
By the Gau{\ss} Theorema Egregium,  the curvature $K$ of
$g_\Sigma$ is given by
$$
K = - g(\II(U,V),\II(V,U)) +g(\II(U,U),\II(V,V)).
$$
According to our discussion, we deduce that
$$
K=g(\II(U,U),\II(U,U))\geq 0.
$$
By the Gau{\ss}-Bonnet formula, a torus with nonnegative curvature has
vanishing curvature. Thus $K=0$, and as a corollary $\II(U,U)=0$, which
implies that $\II=0$. In conclusion the image of $\ell:\Sigma\to\RR^{2n}$
is totally geodesic. The only totally geodesic surfaces of $\RR^{2n}$ are
$2$-planes. This forces the image of $\ell$ to be contained in a
plane. This is not possible for an immersion of a compact surface.
  \end{proof}
In conclusion there is a choice of rotation $r$ such that $(p\circ
r,\ell)$ is nondegenerate,  which proves the proposition.
\end{proof}

\subsection{Schauder Estimates}
The following result is a consequence of a theorem of Thom\'ee, stated in a broader context
\cite{Thomee68}, for various elliptic finite difference operators, in the
case of domains of $\RR^n$ covered by square lattices of step
$h=N^{-1}$. We provide here a statement  adapted to the torus
$\Sigma$ identified to quotients $\RR^2/\Gamma_N$ endowed with its spaces
of discrete functions.
\begin{theo}[Thom\'ee type theorem]
  \label{theo:thomee}
  There exists a constant $c_1>0$ such that for all $N\geq 0$ and for
  all functions $\psi\in C^2(\cQ_N(\Sigma))$, we have
  $$
 \| P_N \psi\|_{\cC^{0,\alpha}_w} + \|\psi\|_{\cC^0} \geq c_1\|\psi\|_{\cC^{2,\alpha}_w},
 $$
 where
 $$
P_N =\hat\Delta^E_N + \Gamma^E_N.
 $$
\end{theo}
\begin{proof}
Proposition~\ref{prop:formopl} can be readily used to prove an analogue of
Theorem~\ref{theo:limit} for the opertors $P_N$. In other words, for
every $k\geq 2$, for every sequence $\psi_{N_j}\in
C^2_+(\cQ_{N_j}(\Sigma))$ such that  (cf.~\eqref{eq:Deltaf} and
Lemma~\ref{lemma:varmu} as well)
\begin{equation}
  \label{eq:convPN}
  \psi_{N_j}\stackrel{\cC^{k}} \longrightarrow \phi \quad
  \Rightarrow \quad
  P_{N_j}\psi_{N_j}\stackrel{\cC^{k-2}}\longrightarrow
\Delta_{\ell}\phi.
\end{equation}

The operators $P_N$ admit canonical lifts $\tilde P_N:C^2_+(\cQ_N(\RR^2))
\to C^2_+(\cQ_N(\RR^2))$. The elliptic operator $\Delta_\ell$ can
also be lifted as an elliptic operator with smooth coefficients $\tilde\Delta_\ell$ acting on functions on the plane.
By Property~\eqref{eq:convPN}, the discrete operators $\tilde P_N:
C^2_+(\cQ_N(\RR^2))\to C^2_+(\cQ_N(\RR^2))$ converge toward the elliptic operator
  $\tilde \Delta_\ell$. This implies that the sequence of discrete
operators $\tilde P_N$ is consistent with the elliptic operator
$\tilde \Delta_\ell$ 
and that the operators $\tilde P_N$ must be  elliptic, for $N$ sufficiently large, in the sense of Thom\'ee~\cite{Thomee68}.

We consider a fundamental domain $\cD$ of the action of $\Gamma$ on
$\RR^2$. For $r_0>0$ sufficiently large, $\cD\subset B(0,r_0)$.
We define
$$\Omega_0 = B(0,r_0+1), \quad \Omega_1 = B(0,r_0+2) \mbox{ and }
\Omega_2=\RR^2,
$$
where $B(0,r)$ is an Euclidean ball of $\RR^2$ or radius $r$, centered
at the origin. By definition we have compact embeddings of the domains $\Omega_0 \Subset
\Omega_1 \Subset \Omega_2$.

The finite differences~\eqref{eq:fd1} and~\eqref{eq:fd2}  used to obtain the discrete finite
difference operators $\tilde P_N$ correspond to the finite differences 
 defined in \cite{Thomee68}, modulo a translation operator for
the retrograde differences. It follows that~\cite[Theorem 2.1]{Thomee68}
applies in our setting: there exists a constant $c>0$ such that for every $N$
sufficiently large and $\phi\in C^2_+(\RR^2)$,
\begin{equation}
  \label{eq:schau1}
\|\phi\|_{\cC_T^{2,\alpha}(\Omega_1)} \leq c \Big \{\| \tilde P_N
\phi\|_{\cC_T^{0,\alpha}(\Omega_2)} + \|\phi\|_{\cC_T^0(\Omega_2)} \Big
\}.
\end{equation}
\begin{rmk}
  \label{rmk:tnorm}
 In the above notations, the
$\cC^{k,\alpha}_T(\Omega_i)$-norm on $C^2_+(\cQ_N(\RR^ 2))$ are the norms defined
in~\cite{Thomee68}, using only forward differences. 
For $\Omega_2=\RR^2$, these norms co\"incide with
the  $\cC^{k,\alpha}$-norms introduced at~\S\ref{sec:dhn}.

If  $\Omega=B(0,R)$, the definition of the norms given at~\eqref{eq:C0}
and~\eqref{eq:C0alpha} has to be modified slightly for the
$\cC_T^{k,\alpha}(\Omega)$-norm. In order to describe what has to be
modified, assume for a moment that $\phi$ is a discrete
function defined only on the set of vertices of $\cG^+_N(\RR^2)$
contained in $\overline\Omega$. Notice that the finite differences
$\frac{\del\phi}{\del u}$ and $\frac{\del\phi}{\del u}$ are defined on
a smaller set, and the second order partial derivative on an even
smaller set, etc... The $\cC^{k,\alpha}_T(\Omega)$-norms are defined
similarly to the $\cC^{k,\alpha}$-norms, by taking the corresponding $\sup$ on a
smaller set of vertices. Namely, the vertices of $\cG^+_N(\RR^2)$ contained
in $\overline \Omega$ where the relevant partial derivatives are well
defined.
\end{rmk}

For $\psi\in C^2_+(\cQ_N(\Sigma))$, we define the lift
$\phi_N=\psi\circ p_N$.
By Remark~\ref{rmk:tnorm}, since $\Omega_2=\RR^2$, the RHS of~\eqref{eq:schau1} applied to $\phi_N$  is equal to
$$c \Big \{\| \tilde P_N
\phi_N\|_{\cC^{0,\alpha}} + \|\phi_N\|_{\cC^0} \Big
\}.$$
 By definition of the discrete H\"older norms, $\|\psi\|_{\cC^0}=
\|\phi_N\|_{\cC^0}$ and since $\tilde P_N\phi_N =
(P_N\psi)\circ p_N$, we have
$$
\| \tilde P_N
\phi_N\|_{\cC^{0,\alpha}} = \|  P_N
\psi\|_{\cC^{0,\alpha}}.
$$
Thus by~\eqref{eq:schau1}
\begin{equation}
  \label{eq:schau2}
\|\psi\circ p_N\|_{\cC_T^{2,\alpha}(\Omega_1)}\leq c \Big \{\|  P_N
\psi\|_{\cC^{0,\alpha}} + \|\psi\|_{\cC^0} \Big \} .
\end{equation}
We conclude using the following result
\begin{lemma}
 There exists a constant $c'>0$
  such that for every   $N$ sufficiently large and all $\psi\in C^2_+(\cQ_N(\Sigma))$
  $$
c'\|\psi\|_{\cC^{2,\alpha}} \leq \|\psi\circ p_N \|_{\cC_T^{2,\alpha} (\Omega_1)}.
  $$
\end{lemma}
\begin{proof}[Proof of the lemma]
  By definition, $\Omega_0$ contains a fundamental domain $\cD$ of
  $\Gamma$, and furthermore $\cD\Subset \Omega_0$. By construction, the lattices $\Gamma_N$ admit fundamental
  domains $\cD_N$ which converge (say in Hausdorff distance) toward
  $\cD$. Therefore $\cD_N \Subset \Omega_0$ for all $N$ sufficiently
  large.

 In particular every vertex $\zert\in\cG_N^+(\Sigma)$ admits a lift
 $\tilde \zert\in\cG_N^+(\RR^2)$ via $p_N$ such that $\zert\in
 \Omega_0$. This shows that
 $$
\|\psi\|_{\cC^0}\leq  \|\psi\circ p_N\|_{\cC^0_T(\Omega_1)}.
$$
If $N$ is sufficiently large, the finite differences of order $1$ or
$2$ of $\psi\circ p_N$ are well defined at $\tilde\zert$ depend only on values taken by
the function on the domain $\overline\Omega_1$. It follows by
Remark~\ref{rmk:tnorm} that
\begin{equation}
  \label{eq:controt1}
  \|\psi\circ p_N\|_{\cC^2}\leq  \|\psi\circ p_N\|_{\cC^2_T(\Omega_1)}.
\end{equation}
If $\xi$ is any discrete function in $C^0(\cG^+_N(\RR^2))$, for every
pair of vertices $\vert_0$, $\vert'$ of $\cG_N^+(\RR^2)$ with $\vert_0 \in
\overline \Omega_0$ and $\vert' \not \in
\overline \Omega_1$, we have
$$
\frac{|\xi(\vert_0)-\xi(\vert ')|}{\|\vert_0-\vert'\|^\alpha} \leq
{|\xi(\vert_0)-\xi(\vert ')|} \leq 2\|\xi\|_{\cC^0}
$$
since $ \|\vert_0-\vert'\|\geq 1$. We apply this inequality to the
second order finite differences of $\psi\circ p_N$. This shows that
the  $\cC^{2,\alpha}$-norm of $\psi\circ p_N$ is controled by its
$\cC^2$-norm  and its $\cC^{2,\alpha}_T(\Omega_1)$-norm. Hence by
\eqref{eq:controt1} the $\cC^{2,\alpha}_T(\Omega_1)$-norm controls the
$\cC^{2,\alpha}$-norm.
\end{proof}
Using the lemma and \eqref{eq:schau2}, we deduce that for every $N$ sufficiently large
and $\psi\in C^2_+(\cQ_N(\Sigma))$, we have
$$
c'\|\psi\|_{\cC^{2,\alpha}} = c'\|\psi\circ p_N\|_{\cC^{2,\alpha}}
\leq
 c \Big \{\| P_N
\psi\|_{\cC^{0,\alpha}} + \|\psi\|_{\cC^0} \Big
\}.
$$
 This proves the theorem for $N$
 sufficiently large and $\psi\in C^2_+(\cQ_N(\Sigma))$.
 The same result holds if $\psi \in C^2_-(\cQ_N(\Sigma))$. For a general
$\psi$, we use the decomposition in components $\psi=\psi^++\psi^-$
and the theorem follows, for $N$ sufficiently large, by definition of the weak H\"older norms.
If the theorem holds for $N$ sufficiently large, it holds for every
$N$ since $C^2(\cQ_N(\Sigma))$ is finite dimensional, and all norms
are equivalent.
\end{proof}
\begin{cor}
  \label{cor:schauder}
    There exists a constant $c_2>0$ such that for all $N\geq 0$ and for
  all functions $\phi  \in C^2(\cQ_N(\Sigma))$, we have
  $$
 \|\Delta_N \phi \|_{\cC^{0,\alpha}_w} + \|\phi \|_{\cC^0} \geq c_2\|\phi\|_{\cC^{2,\alpha}_w}.
 $$
\end{cor}
\begin{proof}
  We use the  decomposition $\phi=\phi^++\phi^-$ and prove the
  Corollary in the case of the operator
  $$ \Delta_N' = \hat\Delta^E_N +\Gamma^E_N + K^E_N + K^I_N
  $$ first.
  
Then  $\Delta'_N\phi = \psi = \psi^++\psi^-$, where $\psi^+ = P_N\phi^+ +
\kappa_N(\phi ^+ - \bar\phi^-)$ and $\psi^- = P_N\phi^- +
\kappa_N(\phi ^- - \bar\phi^+)$. 
Since $\kappa_N$ converges in the sense of Lemma~\ref{cor:kappaconv},
we deduce that $\|\kappa_N\|_{\cC^{0,\alpha}_w}$ is uniformly bounded
for all $N$.  Thus a $\cC^{0,\alpha}_w$-bound on $\phi$ provides
 a $\cC^{0,\alpha}_w$-bound on $\kappa_N\phi^\pm$. Similarly a
  $\cC^{0,\alpha}_w$-bound provides $\cC^{0,\alpha}_w$-bound on
  $\bar\phi^\pm$. In other words, there exists a constant $c'>0$
  independent of $N$ and $\phi$ such that
  $$
\|\kappa_N(\phi ^+ - \bar\phi^-)\|_{\cC_w^{0,\alpha}} \leq
c'\|\phi\|_{\cC^{0,\alpha}_w}  \mbox { and }
\|\kappa_N(\phi ^- - \bar\phi^+)\|_{\cC_w^{0,\alpha}} \leq
c'\|\phi\|_{\cC_w^{0,\alpha}} .
$$
It follows that
$$
\|\Delta'_N\phi\|_{\cC_w^{0,\alpha}} + 2c'\|\phi\|_{\cC_w^{0,\alpha}} \geq  \|P_N\phi\|_{\cC_w^{0,\alpha}} ,
$$
and by Theorem~\ref{theo:thomee}
\begin{equation}
  \label{eq:schtrick}
\|\Delta'_N\phi\|_{\cC_w^{0,\alpha}} + (2c'+1)\|\phi\|_{\cC_w^{0,\alpha}}
\geq  c\|\phi\|_{\cC_w^{2,\alpha}} .
\end{equation}
We are not quite finished since we have a $\cC_w^{0,\alpha}$-estimate
for $\phi$ in the above inequality rather than a $\cC^0$-estimate as in the corollary.
We prove a weaker version of the corollary first: we show that there
exists a constant $c''>0$ such that for all $N$ and for all $\phi$,
\begin{equation}
  \label{eq:schtrick2}
\|\Delta'_N\phi\|_{\cC_w^{0,\alpha}} + \|\phi\|_{\cC^{0}}
\geq  c''\|\phi\|_{\cC_w^{0,\alpha}} .
\end{equation}
If this is true, the corollary trivially follows in the case of $\Delta'_N$ from
\eqref{eq:schtrick} and \eqref{eq:schtrick2}. Finally,
Proposition~\ref{prop:errors} completes the proof in the case of
$\Delta_N =\Delta'_N+\cE^E_N+\cE^I_N$.

Assume that \eqref{eq:schtrick2} does not hold.
Then theres exists a sequence of
discrete functions $\phi_{N_k}\in\cC^2(\cQ_{N_k}(\Sigma))$ with the
property that
$$
\|\phi_{N_k}\|_{\cC_w^{0,\alpha}} = 1,\quad 
\|\Delta'_{N_k}\phi_{N_k}\|_{\cC_w^{0,\alpha}} \to 0  \mbox{ and }
\|\phi_{N_k}\|_{\cC_w^0} \to 0.
$$
Using Inequality~\eqref{eq:schtrick}, we obtain a uniform
$\cC^{2,\alpha}_w$-bound on $\phi_{N_k}$.
By the Ascoli-Arzela theorem~\ref{theo:ascoli}, we may assume up to
extraction of a subsequence,  that $\phi_{N_k}$ converges in the
$\cC^2_w$-sense toward a pair of functions $(\phi^+,\phi^-)$ on
$\Sigma$. Since the convergence is $\cC^2_w$ hence $\cC^0$, the
condition $\|\phi_{N_k}\|_{\cC^0}\to 0$ forces $\phi^+=\phi^-=0$.
This imply that $\phi_{N_k}\stackrel{\cC_w^2}\longrightarrow
(0,0)$, and in particular $\|\phi_{N_k}\|_{\cC_w^2}\to 0$. Since the
$\cC_w^2$-discrete norm controls the $\cC_w^{0,\alpha}$-discrete norm,
this contradicts the assumption $\|\phi_{N_k}\|_{\cC_w^{0,\alpha}}=1$.
\end{proof}

\subsection{Spectral gap}
\label{sec:specgap}
We define the discrete functions
$$
\I_N^\pm \in C^0(\cG_N^\pm(\Sigma)) \simeq   C^2_\pm(\cQ_N(\Sigma))
$$
by
$$
\ip{\I_N^\pm,\zert} = \left \{
\begin{array}{ll}
  1& \mbox{ if }  \zert\in\fC_0(\cG^\pm_N(\Sigma)) \\
  0& \mbox{ if }  \zert\in\fC_0(\cG^\mp_N(\Sigma))
\end{array}
\right .
$$
We also define the discrete functions $\I_N, \zeta_N  \in C^0(\cG_N(\Sigma))$ by
\begin{align*}
  \I_N & = \I^+_N +\I^ -_N\\
\zeta_N &=\theta_N^{-1}\cdot\I_N^+  - \theta_N^{-1}\cdot\I_N^-  
\end{align*}
where $\theta_N$ is any discrete function, sufficiently close to
$\theta^u_N$ or $\theta_N^v$. For instance, we put
$$
\theta_N =\frac 12( \theta^u_N + \theta^v_N).
$$
We define the spaces of discrete functions $\scrK_N \subset  C^0(\cG_N(\Sigma))$ by
\begin{equation}
  \label{eq:almostkernel}
\scrK_N = \left \{
\begin{array}{ll}
  \RR\cdot \I_N & \mbox{ if $(p,\ell)$ is nondegenerate,}  \\
  \RR\cdot \I_N \oplus \RR\cdot \zeta_N &
  \mbox{ in the degenerate case.}
\end{array}
\right .
\end{equation}
 In addition, we denote by
$$
\scrK_N^\perp \subset C^2(\cQ_N(\Sigma))
$$
the  orthogonal complement of $\scrK_N$, with respect to the $\ipp{\cdot,\cdot}$-inner product.

\begin{rmks}
  \begin{itemize}
\item  The function $\I_N$ and more generally, any constant function, is contained in the kernel of
  the operator $\Delta_N$. Indeed, $\Delta_N =\delta_N\delta_N^\star$, but
  $d^\star\I_N = 0$ by formula~\eqref{eq:dNstar}.

\item  The sequence of discrete functions $\zeta_N$ converges toward the
  pair of functions $(\theta^{-1},-\theta^{-1})$, at least in the
  $\cC^2_w$-sense. Whenever $\ell$ is degenerate, we must have
  $$
\Delta_N\zeta_N \stackrel{\cC^0_w}\longrightarrow (0,0),
$$
by Theorem~\ref{theo:limit} and Proposition~\ref{prop:kernel}.

\item The kernel of $\Delta_N$
  is at least $1$ dimensional. If $\ell$ is degenerate, our next result at Theorem~\ref{theo:specgap}, implies that
  for $N$ sufficiently large, $\ker\Delta_N$ has dimension at most
  $2$. Although $\zeta_N$ may not belong to  $\ker\Delta_N$, the previous
  remark shows that this function is approximately in the kernel. In
  this sense, $\scrK_N$ may be thought of as an approximate kernel of $\Delta_N$.
\end{itemize}
\end{rmks}

\begin{theo}
  \label{theo:specgap}
  There exists a real constant $c_3>0$ such that, for all positive integers $N$
  sufficiently large and for all discrete function $\phi\in
  \scrK_N^\perp$, we
  have
  $$
\|\Delta_N\phi\|_{\cC^{0,\alpha}_w} \geq c_3\|\phi\|_{\cC^{2,\alpha}_w}.
  $$
\end{theo}
\begin{proof}
  We are assuming that $\ell$ is degenerate. Since the proof in the
  nondegenerate case is completely similar, we leave the details to the reader.
  We start by proving a weaker version of the theorem:
  \begin{lemma}
    \label{lemma:specgap}
      There exists a real constant $c_4>0$ such that, for all positive integers $N$
  sufficiently large and for all discrete function $\phi\in
  \scrK_N^\perp$, we
  have
  $$
\|\Delta_N\phi\|_{\cC_w^{0,\alpha}} \geq c_4\|\phi\|_{\cC^{0}}.
  $$
  \end{lemma}
\begin{proof}[Proof of Lemma~\ref{lemma:specgap}]
  Assume that the the result is false. Then there exists a sequence
  $\phi_{N_k}\in\scrK_{N_k}^\perp$ such that
  $$
\forall k \quad \|\phi_{N_k}\|_{\cC^0}=1,  \mbox{ and }
\|\Delta_{N_k}\phi_{N_k}\|_{\cC_w^{0,\alpha}}\longrightarrow 0.
$$
 Using Corollary~\ref{cor:schauder}, we deduce a
$\cC^{2,\alpha}$-bound on $\phi_{N_k}$. Thanks to the Ascoli-Arzela
Theorem~\ref{theo:ascoli}, we may assume that $\phi_{N_k}$ converges 
in the weak $\cC^2$-sense, up to further extraction:
\begin{equation}
  \label{eq:lemmaconvphi}
  \phi_{N_k}\stackrel{\cC^2_w}\longrightarrow (\phi^+,\phi^-).
\end{equation}
By Theorem~\ref{theo:limit}, we conclude that
$$
\Delta_{N_k}\phi_{N_k}\stackrel{\cC^0}\longrightarrow \Xi(\phi^+,\phi^-).
$$
The condition $\|\Delta_{N_k}\phi_{N_k}\|_{\cC^{0,\alpha}}\to 0$
implies that $\|\Delta_{N_k}\phi_{N_k}\|_{\cC^{0}}\to 0$, which shows
that the limit is $(0,0)$. Therefore
\begin{equation}
  \label{eq:kerxi}
  (\phi^+,\phi^-)\in\ker\Xi.
\end{equation}
We are assuming now that we are in the degenerate case as before. The
nondegenerate case is treated similarly.
By assumption $\phi_{N_k}$ is orthogonal to $\scrK_N$, hence
$\ipp{\phi_{N_k},\I_{N_k}}=\ipp{\phi_{N_k},\zeta_{N_k}}=0$. Since all
these discrete functions converge in the $\cC^0$-sense, we deduce that
the limit also satisfy the orthogonality relation, that is
$$\ipp{(\phi^+,\phi^-), (\I,\I)} = \ipp{(\phi^+,\phi^-),
  (\theta^{-1},-\theta^{-1})}=0.$$
In other words $(\phi^+,\phi^-)$ is $L^2$-orthogonal to $\ker\Xi$. In
view of~\eqref{eq:kerxi} we deduce that
$$
\phi^+=\phi^-=0,
$$
and by~\eqref{eq:lemmaconvphi}, we deduce that
$$
\|\phi_{N_k}\|_{\cC^0}\longrightarrow 0
$$
which contradicts the assumption $\|\phi_{N_k}\|_{\cC^ 0}=1$.
This completes the proof of the lemma.
\end{proof}
By Lemma~\ref{lemma:specgap}, we have for every
$\phi\in\scrK_N^\perp$
$$
(1+c_4^{-1})\|\Delta_N\phi\|_{\cC_w^{0,\alpha}} \geq
\|\Delta_N\phi\|_{\cC_w^{0,\alpha}}+ \|\phi\|_{\cC^{0}}.
$$
By Corollary~\ref{cor:schauder}, the RHS is an upper bound for
$c_2\|\phi\|_{\cC_w^{2,\alpha}}$. The constant,
$$
c_3 =\frac{c_2}{1+c_4^{-1}}
$$
satisfies the theorem, which completes the proof.
\end{proof}
\begin{cor}
  \label{cor:greenprep}
  \begin{enumerate}
  \item   If $(p,\ell)$ is nondegenerate, then for every $N$ sufficiently large,
  the kernel of $\Delta_N$  is given by $\ker\Delta_N =\scrK_N=
  \RR\cdot \I_N$.
  Furthermore   $\scrK_N^\perp$ is
  preserved by $\Delta_N$ which induces an isomorphism $\Delta_N:\cK_N^\perp\to
  \cK_N^\perp$. 

  \item More generally, including the case where $(p,\ell)$ is degenerate,
    there is a direct sum decomposition for every $N$ sufficiently large
  $$
C^2(\cQ_N(\Sigma)= \scrK_N\oplus \Delta_N(\scrK_N ^\perp),
$$
and a constant $c_6$ independent of $N$, such that for all $\phi \in
C^2(\cQ_N(\Sigma))$ decomposed according to the above splitting as
$\phi=\bar\phi +\phi^ \Delta$, we have
\begin{equation}
  \label{eq:controlgreen}
c_6\|\phi\|_{\cC_w^{0,\alpha}}\geq \|\bar \phi\|_{\cC_w^{0,\alpha}} +
\|\phi_\Delta \|_{\cC_w^{0,\alpha}}.
\end{equation}
  \end{enumerate}
\end{cor}  
\begin{proof}
  The first statement is a consequence of the second statement: In the
  nondegenerate case, $\scrK_N$ is one dimensional and
  $\scrK_N\subset\ker\Delta_N$. Hence $\Delta(\scrK_N^\perp)$ has
  codimension at most $1$ in $C^2(\cQ_N(\Sigma))$. By the second
  statement the codimension is exactly $1$. Therefore $\ker\Delta_N
  =\scrK_N$. The rest of the statement follows using the fact that
  $\Delta_N$ is selfadjoint.

  We merely have to prove the second statement of the corollary.
  We start by proving that we have a splitting as claimed. Suppose
  that the intersection $\scrK_N\cap\Delta(\scrK_N^\perp)$ is not
  reduced to $0$ for arbitrarily large $N$. Then we may find a
  sequence $\phi_{N_k}$ contained in the intersections and such that
  $\|\phi_{N_k}\|_{\cC^{0}} = 1$.

  We notice that $\|\I_N\|_{\cC^{0}} =1$ and that
  $\|\zeta_N\|_{\cC^{0}}$ converges toward a positive constant, since
  $\zeta_N$ converges toward the pair of functions $(\theta^{-1},-\theta^{-1})$.
  Since $\phi_{N_k}\in \scrK_{N_k}$, we may write
  $$
\phi_{N_k}= a_k\I_{N_{k}} \mbox { resp. } \phi_{N_k}=
a_k\I_{N_{k}}+b_k\zeta_{N_k} \mbox{ in the degenerate case.}
  $$
  We deduce that the uniform $\cC^{0}$-bound on $\phi_{N_k}$
  provides a uniform bound on the coefficients $a_k$ and $b_k$.
  We may after extracting a suitable subsequence assume that the
  coefficients converge as $k$ goes to infinity. In particular
  $\phi_{N_k}$ converges toward an element of $\ker\Xi$, say in the
  $\cC^0$-sense.
  By construction we
  have a uniform $\cC^1_w$-bound on $\Phi_{N_k}$, which provides a
  uniform $\cC^{0,\alpha}_w$-bound.

  On the other hand $\phi_{N_k}\in \Delta_{N_k}(\scrK_N^\perp)$ so
  that there exists a sequence $\psi_{N_k}\in\scrK_{N_k}^\perp$ with
  $\phi_{N_k}=\Delta_{N_k}\psi_{N_k}$. By Theorem~\ref{theo:specgap},
  the uniform $\cC_w^{0,\alpha}$-bound on $\phi_{N_k}$ provides a
  uniform $\cC_w^{2,\alpha}$-bound on $\psi_{N_k}$. By Ascoli-Arzella
  Theorem~\ref{theo:ascoli}, we may assume that $\psi_{N_k}$ converges 
  in the $\cC^2_w$-sense toward a limit $(\psi^+,\psi^-)$ after
  extraction. It follows that $\phi_{N_k}=\Delta_{N_k}\psi_{N_k}$
  converges in the $\cC^0$-sense toward $\Xi(\psi^+,\psi^-)$.
In conclusion $\phi_{N_k}$ converges in the $\cC^0$-sense to an element
of $\Im\Xi\cap\ker\Xi =\{ 0 \}$.

In conclusion, the limit of $\phi_{N_k}$ must be the pair of functions
$(0,0)$, which contradicts the fact that $\| \phi_{N_k}
\|_{\cC^0}=1$. Thus
$$
\scrK_N \cap \Delta_N(\scrK_N^\perp) = \{0\}
$$
for all sufficiently large $N$. By Theorem~\ref{theo:specgap}, we know
that the restriction of $\Delta_N$ to $\scrK_N^\perp$ is injective
provided $N$ is large enough. For dimensional reasons, we have a
splitting 
$$
\scrK_N \oplus \Delta_N(\scrK_N^\perp) = C^2(\cQ_N(\Sigma)).
$$

We now proceed to the last part of the second statement. If the
control~\eqref{eq:controlgreen} does not hold, we find a sequence of
discrete functions 
$\phi_{N_k}\in C^2(\cQ_{N_k}(\Sigma))$ with decompositions
$$
\phi_{N_k}=\bar\phi_{N_k}+\phi^\Delta_{N_k},
$$
and the property that
$$
\|\phi_{N_k}\|_{\cC_w^{0,\alpha}}\to 0 \mbox{ and }
\|\bar\phi_{N_k}\|_{\cC_w^{0,\alpha}}+
\|\phi^\Delta_{N_k}\|_{\cC_w^{0,\alpha}} = 1.
$$
The $\cC^{0,\alpha}_w$-bound on $\bar\phi_{N_k}$ provides a uniform
$\cC^0$-bound. As in the first part of the proof, we may use this
bound to show that, up to extraction of a subsequence, $\bar\phi_{N_k}$
converges in the $\cC^1_w$-sense toward a limit $(\bar\phi^+,\bar\phi^-) \in \ker\Xi$.

Similarly, the $\cC_w^{0,\alpha}$ bound on $\phi^\Delta_{N_k}$ can be
used to show that, up to extraction of a subsequence, the sequence
converges in the $\cC^0$-sense toward a limit $(\phi^+_\Delta,\phi^-_\Delta)$
in the image of $\Xi$.

Eventually, we may assume that $\phi_{N_k}$ converges 
in the $\cC^0$-sense toward a limit $(\bar \phi^+ +\phi^+_\Delta,\bar
\phi^- +\phi^-_\Delta) \in \ker\Xi\oplus \Im \Xi$.
The fact that $\|\phi_{N_k}\|_{\cC_w^{0,\alpha}} \to 0$ implies that
$\|\phi_{N_k}\|_{\cC^0}\to 0$ and
by
uniqueness of the limit, we deduce that $\bar\phi^\pm=0$ and
$\phi_\Delta=0$.

However, $\bar\phi_{N_k}$ converges in the stronger, says, $\cC^1_w$-sense,
hence $\|\bar\phi_{N_k}\|_{\cC^{0,\alpha}_w}\to 0$. We
deduce that
$$\|\bar\phi_{N_k}\|=\|\phi_{N_k}-\phi^\Delta_{N_k}\|_{\cC^{0,\alpha}_w}\leq
\|\phi_{N_k}\|_{\cC^{0,\alpha}_w} +\| \phi^\Delta_{N_k}\|_{\cC^{0,\alpha}_w}\to 0,$$
which contradicts the assumption $\|\phi_{N_k}\|_{\cC_w^{0,\alpha}} + \|\phi^\Delta_{N_k}\|_{\cC_w^{0,\alpha}}=1$.
\end{proof}
\subsection{Modified construction in the degenerate case}
\label{sec:deg}
The situation for degenerate  pairs $(p,\ell)$
 came as a surprise to us. Our first guess was
that the operators $\Delta_N$ should converge in a reasonnable sense
toward
the operator involved in the smooth setting~\eqref{eq:Deltaf}. Consequently, we expected  the kernel of $\Delta_N$ to be one dimensional,
at least for $N$ large enough.
 The
 first clue that this was not true came from a local model: in this model,
 we do not choose $\Sigma$ to
be a torus, but a copy of $\RR^2$ embedded in $\RR^{2n}$ as an isotropic
Euclidean plane identified to $\RR^2$ with its quadrangulation $\cQ_N(\RR^2)$. Then one can check that the funtion $\I_N^+-\I_N^-$
belongs to the kernel of $\delta_N^\star$  directly from the formula~\eqref{eq:dNstar}.

The presence of a $2$-dimensional almost kernel $\scrK_N$ in the
degenerate case will create some trouble for solving our
probem. We may overcome them by changing slightly our construction. 

\subsubsection{The setup}
We start with a degenerate pair $(p_S,\ell_S)$, where
$$
\ell_S : S\longrightarrow \RR^4,
$$
is an isotropic immersion and  $S$ is a surface diffeomorphic to an oriented torus.
We carry out the constructions of quadrangulations $\cQ_N(S)$, graphs $\cG_N(S)$ exactly as 
in the case of $\Sigma$ (cf. \S\ref{sec:anal}), except one crucial detail. 
The lattice group $\Gamma(S)$ of the covering map $p_S:\RR^2\to S$
admits oriented basis
$(\gamma_1(S),\gamma_2(S))$. This is where
 comes the difference
with~\S\ref{sec:sqlat}: we choose a best approximation $\gamma_2^N(S)\in\Lambda_N^{ch}$
of $\gamma_2(S)$ and 
 $\gamma_1^N(S)\in \Lambda_N\setminus \Lambda_N^{ch}$ for
$\gamma_1(S)$. Notice that in the case of $\Sigma$, both
$\gamma_i^N(S)$ were chosen in $\Lambda^{ch}_N$.

This minor change
still allows us to construct  families of quadrangulations and
checkers graph. The only difference is that action of the lattice
$$
\Gamma_N(S)=span\{\gamma_1^N(S),\gamma_2^N(S)\}
$$
does not preserve the connected components of the decomposition
$$\cG_N(\RR^2)= \cG^+_N(\RR^2) \cup \cG^-_N(\RR^2),$$
and this splitting
does not descend as a splitting of $\cG_N(S)$. In particular discrete
functions $\phi\in C^2(\cQ_N(S))$ do not split into a positive and
negative component.  However, we may construct
the constant function $\I_N\in C^2(\cQ_N(S))$, which is the constant
$1$ on every face of the quadrangulation.

\subsubsection{The double cover}
We define
$$
\Gamma = span \{ \gamma_1,\gamma_2 \}
$$
where
$$
\gamma_1=2\gamma_1(S) \mbox{ and } \gamma_2=\gamma_2(S).
$$
The quotient $\Sigma = \RR^2/\Gamma$  comes with a covering map of
index $2$
\begin{equation}
  \label{eq:piS}
  \Phi^S:\Sigma\to S,
\end{equation}
and an action of $G\simeq \Gamma/\Gamma(S)\simeq \ZZ_2$ on $\Sigma$ by deck tranformations. 
We define accordingly
$$
\Gamma_N = span \{ \gamma_1^N,\gamma_2^N \}
$$
where
$$
\gamma_1^N=2\gamma^N_1(S) \mbox{ and } \gamma^N_2=\gamma^N_2(S).
$$
Notice that $2\Lambda_N\subset\Lambda_N^{ch}$, hence by definition
$\gamma_i^N\in\Lambda^{ch}_N$.
We also have double covers
$$
\Phi_N^S:\Sigma\to S
$$
with deck transformations $G_N\simeq \Gamma_N/\Gamma_N(S)\simeq \ZZ_2$
which come from the canonical projections
$\RR^2/\Gamma_N\to\RR^2/\Gamma_N(S)$.
In particular there are canonical embeddings of discrete functions
spaces induced by pullback 
$$
(\Phi^S_N)^*:C^2(\cQ_N(S))\longrightarrow C^2(\cQ_N(\Sigma)).
$$
The action of $G_N$ induces an action on $C^2(\cQ_N(\Sigma))$ and the
image of $(\Phi^S_N)^*$ consists of the discrete functions which are $G_N$-invariant.

\subsubsection{Meshes and operators for the modified construction}
Like for $\Sigma$, we may define the samples $\tau_N^S\in
\scrM_N(S)=C^0(\cQ_N(S))\otimes\RR^{2n}$ of the map $\ell_S:S\to\RR^{2n}$, the
inner product $\ipp{\cdot,\cdot}$ and the 
operators $\delta_N$, $\delta_N^\star$, etc... Using the canonical projections
$$
(\Phi^S_N)^*:C^0(\cQ_N(S))\otimes \RR^{2n} \longrightarrow C^0(\cQ_N(\Sigma))\otimes\RR^{2n}.
$$
we see that the pullbacks satisfy $\tau_N=(\Phi^S_N)^*\tau_N(S)$. In other words, they are also the
samples of the lifted isotropic immersion $\ell =\ell_S \circ \Phi^S:\Sigma \to \RR^{2n}$. Then
$\tau_N$ also induces operators denoted $\delta_N$, $\delta^\star_N$ and $\Delta_N$ which  commute with the pullback operation, by naturality of the construction.

\subsubsection{Spectral gap for the degenerate case}
All the norms defined on $C^2(\cQ_N(\Sigma))$ induce norms on
$C^2(\cQ_N(S))$ via the pullbacks $(\Phi_N^S)^*$, denoted in the same way. For instance,
for $\phi\in C^2(\cQ_N(S))$, we have
$$
\|\phi\|_{\cC^{2,\alpha}} = \|\phi\circ \Phi_N^S\|_{\cC^{2,\alpha}}.
$$
Then we prove the following result:
\begin{theo}
  \label{theo:specgap2}
  Let $\ell_S:S\to\RR^{2n}$ be an isotropic immersion of an oriented
  surface diffeomorphic to a torus with a conformal cover $p:\RR^2\to S$. There exists a constant $c_5>0$
  such that  for every $N$ sufficiently
  large and every $\phi\in C^2(\cQ_N(S))$ with $\ipp{\phi,\I_N}=0$, we
  have
    $$
\|\Delta_N\phi\|_{\cC^{0,\alpha}} \geq c_5\|\phi\|_{\cC^{2,\alpha}}.
  $$
\end{theo}
\begin{proof}
  We choose $N$ sufficiently large, so that the assumptions of Theorem~\ref{theo:specgap} are satisfied.
  Let $\phi\in C^2(\cQ_N(S))$ be a discrete function such that  $\ipp{\phi,\I_N}$. There may be some ambiguity in our notations, so we should emphasize that
  $(\Phi^S_N)^*\I_N$ is equal to $\I_N\in C^2(\cQ_N(\Sigma))$.
  
  We consider the pullback $\tilde \phi_N = \phi\circ\Phi_N^S$ of $\phi$ regarded as an element of $C^2(\cQ_N(\Sigma))$.
  By definition of inner products and pullbacks by $2$-fold covers, we have
  $$
 \ipp{\phi,\I_N} = \frac 12\ipp{\tilde \phi_N,\I_N} ,
 $$
hence, by assumption, $\ipp{\tilde \phi_N,\I_N}=0$.

 Notice that the action of $G_N=\ip{\Upsilon_N}\simeq\ZZ_2$ on
 $C^2(\cQ_N(\Sigma))$ respects the inner product $\ipp{\cdot,\cdot}$. Since $\gamma_1^N(S)\not\in \Lambda^{ch}_N$, we also have
 $$
\Upsilon_N\cdot \I_N^+ = \I_N^- \mbox { and conversely } \Upsilon_N\cdot \I_N^- = \I_N^+.
 $$
By construction the discrete function $\theta_N$ is $G_N$
invariant. Thus
$$
\Upsilon_N\cdot\zeta_N = \Upsilon_N\cdot(\theta_N^{-1}(\I^+_N -
\I^-_N)) = \theta_N^{-1}(-\I^+_N +  \I^-_N) = -\zeta_N.
$$
The above property implies that any $G_N$-invariant discrete function
is orthogonal to $\zeta_N$:
$$
 \ipp{\tilde\phi_N,\zeta_N}=\ipp{\Upsilon_N\cdot \tilde\phi_N,\Upsilon_N\cdot\zeta_N}=\ipp{\tilde\phi_N,-\zeta_N}
 $$
 therefore
 $$
\ipp{\tilde\phi_N,\zeta_N}=0.
$$
In conclusion $\tilde\phi_N$ is orthogonal to $\scrK_N$ and we may
apply Theorem~\ref{theo:specgap} to $\tilde\phi_N$, which proves the
theorem with $c_5=c_3$.
\end{proof}

\begin{rmk}
  Notice that Theorem~\ref{theo:specgap2} applies whether the pair $(p,\ell_S)$
is degenerate or
  nondegenerate. The applications are different from
  Theorem~\ref{theo:specgap}, in the sense that we are dealing with different
  type of quadrangulations and meshes.
  For instance the spaces of quadrangular meshes $\scrM_N$ admits a shear action whereas $\scrM_N(S)$ does not. Indeed the checkers graph associated to
  $\cQ_N(S)$ is connected whereas the checkers graph of  $\cQ_N(\Sigma)$ is not.
\end{rmk}

\section{Fixed point theorem}
\label{sec:fpt}
\subsection{Fixed point equation}
All the tools have been introduced in order to be able to apply the
contraction mapping principle.
We consider an isotropic immersion  $\ell:\Sigma\to\RR^{2n}$, its sequence of samples $\tau_N\in \scrM_N$ as before and the map
$$
F_N:C^2(\cQ_N(\Sigma))\longrightarrow C^2(\cQ_N(\Sigma))
$$
defined by
$$F_N(\phi) = \mu^r_N\left (\tau_N - J\delta_N^\star\phi \right ).$$
Solving the equation $F_N(\phi) = 0$  provides an isotropic
perturbation of the sample mesh $\tau_N$.
\begin{rmk}
  The perturbative approach introduced here is an analogue of
  Theorem~\ref{theo:perttoy} in the smooth setting.
  Indeed, let us denote by $f_N:\Sigma\to \RR^{2n}$ a smooth perturbation of
  $\ell:\Sigma\to\RR^{2n}$ and $h$ a smooth function on $\Sigma$. 
  The perturbation $\tau_N-J\delta_N^\star\phi$ is a discrete analogue
  of the smooth perturbation $K(h,f_N)=\exp_{f_N}( - i h)=
  f_N-\fJ\delta_{f_N}^\star h$. Thus the equation $F_N(\phi)=0$ is the discrete analogue of the equation
  $F(h,f_N)=0$ (cf. \eqref{eq:implicit}) in the smooth setting. 
\end{rmk}
The differential of the map is given by
$DF_N|_0 \cdot\phi= - D\mu^r_N|_{\tau_N}\circ J  \delta_N^\star (\phi) = \delta_N\delta_N^\star \phi$ hence
$$
DF_N|_0 \cdot\phi= \Delta_N\phi.
$$
As  pointed out in Lemma~\ref{lemma:quadratic}, the map $\mu_N^r$ is a quadratic. According to Definition~\ref{dfn:quadratic} and~\eqref{eq:etaN} one can write
$$
\Delta_N \phi = 2\Psi_N^r(\tau_N,\phi)
$$
and
$$
F_N(\phi) = \eta_N + \Delta_N \phi + \mu^r_N(J\delta_N^\star \phi),
$$
where $\eta_N=F_N(0)$ is the error term.
We introduce the space
$$\scrH_N = \Delta_N \left (\scrK_N^\perp \right ),$$
where $\scrK_N$ is the almost kernel of $\Delta_N$ defined at~\eqref{eq:almostkernel}.
By Corollary~\ref{cor:greenprep}, we have a direct sum decomposition
$$
C^2(\cQ_N(\Sigma))=\scrK_N\oplus \scrH_N,
$$
for every $N$ sufficiently large.

We define the Green  operator $\Green_N$ of $\Delta_N$ by
$$
\Green_N (\psi)= \left \{
\begin{array}{l}
  0 \mbox { if $\psi \in \scrK_N $}\\
  \phi \in \scrK_N^\perp \mbox{ with the property that
    $\Delta_N\phi =\psi$  if $\psi\in  \scrH_N$}
\end{array}
\right .
$$
The Green operator is bounded independently of $N$, which is a crucial
property for the application of the fixed point principle:
\begin{prop}
  \label{prop:green}
  There exists a constant $c_8>0$ such that for all positive integers
  $N$ and $\phi\in C^2(\cQ_N(\Sigma))$, we have
  $$
c_8 \|\phi\|_{\cC^{0,\alpha}_w} \geq \|\Green_N(\phi)\|_{\cC^{2,\alpha}_w}.
  $$
\end{prop}
\begin{proof}
  For every $N\geq N_0$ sufficiently large, we may use the decomposition
  $\phi=\bar\phi + \phi_\Delta$ and the fact that
  \begin{equation}
    \label{eq:propgreen}
    c_6\|\phi\|_{\cC^{0,\alpha}_w}\geq \|\bar \phi\|_{\cC^{0,\alpha}_w} +
\|\phi_\Delta \|_{\cC^{0,\alpha}_w},    
  \end{equation}
thanks to Corollary~\ref{cor:greenprep}.
By definition $\phi_\Delta = \Delta_N\psi$ for some $\phi\in
\scrK_N^\perp$, and $\Green_N(\phi)=\psi$. By
Theorem~\ref{theo:specgap} and~\eqref{eq:propgreen}, we deduce that
$$
c_6\|\phi\|_{\cC^{0,\alpha}_w} \geq
\|\phi_\Delta\|_{\cC^{0,\alpha}_w}\geq c_3\| \psi\|_{\cC^{2,\alpha}_w}.
$$
This proves the proposition.
\end{proof}

Notice that by definition, $\Green_N$ takes values in 
$\scrK_N^\perp$ and has kernel $\scrK_N$. If $\phi \in \scrK_N^\perp$, we have
$\Green_N\circ\Delta_N \phi=\phi$. Therefore
$$
\Green_N\circ F_N(\phi) = \phi + \Green_N \left (\eta_N
+\mu^r_N(J\delta_N^\star \phi) \right). 
$$
For $\phi\in \scrK_N^\perp$, the equation
$$F_N(\phi)\in 
\scrK_N
$$
is  equivalent to
$$
\phi = T_N(\phi)
$$
where
$$
T_N:\scrK_N^\perp\longrightarrow \scrK_N^\perp
$$
is the map defined by
$$
T_N(\phi) =- \Green_N (\eta_N +\mu^r_N(J\delta_N^\star \phi)).
$$
We merely need to apply the fixed point principle to the map $T_N$.

\subsection{Contracting map}
Notice that
\begin{align*}
T_N(\phi)-T_N(\phi')= &\Green_N\circ \mu_N^r(J\delta_N^\star \phi) -
\Green_N\circ \mu_N^r(J\delta_N^\star
\phi') \\
= & \frac 12\Green_N\circ \Psi_N^r
(J\delta_N^\star \phi'+ J\delta_N^\star \phi, J\delta_N^\star \phi'- J\delta_N^\star
\phi)
\end{align*}
hence
\begin{equation}
  \label{eq:contractmap}
  T_N(\phi)-T_N(\phi')=
   \frac 12\Green_N\circ \Psi_N^r
\left (J\delta_N^\star (\phi'+\phi), J\delta_N^\star (\phi'-  \phi)\right ).  
\end{equation}
\begin{prop}
  \label{prop:quadest}
  There exists a constant $c_7>0$ such that for every $\phi,\phi'\in
  C^2(\cQ_N(\Sigma))$, we have
  $$
  \|\Psi_N^r(J\delta_N^\star\phi,J\delta_N^\star \phi')\|_{\cC_w^{0,\alpha}}\leq c_7
  \|\phi\|_{\cC_w^{2,\alpha}}   \|\phi'\|_{\cC_w^{2,\alpha}}
  $$
\end{prop}
\begin{proof}
  Recall that, for $\tau\in\scrM_N$,  $\mu^r_N(\tau)\in C^2(\cQ_N(\Sigma))$ is given by the Formula
  $$
\ip{\mu^r_N(\tau), \face} = \frac 12\omega(\scrU_\tau(\face),\scrV_\tau(\face)).
$$
We deduce that
  $$
\ip{\Psi^r_N(\tau,\tau'), \face} = \frac 14\Big [
\omega(\scrU_\tau(\face),\scrV_{\tau'}(\face))
+ \omega(\scrU_{\tau'}(\face),\scrV_\tau(\face))
\Big ],
$$
and it follows that for some universal constant $c'_7>0$, we have
\begin{equation}
  \label{eq:trickholderpsi}
  \|\Psi^r_N(\tau,\tau')\|_{\cC_w^{0,\alpha}} \leq c'_7\Big [ \|\scrU_\tau
  \|_{\cC_w^{0,\alpha}}\|\scrV_{\tau'} \|_{\cC_w^{0,\alpha}}
  +  \|\scrV_\tau
  \|_{\cC_w^{0,\alpha}}\|\scrU_{\tau'} \|_{\cC_w^{0,\alpha}}\Big ].
\end{equation}
\begin{lemma}
  \label{lemma:trickholderpsi}
  There exists a universal constant $c''_7>0$  such that for all discrete
  function $\phi$ and  $\tau=J\delta_N^
  \star\phi$, we have
  $$
\|\scrU_\tau\|_{\cC_w^{0,\alpha}} \mbox{ and } \|\scrV_\tau\|_{\cC_w^{0,\alpha}} \leq
c''_7 \|\phi\|_{\cC_w^{2,\alpha}}.
  $$
\end{lemma}
\begin{proof}
  We carry out the proof in the case of $\scrU_\tau$, as the proof for
  $\scrV_\tau$ is almost identical.
  Using the index notations, we have
  $$
\ip{\scrU_\tau,\face_{kl}} = \frac 1N\Big (\ip{\delta_N^\star\phi,\vert_{k+1,l+1}}
- \ip{\delta_N^\star\phi,\vert_{k,l}} \Big ).
  $$
Using the expression of $\delta_N^\star$, we obtain
\begin{align*}
\ip{\scrU_\tau,\face_{kl}} =& \frac 1{N^2}\Big ( 
(\phi_{k+1,l}D^u_{k+1,l} -\phi_{k,l-1}D^u_{k,l-1}) -
(\phi_{k,l+1}D^u_{k,l+1} -\phi_{k-1,l}D^u_{k-1,l}) \\
& + (\phi_{k+1,l+1}D^v_{k+1,l+1} - 2\phi_{kl}D^v_{kl} + \phi_{k-1,l-1}D^v_{k-1,l-1})
 \Big ).  
\end{align*}
The first line in the above computation is related to  the second order
finite difference $\frac{\del^2}{\del u\del v}$ of $\phi$ whereas the
second line is related to  the finite difference $\frac{\del^2}{\del
  u^2}$ of $\phi$. The fact the the renormalized diagonals  converge
smoothly allows to control the $\cC_w^{0,\alpha}$-norms of these terms
using the $\cC_w^{2,\alpha}$-norm of $\phi$.
\end{proof}
Inequality~\eqref{eq:trickholderpsi} together with
Lemma~\ref{lemma:trickholderpsi} completes the proof of the proposition.
\end{proof}

\begin{cor}
  \label{cor:contract}
For all $\epsilon>0$ there exists $N_0\geq 1$ and $\delta >0$ such
that for all $N\geq N_0$,  $\phi, \phi' \in \scrK_N^\perp$, such that
$\|\phi\|_{\cC^{2,\alpha}}\leq \delta$  and
$\|\phi'\|_{\cC^{2,\alpha}}\leq \delta$, we have
$$
\|T_N(\phi)-T_N(\phi')\|_{\cC_w^{2,\alpha}}\leq \epsilon \|\phi-\phi'\|_{\cC_w^{2,\alpha}}.
$$
\end{cor}
\begin{proof}
  By~\eqref{eq:contractmap}, Proposition~\ref{prop:green} and Proposition~\ref{prop:quadest} 
  \begin{align*}
  \| T_N(\phi)-T_N(\phi')\|_{\cC_w^{2,\alpha}} = &
   \frac 12\|\Green_N\Psi_N^r
\left (J\delta_N^\star (\phi'+\phi), J\delta_N^\star (\phi'-  \phi)\right
)\|_{\cC_w^{2,\alpha}} \\
\leq & \frac{c_7c_8}2 \|\phi+\phi'\|_{\cC_w^{2,\alpha}}\|\phi-\phi'\|_{\cC_w^{2,\alpha}}.    
  \end{align*}
In conclusion we may choose $\delta=\frac\epsilon{c_7c_8}$, which proves the corollary.
\end{proof}
\subsection{Fixed point principle}
The idea, as usual is to check whether the sequence $T_N^k(0)$
converges. If so, the limit must be a fixed point of $T_N$.
We have the following classical proposition
\begin{prop}
  Let $(\EE,\|\cdot\|)$ be a finite dimensional (or Banach) normed vector space and
  $T:\EE\to\EE$ an application such that
  \begin{enumerate}
    \item There exists $\delta >0$ such that the restriction of $T$ to
      the closed ball $\bar B_\delta$ of $\EE$,  centered at $0$ with
      radius $\delta$, is $\frac 
      12$-contractant, i.e.
      $$
\forall x,y\in \EE, \|x\|\leq \delta \mbox{ and } \|y\|\leq\delta
\Rightarrow \|T(x)-T(y)\|\leq \frac 12\|x-y\|.
      $$
  \item $\|T(0)\|\leq \frac \delta 2$
  \end{enumerate}
  Then the sequence $(t_k)$  defined by $t_0=0$ and $t_{k+1}=T(t_k)$
  converges to an element $t_\infty\in\EE$ with $\|t_\infty\|\leq
  \delta$. Furthermore, $t_\infty$ is a fixed point for $T$. Such fixed
  point are unique in the ball $\bar B_\delta$. In addition, we have
  $\|t_\infty\|\leq 2\|T(0)\|$.
\end{prop}
\begin{proof}
  The uniqueness of fixed points is a trivial consequence of
  the  contracting property of $T$ in the ball $\bar B_\delta = \{x\in
  \EE\|x\|\leq \delta\}$.

  For the convergence, we show first by inductions that $t_k$ remains in $\bar
  B_\delta$ for all $k$ : this is the case for $t_0=1$ and $t_1$ by
  assumption. Assume now that if $t_0,\cdots,t_{k-1}\in \bar
  B_\delta$.
  Then
  $$
\|t_k-t_{k-1}\| =\|T(t_{k-1}) -T(t_{k-2})\|\leq \frac 12\|t_{k-1}-t_{k-2}\|
  $$
 and by induction 
  $$
\|t_k-t_{k-1}\| = \frac 1{2^{p}} \|t_{k-p}-t_{k-p-1}\|.
  $$
  In particular
  $$
\|t_k-t_{k-1}\| = \frac 1{2^{k-1}} \|t_{1}-t_{0}\|=  \frac 1{2^{k-1}}\|T(0)\|.
$$
In turn we have
$$
t_k= t_k-t_0=(t_{k}- t_{k-1})+(t_{k-1}-t_{k-2})+\cdots +(t_1-t_0)
$$
and by the triangle inequality, 
$$
\|t_k\| \leq \|T(0)\|\sum_{j=0}^{k-1}\frac 1{2^k} =
\|T(0)\|\frac{1-\frac 1{2^j}}{1-\frac 12} \leq 2\|T(0)\|\leq \delta,
$$
so that $t_k\in \bar B_\delta$. This completes the induction and shows
that $t_k$ remains in $\bar B_\delta$.

Eventually, we just have to prove that $t_k$ converges. But this is
clear since
$$
t_{k+p}-t_k = (t_{k+p} - t_{k+p-1}) + (t_{k+p-2}- t_{k+p-2})+\cdots + (t_{k+1}-t_k)
$$
and by the triangle inequality
$$
\|t_{k+p}-t_k\| \leq  \|t_{k+1}-t_k\|\sum_{j=0}^\infty \frac 1{2^j}
\leq 2\|t_{k+1}-t_k\| \leq \frac 2{2^{k-1}}\|T(0)\|
$$
which shows that $t_k$ is Cauchy hence convergent in the closed ball
$\bar B_\delta$. The fact that the limit of $t_k$ is a fixed point of
$T$ is clear from the definition of the sequence, by uniqueness of the limit.
\end{proof}

We obtain the following result
\begin{theo}
  \label{theo:fpt}
  There exists a positive integer $N_0$ and a real number $\delta>0$
  such that for all $N\geq N_0$ there exists a unique $\phi_N\in
  \scrK_N^\perp$ that satisfies
  $$
\|\phi_N\|_{\cC_w^{2,\alpha}}\leq
  \delta  \mbox{ and } F_N(\phi_N)\in\scrK_N.
  $$
  Furthermore the sequence satisfies $\|\phi_N\|_{\cC_w^{2,\alpha}}=\cO(N^{-1})$.
\end{theo}

\begin{prop}
\label{prop:isoquad}
  Let $\ell:\Sigma \to \RR^{2n}$ be an isotropic immersion and
  $p:\RR^2\to\Sigma$ a conformal cover introduced before,  such that
  the pair $(p,\ell)$  is  nondegenerate. 
  Then the meshes
  $$
\rho_N=\tau_N-J\delta_N^\star\phi_N \in \scrM_N
$$
where $\phi_N$ is defined by Theorem~\ref{theo:fpt} for every $N\geq
N_0$ are isotropic.
\end{prop}
\begin{proof}[Proof of Proposition~\ref{prop:isoquad}]
  By definition $\mu_N^r(\rho_N)\in\cK_N$. By nondegeneracy,
  $\cK_N=\RR\I_N$, so that $\mu_N^r(\rho_N) = \lambda\I_N$ for some
  constant $\lambda$. We deduce that
  $$
\ipp{\mu_N^r(\rho_N),\I_N} = \lambda\ipp{\I_N,\I_N}.
$$
This quantity does not vanish unless $\lambda =0$. But
$\ipp{\mu_N^r(\rho_N),\I_N}$ is the total symplectic area of the mesh
$\rho_N$, which has to vanish by Stokes theorem, since the symplectic
form of $\RR^4$ is exact. In conclusion $\lambda=0$ so that~$\mu_N^r(\rho_N)=0$.
\end{proof} 
\subsection{Proof of Theorem~\ref{theo:mainquad}}
We merely need to gather the previous technical results so that the
proof and our main result follows as a corollary.
\begin{proof}[Proof of Theorem~\ref{theo:mainquad}]
  Let $\ell:\Sigma\to\RR^{2n}$ we a smooth isotropic immersion. By
  Proposition~\ref{prop:avoid}, we may always assume that the
  conformal cover $p:\RR^2\to\Sigma$ is chosen in such a way that the
  pair $(p,\ell)$ is non degenerate.
  By Proposition~\ref{prop:isoquad}, the quadrangular meshes $\rho_N$
  provided by Theorem~\ref{theo:fpt}, for $N$ sufficiently large,
  are isotropic. The estimate
  $\|\phi_N\|_{\cC^{2,\alpha}}=\cO(N^{-1})$ implies that
  $$
  \sup_{\vert\in\fC_0(\cQ_N(\Sigma))}\|\rho_N(\vert)- \tau_N(\vert)\|
  = \cO(N^{-1}).
  $$
  It follows that
  $$
  \sup_{\vert\in\fC_0(\cQ_N(\Sigma))}\|\rho_N(\vert)- \ell(\vert)\|
  = \cO(N^{-1}),
  $$
  which proves the theorem.
\end{proof}
\subsection{The degenerate case}
If $(p,\ell)$ is a degenerate pair, Theorem~\ref{theo:fpt} still provides a family of quadrangular meshes $\rho_N$ with the property that $\rho_N\in\scrK_N$. However
$\rho_N$ may not be an isotropic mesh since $\scrK_N$ may not reduce to constants.
This difficulty can be taken care of by working $\ZZ_2$-equivariantly.
Given a degenerate pair $(p_S,\ell_S)$,
we construct modified quadrangulations $\cQ_N(S)$ as in \S\ref{sec:deg}.
Using the notation introduced at \S\ref{sec:deg}, we consider the lifted pair $(p,\ell)$ given by $p=p_S\circ\Phi^S$ and $\ell=\ell_S\circ\Phi^S$, where $\Phi^S:\Sigma\to S$ is a double cover introduced at~\eqref{eq:piS}.
The pair $(p,\ell)$ is  degenerate as well.
Using Theorem~\ref{theo:fpt}, we find a corresponding family of quadrangular meshes $\rho_N$. All these construction are $G_N$-equivariant. In particular $\rho_N \in \scrK_N$ is also $G_N$-invariant. We have
$$
\rho_N = a_N\I_N +b_N\zeta_N,
$$
where $\rho_N$ and $\I_N$ are $G_N$-invariant. However $\zeta_N$ is $G_N$-anti-invariant (cf. proof of Theorem~\ref{theo:specgap2}), which implies that $b_N=0$. We conclude that $a_N=0$ as in the proof of Proposition~\ref{prop:isoquad}.

In conclusion $\rho_N$ descends to the $G_N$-quotient as an isotropic quadrangular mesh $\rho_N^S\in\scrM_N(S)$ and we have proved the following result
\begin{prop}
  \label{prop:degcase}
  Let $(p_S,\ell_S)$ be any pair, where $\ell_S:S\to\RR^{2n}$ is an isotropic immersion and $p_S:\RR^2\to S$ an associated conformal cover.
  Let $\tau_N^S\in \scrM_N(S)$ be the family of samples of $\ell_S$. Then, there exists a family of isotropic quadrangular meshes $\rho_N^S\in\scrM_N(S)$ such that
  $$
\max_{\vert}\|\rho^S_N(\vert)-\tau_N^S(\vert)\| =\cO(N^{-1}).
  $$
\end{prop}
\begin{rmk}
  The approach presented in Proposition~\ref{prop:degcase} appears as
  a good solution to treat our perturbation problem in a uniforma
  manner, whether or not the pair $(p,\ell)$ is degenerate.
  The main flaw of such technique, relying on $\ZZ_2$-equivariant
  constructions, is that the moduli spaces $\scrM_N(S)$ do not admit a
  shear action as defined in~\S\ref{sec:shear} (this is due to the
  connectedness of the checkers graph of $\cQ_N(S)$).
  Unfortunately, the shear
  action is used in a crucial way at~\S\ref{sec:quadtri} to obtain
  generic quadrangular meshes that will allow to construct piecewise
  linear immersions as in Theorem~\ref{theo:maindiscr}.
\end{rmk}
\section{From quadrangulations to triangulations}
\label{sec:quadtri}
The previous section was devoted to the construction of isotropic
meshes associated to quadrangulations, sufficiently close to a given smooth isotropic immersion
$\ell:\Sigma \to \RR^ {2n}$. In this section, we explain how to define a
nearby isotropic piecewise linear map as an approximation of $\ell$.
The idea is to pass from an isotropic quadrangulation to an isotropic triangulation.

\subsection{From  quadrilaterals to  pyramids}
\label{sec:pyramid}
The goal of this section is to explain how to pass from an isotropic quadrilateral to an isotropic pyramid, by adding
one apex to the quadrilateral.
We start by studying a single isotropic quadrilateral
$(A_0,A_1,A_2,A_3)$, where $A_i$ are points in $\RR^{2n}$.
We shall use the notations
\begin{equation}
  \label{eq:diag}
  D_0=\overrightarrow {A_0A_2} \mbox{ and } D_1=\overrightarrow {A_1A_2}
\end{equation}
for the two diagonals of the quadrilateral. Recall that the
quadrilateral is isotropic if, and only if
$$
\omega(D_0,D_1)=0.
$$
\begin{rmk}
If the diagonals of an isotropic quadrilateral are linearly
independent vectors of $\RR^{2n}$, this implies that
$$L=\RR D_0\oplus \RR
D_1$$
is an isotropic  plane of $\RR^{2n}$.  
\end{rmk}

\begin{dfn}
 \label{dfn:pyramid}
A  \emph{pyramid} is given by five points $(P,A_0,A_1,A_2,A_3)$ of $\RR^{2n}$. The four points of quadrilateral
$(A_0,\cdots, A_3)$, called the \emph{base of the pyramid} and the \emph{apex} $P\in \RR^{2n}$.
If the four triangles given by
$(PA_iA_{i+1})$, where~$i$ is understood as an index modulo~$4$, are contained in
isotropic planes of~$\RR^{2n}$, we say the the pyramid is an
\emph{isotropic pyramid} (cf. Figure~\ref{figure:pyramid}).  
\end{dfn}

\begin{figure}[H]
  \begin{pspicture}[showgrid=false](-2,-1)(2,1)
    \psline(-1,-1)(0,-1)(.5,0)(-0.5,0)(-1,-1)
    \psline(.5,1)(-1,-1)(0,-1)(.5,0)(-0.5,0)
    \psline(.5,1)(-1,-1)
    \psline(.5,1)(0,-1)
    \psline(.5,1)(.5,0)
    \psline(.5,1)(-0.5,0)

    \rput(-1.5,-1){$A_0$}
    \rput(0.5,-1){$A_1$}
    \rput(1.0,0){$A_2$}
    \rput(-1.0,0){$A_3$}

    \rput(0.1,1.1){$P$}

   \psset{fillstyle=solid,fillcolor=black}
   \pscircle(-1,-1){.1}
   \pscircle(0.5,0){.1}
   \pscircle(-0.5,0){.1}
   \pscircle(0,-1){.1}
   \pscircle(0.5,1){.1}
   
  \end{pspicture}
  \caption{Pyramid with apex $P$ and base $(A_0,A_1,A_2,A_3)$}
  \label{figure:pyramid}
\end{figure}
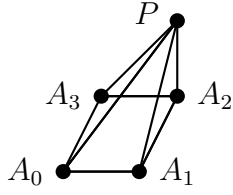

The following Lemma shows a first relation between isotropic quadrilaterals and isotropic pyramids:
\begin{lemma}
  The base of an isotropic  pyramid is an isotropic quadrilateral.
\end{lemma}
\begin{proof}
  The result is obtained as a trivial consequence of the Stockes
  theorem, or by elementary algebraic manipulations.
\end{proof}
Conversely, we have the following result:
\begin{lemma}
\label{lemma:pyramid}
Let  $Q=(A_0,\cdots,A_3)$ be an isotropic quadrilateral of $\RR^{2n}$
with linearly independent diagonals. We denote by $W'_Q$ be the
symplectic orthogonal of the vector space  spanned by the sides of the
quadrilateral $Q$. Let  $W_Q$ be the set of points $P\in \RR^{2n}$
which are the apexes of  isotropic pyramids with base 
given by the quadrilateral $Q$. Then $W_Q$ 
 is an affine subspace of $\RR^{2n}$  with underlying
 vector space $W'_Q$. Its dimension is $2n-2$ if $Q$ is
 flat and  $2n-3$ otherwise.
  \end{lemma}
\begin{proof}
We are looking for a solution of
  the linear system of four equations
  $$
\omega(\overrightarrow{PA_i},\overrightarrow{A_iA_{i+1}})= 0,
$$
where  $0\leq i \leq 3$.
Put
\begin{equation}
  \label{eq:gp}
  X=\overrightarrow{GP}
\end{equation}
where $G$ is by convention the barycenter of the quadrilateral.
The system can be expressed as
  $$
\omega(X,\overrightarrow{A_iA_{i+1}})= \omega(\overrightarrow{GA_i},\overrightarrow{A_iA_{i+1}}).
$$
The LHS correspond to a linear map with kernel~$W'_Q$.

If the quadrilateral is flat, it is contained in an isotropic affine
plane parallel to $L=\RR D_0\oplus \RR D_1$. Any point $P$ in the
plane of the quadrilateral is the apex of an isotropic
pyramid. Furthermore,
the space of solutions is an affine space of codimension $2$.

If the quadrilateral is not flat, then $\dim W'_Q=2n-3$ and the LHS of the linear system has rank 3.
The condition that the quadrilateral is isotropic is
precisely the compatibility condition, that insures that the RHS of
the equations is in the image of the Linear map. We conclude that the
system of equations admits a $2n-3$-dimensional affine space of
solutions. 
\end{proof}

Lemma~\ref{lemma:pyramid} is a excellent tool for passing from
isotropic meshes associated to  quadrangulations to isotropic
meshes associated to  triangulations and, in turns, to piecewise
 linear isotropic maps. One issue, that has to dealt with, is how
$C^0$-estimates are preserved and also, whether the piecewise linear map induced by this construction are still immersions. Indeed, Lemma~\ref{lemma:pyramid} does
provide any information  about the distance from $W_Q$ to the
quadrilateral.

\subsubsection{Optimal apex}
\label{sec:optiapex}
  There exists large families of isotropic pyramids as shown by Lemma~\ref{lemma:pyramid}.
  In this section we introduce some particular solutions of the corresponding linear system, called \emph{optimal pyramids} and \emph{optimal apex}.
  
We use the notations introduced in the proof of
Lemma~\ref{lemma:pyramid}. Again, we consider an isotropic quadrilateral
$Q=(A_0,\cdots,A_3)$. We are assuming that $Q$ has linearly independent diagonals $D_0, D_1$.
Hence the diagonals span an isotropic  plane $L=\RR D_0\oplus \RR D_1$. We may consider its
complexification
\begin{equation}
  \label{eq:cxquad}
L^\CC = L\oplus JL,  
\end{equation}
and the corresponding orthogonal complex (and symplectic) spliting
$$
\RR^{2n}=L^\CC \oplus M.
$$
Notice that the real dimension of $L^\CC$ is $4$.

We are looking for a point  $P\in\RR^ {2n}$ solution of the linear system
\begin{equation}
  \label{eq:sys1}
    \omega(\overrightarrow{GP},\overrightarrow {A_iA_{i+1}})= \gamma_i
     \quad {0\leq i \leq 3}
\end{equation}
where
\begin{equation}
  \gamma_i=\omega(\overrightarrow{GA_i},\overrightarrow {GA_{i+1}}).
\end{equation}
  According to Lemma~\ref{lemma:pyramid}, the affine  space of
  solutions $W'_Q$ has dimension  $2n-2$ or  $2n-3$ in $\RR^{2n}$
  depending on  the flatness of the quadrilateral. We may reduce to particular solutions by adding the constraint
  \begin{equation}
    \label{eq:sys2}
    \overrightarrow{PG}\in L^\CC .
  \end{equation}
We  use the notation $X=\overrightarrow{GP}$. A quadrilateral $(A_0,\cdots,A_3)$ is determined
by specifying its barycenter $G$, the side vector
$V=\overrightarrow{A_0A_1}$ and the diagonals $D_0$ and $D_1$.
We first compute the terms $\gamma_i$ of the RHS in terms of these quantities.
By definition
$$
\begin{array}{llll}
 4 \overrightarrow{A_0G} &= \overrightarrow{A_0A_1}+
 \overrightarrow{A_0A_2}+ \overrightarrow{A_0A_3} &= & D_0 + D_1 +2V \\
 4 \overrightarrow{A_1G} &= \overrightarrow{A_1A_0}+
 \overrightarrow{A_1A_2}+ \overrightarrow{A_1A_3} &= & D_0 +D_1 -2V \\
 4 \overrightarrow{A_2G} &= \overrightarrow{A_2A_0}+
 \overrightarrow{A_2A_1}+ \overrightarrow{A_2A_3}  &= &-3  D_0+D_1+2V\\
 4 \overrightarrow{A_3G} &= \overrightarrow{A_3A_0}+
 \overrightarrow{A_3A_1}+ \overrightarrow{A_3A_2} & = & D_0-3D_1 -2V
\end{array}
$$
Hence
\begin{equation}
  \label{eq:gammai}
\begin{array}{rll}
  16\gamma_0 &=&\omega(D_0+D_1+2V,D_0+D_1-2V) =4\omega(V,D_0+D_1)\\
  16\gamma_1 &=& \omega(D_0+D_1-2V,-3D_0+D_1+2V) =4\omega(V,D_0-D_1)\\
  16\gamma_2 & =& \omega(-3D_0+D_1+2V,D_0-3D_1-2V) =-4\omega(V,D_0+D_1)\\
  16\gamma_3 & =& \omega(D_0-3D_1-2V,D_0+D_1+2V) =-4\omega(V,D_0-D_1)
\end{array}  
\end{equation}
Let $D'_0$ and $D'_1$ be the basis of $L$ defined by the orthogonality
conditions
$$
\forall i,j\in\{0,1\},\quad g(D_i,D'_j)= \delta_{ij}
$$
and put
\begin{equation}
  \label{eq:Bi}
  B_i=JD'_i, \quad i\in\{0,1\}
\end{equation}
which are a basis for $JL$. Notice that by definition
$$
\omega(D_i,B_j)=g(D_i,D_j)=\delta_{ij}.
$$
We may express the vectors $X$ and $V$ using the basis $(D_0,D_1,B_0,B_1)$ of $L^\CC$~as
\begin{align}
  \label{eq:decX}
  X &=a_0D_0 +a_1D_1 + b_0B_0+b_1B_1 \\
  \label{eq:decV}
V &=\alpha_0 D_0 +\alpha_1D_1 + \beta_0B_0+\beta_1B_1 +V_M,
\end{align}
where $\alpha_i,\beta_i,a_i,b_i\in\RR$ and $V_M\in M$.
By~\eqref{eq:gammai} \eqref{eq:decX} and  \eqref{eq:decV}, we have
\begin{align*}
  4\gamma_0 &= \beta_0+\beta_1 \\
  4  \gamma_1 & = \beta_0-\beta_1 \\
  4  \gamma_2 &= - \beta_0 -\beta_1  \\
      4  \gamma_3 &= -\beta_0+\beta_1 
\end{align*}
The linear system~\eqref{eq:sys1} with constraint~\eqref{eq:sys2}
is equivalent to (after adding up lines)
\begin{equation}
  \label{eq:sys3}
  \left \{
\begin{array}{lcl}
  \omega(X,V) &=& \gamma_0 \\
  \omega(X,D_0)&=&\gamma_0+\gamma_1\\
  \omega(X,D_1) &=& \gamma_1+\gamma_2 
\end{array}
\right .  
\end{equation}
where we have removed the last redundent equation.
Eventually, a solution of~\eqref{eq:sys3} is given by
\begin{equation}
  \label{eq:sys4}
  \left \{
\begin{array}{rcl}
  a_0 \beta_0 + a_1 \beta_1 - b_0\alpha_0  - b_1\alpha_1  &=& \frac
  14(\beta_0+\beta_1) \\
  b_0 &=&-\frac 12 \beta_0\\
  b_1 &=& \frac 12 \beta_1
\end{array}
\right .  
\end{equation}
i.e.  the solutions $X$ are given by
\begin{equation}
  \label{eq:sys6}
  X = a_0D_0+a_1D_1  - \frac {\beta_0}2B_0 +\frac{\beta_1}2B_1
\end{equation}
where $a_0$ and $a_1$ satisfy the affine equation
\begin{equation}
  \label{eq:sys5}
  \beta_0 a_0 +\beta_1 a_1   = \frac  14(\beta_0+\beta_1) + \frac {\beta_1\alpha_1-\beta_0\alpha_0}2.
\end{equation}
In conclusion, any solution $(a_0,a_1)$ of the affine equation
\begin{equation}
  \label{eq:sys7}
  \beta_0 a_0 +\beta_1 a_1   = \xi(V)
\end{equation}
where
$$
\xi(V):=\frac {\beta_0(1-2\alpha_0)+\beta_1(1+2\alpha_1)}4
$$
provides a solution to our linear system.
If the orthogonal projection of $V$ onto $JL$ does not vanish, we have
$(\beta_0,\beta_1)\neq (0,0)$ and the above equation defines a
line, which, in turn defines a one dimensional space of solutions
$X\in L^\CC$.
We summarize our computations in the following lemma:
\begin{lemma}
  Assume that $Q=(A_0,\cdots,A_3)$ is an isotropic quadrilateral with
  linearly independent diagonals in $\RR^{2n}$. Assume that the
  orthogonal projection of $Q$ in $L^\CC$ is not a flat quadrilateral.
  Then set of points $P\in G+L^\CC$ which are the apex of an isotropic
  pyramid over $Q$, form a
  $1$-dimensional affine space. 
\end{lemma}
Under the assumptions of the lemma, we may consider  a particular solution given by
\begin{equation}
  \label{eq:optimal}
X= \xi(V) \frac{\beta_0}{\beta_0^2 + \beta_1^2} D_0 + \xi(V)
\frac{\beta_1}{\beta_0^2 + \beta_1^2}D_1 - \frac{\beta_0}2B_0 +
\frac{\beta_1}2B_1.
\end{equation}
 The above solution corresponds to the apex $P$, which is the closest point to the barycenter $G$, with the
 property that the corresponding pyramid is isotropic and $\overrightarrow {GP}\in L^\CC$.
 This leads us to the following definition:
 \begin{dfn}
  Let $(A_0,\cdots,A_3)$ be an isotropic quadrilateral of $\RR^{2n}$ with linearly independent diagonals
  and $G$ be its barycenter. The closest point $P$ to $G$ in $G+L^\CC$
  such that $(P,A_0,\cdots,A_3)$ isotropic is called \emph{the optimal apex}, and the
  corresponding pyramid an \emph{optimal isotropic pyramid}.
 \end{dfn}
 \begin{rmk}
   If the orthogonal projection of the quadrilateral in $L^\CC$ is flat, then the optimal apex is just the barycenter $G$ of the quadrilateral. If it is not flat the optimal apex is given by $X=\overrightarrow{GP}$, where $X$ is given  by Formula~\eqref{eq:optimal}.
 \end{rmk}

Optimal pyramids enjoy nice properties. We first point out that they are almost always non degenerated in the sense of the following lemma: 
\begin{lemma}
  \label{lemma:isotroppyr}
  Let $(A_0,\cdots,A_3)$ be an isotropic quadrilateral such that
  its orthogonal projection on $L^\CC$ is not flat. Using the above notations,
  let $V'$ be the orthogonal projection of $V$ on $L^\CC$ and $X$ be the
  optimal solution. Then $D_0,
  D_1, V',X$  is a basis of $L^\CC$, unless $\beta_0=0$ or
  $\beta_1=0$.
  If $\beta_0\beta_1\neq 0$  the rays of the optimal isotropic pyramid
  $\overrightarrow{PA_i}$, for $0\leq i\leq 3$
  are linearly independent.
\end{lemma}
\begin{proof}
  Easy manipulations on vectors show that the vector space spanned by $D_0,D_1,V', X$ is also spanned by $D_0, D_1$ and the vectors
  $$
\beta_0 B_0 +\beta_1B_1 \mbox{ and } -\beta_0 B_0 +\beta_1B_1.
$$
The two above vectors belong to $JL$ and they are linearly independent if, and only if
$$
\beta_0\beta_1\neq 0,
$$
which proves the lemma as the second statement is an immediate consequence of the first.
\end{proof}

\subsection{$\cC^0$-estimates for optimal pyramids}

  \begin{dfn}
    A quadrilateral of $\RR^{2n}$ with orthonormal diagonals  $(D_0,D_1)$ is called an
    \emph{orthonormogonal quadrilateral}. If  the diagonals
    satisfy
    $$\forall i,j\in \{0,1\} \quad \Big |  g(D_i,D_j)-\delta_{ij}\Big | \leq\epsilon
    $$
    for some $\epsilon >0$, we say that they are
    \emph{$\epsilon$-orthonormal}. Under this assumption,   we say that \emph{the quadrilateral is $\epsilon$-orthonormogonal}.
  \end{dfn}

By continuity, we have the following result:
\begin{lemma}
  \label{lemma:continuityc0}
 For every pair  $D_0$, $D_1\in\RR^{2n}$ of  linearly independent vectors, we define
  $D'_0$, $D_1'\in span(D_0,D_1)$ by  the orthogonality relations
  $$
g(D_i,D'_j) =\delta_{ij}, \forall i,j\in\{0,1\}.
$$
Then $D'_0,D'_1$ is a basis of $span(D_0,D_1)$. Furthermore, 
for every $\epsilon'>0$ there exists $\epsilon_0>0$
such that for every $0<\epsilon<\epsilon_0$ and every
$\epsilon$-orthonormal family $(D_0,D_1)$, the family
$(D'_0,D_1')$ is $\epsilon'$-orthonormal.  
\end{lemma}

\begin{rmk}
  \label{rmk:ong}
We shall assume from now on that the quadrilateral is
$\epsilon$-orthonormogonal, with $\epsilon>0$ sufficiently small, so
that $D_0,D_1$ are linearly independent,
 $\|D_i\|\leq 2$ and  $\|D'_i\|\leq 2$.  
\end{rmk}

  \begin{prop}
\label{prop:boundedpyramid}
  There exist $C>0$ and $\epsilon >0$ such that
for every $\epsilon$-orthonormogonal isotropic quadrilateral of
diameter $d$,   the diameter $d'$  of the corresponding optimal isotropic pyramid  satisfies
$$
d' \leq C(d+1).
$$
\end{prop}
Loosely stated, the above proposition says that, for every  isotropic
quadrilateral which is almost orthonormogonal, the diameter of the
optimal isotropic pyramid is commensurate with the diameter of the
quadrilateral. 
\begin{proof}
  If the projection of the quadrilateral in $L^\CC$ is flat, then the optimal apex coincide with the barycenter of the quadrilateral and the proposition is obvious. Thus, we will assume that the projection of the quadrilateral is not flat in the rest of the proof.
  
  As $\epsilon\to 0$, the basis $D_0,D_1,B_0,B_1$ becomes almost orthonormal.
In particular, there exists $\epsilon>0$ sufficiently small such that under the assumptions of the proposition, we have
  $$
\max(|\alpha_0|,|\alpha_1|,|\beta_0|,|\beta_1|)\leq 2\|V\|.
$$
Then Formula~\eqref{eq:optimal} for the optimal solution $X=a_0D_0+a_1D_1+b_0B_0+b_1B_1$
  shows that all the coefficients
  $a_i$ and $b_i$ are controlled by $\|V\|+1$ (up to multiplication by a universal constant). Now,
  $$
\|X\|\leq |a_0|\|D_0\| + |a_1|\|D_1\| +|b_0|\|B_0\| + |b_1|\|B_1\|.
  $$
According to Remark~\ref{rmk:ong}, if $\epsilon$ is sufficiently small,
we have
$$
\|D_i\| \mbox{ and } \|D'_j\| \leq 2.
$$
Hence $\|B'_j\| \leq 2$ and it follows from the triangle inequality that
$$
\|X\|\leq 2(|a_0|+ |a_1| +|b_0|+ |b_1|).
$$
This shows that the distance between the optimal apex and the center of gravity of the quadrilateral is controlled by $\|V\|+1$, up to multiplication by a universal constant. The diameters of the quadrilateral controls $\|V\|$, hence the diameter of the
quadrilateral controls $\|X\|$ and the lemma follows.
\end{proof}

\subsection{Many quadrilaterals and pyramids}
Every faces of an isotropic quadrangular mesh  $\tau\in\scrM_N$ can be
seen  as a collection of isotropic quadrilaterals of $\RR^{2n}$.
In this section we explain how to define particular triangulations
$\cT_N(\Sigma)$ as a refinement of the quadrangulations
$\cQ_N(\Sigma)$. Then we explain how to deduce an isotropic
quadrangular mesh
$\tau'\in\scrM'_N=C^0(\cT_N(\Sigma))$from $\tau$.

\subsubsection{Triangulations obtained by refinement}
We define  triangulation $\cT_N(\RR^2)$ by replacing each face $\face$ of $\cQ_N(\RR^2)$ with
its barycenter $\zert_\face\in\RR^2$. The barycenter $\zert_\face$ is joined to the
vertices of the face $\face$ by straight line
segments. We also add four faces given by the four triangles which
appear as in the picture below. This
operation is better understood by drawing  a local picture of the
corresponding $CW$-complexes of $\RR^2$:
  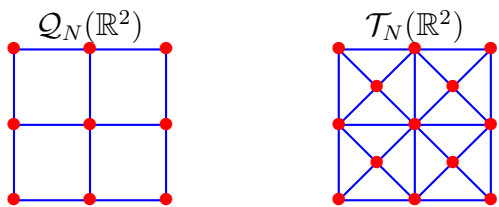
\begin{figure}[H]
  \begin{pspicture}[showgrid=false](-1,-1.5)(1,1.5)
    \psset{linecolor=blue }
    \psline (-1,1) (1,1)
    \psline (-1,0) (1,0)
    \psline (-1,-1) (1,-1)
    \psline (-1,1) (-1,-1)
    \psline (0,1) (0,-1)
    \psline (1,1) (1,-1)
    \color{red}
      \rput (-1,1){$\bullet$}
      \rput (-1,0){$\bullet$}
      \rput (-1,-1){$\bullet$}
      \rput (0,1){$\bullet$}
      \rput (0,0){$\bullet$}
      \rput (0,-1){$\bullet$}
      \rput (1,1){$\bullet$}
      \rput (1,0){$\bullet$}
      \rput (1,-1){$\bullet$}
      \color{black}
      \rput (0,1.3){$\cQ_N(\RR^2)$}
  \end{pspicture}
  \hspace{2cm}
  \begin{pspicture}[showgrid=false](-1,-1.5)(1,1.5)
    \psset{linecolor=blue }
    \psline (-1,1) (1,1)
    \psline (-1,0) (1,0)
    \psline (-1,-1) (1,-1)
    \psline (-1,1) (-1,-1)
    \psline (0,1) (0,-1)
    \psline (1,1) (1,-1)
    \psline (-1,1) (1,-1)
    \psline (0,1) (1,0)
    \psline (-1,0) (0,-1)
    \psline (-1,-1) (1,1)
    \psline (-1,0) (0,1)
    \psline (0,-1) (1,0)
    \color{red}
      \rput (-1,1){$\bullet$}
      \rput (-1,0){$\bullet$}
      \rput (-1,-1){$\bullet$}
      \rput (0,1){$\bullet$}
      \rput (0,0){$\bullet$}
      \rput (0,-1){$\bullet$}
      \rput (1,1){$\bullet$}
      \rput (1,0){$\bullet$}
      \rput (1,-1){$\bullet$}
      \rput (-.5,.5){$\bullet$}
      \rput (.5,.5){$\bullet$}
      \rput (-.5,-.5){$\bullet$}
      \rput (.5,-.5){$\bullet$}
      \color{black}
      \rput (0,1.3){$\cT_N(\RR^2)$}
  \end{pspicture}
  \caption{Triangular refinement of a quadrangulation}
  \end{figure}
As explained in~\S\ref{sec:quadconv} in the case of quadrangulations, the
triangulations $\cT_N(\RR^2)$ descend to $\Sigma$t via
the covering map
$p_N:\RR^2\to \Sigma$. The resulting triangulation of $\Sigma$ is
denoted $\cT_N(\Sigma)$. We define a moduli space of mesh associated to such triangulation 
$$
 \scrM'_N=C^0(\cT_N(\Sigma))\otimes \RR^{2n}.
 $$
 \subsubsection{Optimal triangulation of isotropic quadrangular mesh}
 \label{sec:optitri}
 Let $\tau\in\scrM_N$ be an isotropic quadrangular mesh.
In addition, we are assuming that  the quadrilateral
of $\RR^{2n}$ associated to each face
$\face$ of $\cQ_N(\Sigma)$  via
$\tau$  have linearly independent diagonals.
For each face $\face$ of  $\cQ_N(\Sigma)$, the mesh $\tau$ associates an
isotropic quadrilateral with linearly independent diagonals. We
associate an optimal apex
 $P_\face\in \RR^{2n}$ to such a quadrilateral.
Then, we define a triangular mesh $\tau'\in \scrM'_N$ as follows:
\begin{itemize}
\item If $\vert$ is a vertex of $\cQ_N(\Sigma)$, we define
  $\tau'(\vert)=\tau(\vert)$.
\item If $\zert$ is a vertex of $\cT_N(\Sigma)$ which is not a vertex of
  $\cT_N(\Sigma)$, it is the barycenter of a face $\face$ of
  $\cQ_N(\Sigma)$ and we put $\tau'(\zert)=P_f$, where $P_f$ is the optimal
  vertex defined via $\tau$.
\end{itemize}
This leads us to the following definition
\begin{dfn}
    Let $\tau\in\scrM_N$ be an isotropic quadrangular mesh with linearly
  independent diagonals. The triangular mesh $\tau'\in\scrM_N'$
  defined above is called the optimal triangulation of the isotropic mesh $\tau$.
\end{dfn}
By construction, we have the following obvious
property:
\begin{prop}
  Let $\tau\in\scrM_N$ be an isotropic quadrangular mesh with linearly
  independent diagonals. The optimal triangulation $\tau'\in\scrM_N'$
  of the quadrangular mesh $\tau$  defines a piecewise linear map
  $\ell':\Sigma\to\RR^{2n}$, which is isotropic.
\end{prop}

\subsection{Approximation by isotropic triangular mesh}
\label{sec:trimesh}
In Theorem~\ref{theo:fpt}, we construct a sequence of isotropic
quadrangular meshes $\rho_N\in\scrM_N$ out of  a smooth isotropic
immersion $\ell:\Sigma\to\RR^{2n}$. By construction,
$$
\rho_N=\tau_N- J\delta^\star_N\phi_N,
$$
where $\|\phi_N\|_{C^{2,\alpha}} =\cO(N^{-1})$.
By Proposition~\ref{prop:convell}, the renormalized diagonals of
$\tau_N$ converge towards the partial derivatives of $\ell$. Thus, the
same holds for  $\rho_N$, i.e.
\begin{equation}
  \label{eq:convdiag}
\scrU_{\rho_N}^\pm\longrightarrow \frac{\del\ell}{\del u} \mbox{ and }
\scrV_{\rho_N}^\pm\longrightarrow \frac{\del\ell}{\del v}.
\end{equation}
In particular the diagonals are linearly independent for every $N$
sufficiently large and  we may define an optimal isotropic
triangulation $\rho'_N\in\scrM'_N$ associated to $\rho_N$ as in the previous
section.  It turns out that the triangular meshes $\rho_N'$ are also
good $\cC^0$ approximations of the map $\ell$ in the sense of the
following proposition:
\begin{prop}
  \label{prop:convrhoprime}
  There exists a constant $C>0$, and $N_0>0$ such that for every
  integer $N\geq N_0$ and every vertex $\vert\in \cT_N(\Sigma)$
  $$
\|\ell(\vert) - \rho'_N(\vert)\|\leq \frac CN.
  $$
\end{prop}
\begin{proof} 
 Since $\|\phi_N\|_{\cC^{2,\alpha}_w}=\cO(N^{-1})$, we
deduce that $\|\phi_N\|_{\cC^{1}_w}=\cO(N^{-1})$. It follows that
there exists a constant $C_1>0$, such that
$\|\rho_N(\vert)-\tau_N(\vert)\|= \|\delta_N^\star\phi_N(\vert)\|\leq C_1N^{-1}$ for every $N$
sufficiently large and every vertex $\vert$ of $\cQ_N(\Sigma)$.
In such case, we have $\tau_N(\vert)=\ell(\vert)$ and
$\rho_N(\vert)=\rho'_N(\vert)$ so that
\begin{equation}
    \label{eq:comparerho1}
    \|\ell(\vert)-\rho'_N(\vert)\|\leq \frac {C_1}N.
\end{equation}
If $\vert$ is a vertex of $\cT_N(\Sigma)$ but not a vertex of
$\cQ_N(\Sigma)$, it is associated to a face $\face$ of the
quadrangulation and $\rho'_N(\vert)$ is the optimal apex associated to
$\face$ and $\rho_N$, by definition of $\rho'_N$ (cf. \S\ref{sec:optitri}).
The renormalized diagonals 
  $\scrU_{\rho_N}^\pm$ and $\scrV_{\rho_N}^\pm$ converge  toward
$\frac{\del\ell}{\del u}$ and $\frac{\del\ell}{\del v}$ by~\eqref{eq:convdiag}.
The partial derivatives $\frac{\del\ell}{\del u}$ and
$\frac{\del\ell}{\del v}$ are orthogonal, with norm
$\sqrt{\theta}$. Therefore
\begin{equation}
  \label{eq:vfon}
  \frac 1{\sqrt\theta_N}\scrU_{\rho_N}^\pm, \quad \frac
1{\sqrt\theta_N}\scrV_{\rho_N}^\pm
\end{equation}
converge toward a pair of smooth orthonormal vector fields on $\Sigma$.
In particular,  there exists $N_0$ such
that for all $N\geq N_0$, the vectors fields~\eqref{eq:vfon} are
$\epsilon$-orthonormal, where $\epsilon >0$ is chosen according to
Proposition~\ref{prop:boundedpyramid}. Since $\theta$ is a positive smooth function
on a compact surface, it is bounded above and below by positive
constants. Since $\theta_N^\pm\to\theta$, it follows that $\theta_N$
is also uniformly bounded above and below by positive constants for
$N$ sufficiently large. Using 
Proposition~\ref{prop:boundedpyramid} with the rescaled pyramid,
we deduce that  the apex $\vert$ is close to all the vertices $\zert$ of
$\face$ in the sense that, for some constant $C_2>0$ independent of $N_0$,
$\vert$ and $\zert$, we have
\begin{equation}
  \label{eq:comparerho}
  \|\rho'_N(\vert)-\rho_N(\zert)\|\leq \frac {C_2}N.
\end{equation}
Since $\ell$ is smooth, there exists a constant $C_3>0$ such that
for every pair of points $w_1,w_2\in \Sigma$ contained in the same
face of $\cQ_N(\Sigma)$, we have
\begin{equation}
  \label{eq:varell}
  \|\ell(w_1)-\ell(w_2)\|\leq\frac { C_3}N
\end{equation}
 In particular, for $\vert$ and $\zert$ as above, 
$$
\|\ell(\vert)-\rho'_N(\vert)\|\leq
\|\ell(\vert)-\ell(\zert)\|+\|\ell(\zert)-\rho_N(\zert) \| +\|\rho_N(\zert)-\rho'_N(\vert)\|.
$$
Since $\zert$ and $\vert$ belong to the same face,
$\|\ell(\vert)-\ell(\zert)\|\leq C_3N^{-1}$ by~\eqref{eq:varell}.
The second term satisfies  $\|\ell(\zert)-\rho_N(\zert) \| \leq C_1N^{-1}$
by~\eqref{eq:comparerho1} and the third term
$\|\rho_N(\zert)-\rho'_N(\vert)\|\leq C_2N^{-1}$
by~\eqref{eq:comparerho}. The proposition follows, with $C=C_1+C_2+C_3$.
\end{proof}

We deduce  the following result, which proves the first part
of Theorem~\ref{theo:maindiscr}
\begin{theo}
  \label{theo:maindiscrweak}
  The piecewise linear maps $\ell_N:\Sigma\to\RR^{2n}$ associated to the triangular
  meshes $\rho'_N$ are isotropic. Furthermore
  $$
\|\ell-\ell_N\|_{\cC^0} =\cO(N^{-1}),
$$
where $\|\cdot\|_{\cC^0}$ denotes the usual $\cC^0$-norm for maps $\Sigma\to\RR^{2n}$.
\end{theo}
\begin{proof}
The first part of the theorem is obvious. By definition of an
isotropic triangular mesh, the piecewise linear map $\ell_N$ associated
to $\rho'_N$ is isotropic.

The following lemma is a trivial consequence of the convergence
statement of
Proposition~\ref{prop:convrhoprime}. This roughly says that the
triangles of the mesh $\rho'_N$ have diameter of order $\cO(N^{-1})$.
\begin{lemma}
  \label{lemma:triest}
  There exists a constant $C_4>0$ such that for every $N$ sufficiently
  large  and  every pair of vertices $\vert_1,\vert_2$ of
  $\cT_N(\Sigma)$ which belong to the same face
  $$
\|\rho'_N(\vert_1)-\rho'_N(\vert_2)\|\leq \frac{C_4}N.
$$
\end{lemma}
Lemma~\ref{lemma:triest} applied to the piecewise linear maps $\ell_N$
shows that
there exists a constant $C_5>0$ such that for every $N$ sufficiently large and $w_1,w_2\in \Sigma$ which belong
to the same triangular face of $\cT_N(\Sigma)$, we have
\begin{equation}
  \label{eq:varellprimeN}
  \|\ell_N(w_1)-\ell_N(w_2)\|\leq \frac{C_5}N.  
\end{equation}
For $N$ sufficiently large, we may assume the
control~\eqref{eq:varell}. For $w\in\Sigma$ and $N$ sufficiently
large, we choose a vertex $\vert$ of the face of $\cT_N(\Sigma)$ which
contains $w$. Then
$$
\|\ell_N(w)-\ell(w)\|\leq \|\ell_N(w)-\ell_N(\vert)\|+ \|\ell_N(\vert)-\ell(\vert)\|
+\|\ell(\vert)-\ell(w)\|.
$$
The first term is bounded by~\eqref{eq:varellprimeN}, the second term
is bounded by Proposition~\ref{prop:convrhoprime} and the third is
bounded by \eqref{eq:varell}. This proves the theorem.
\end{proof}

\subsection{Piecewise linear immersions}
Recall that a piecewise linear map is an immersion if, and only if, it
is a locally injective map.
The piecewise linear isotropic approximations $\ell_N$ of a smooth
isotropic immersion $\ell:\Sigma\to\RR^{2n}$ considered at \S\ref{sec:trimesh} are only close in
$C^0$-norm by Theorem~\ref{theo:maindiscrweak}.
Since this estimate is rather weak, we cannot deduce from this fact that
$\ell_N$ is  an immersion for $N$ sufficiently large.
However there are many free parameters in our construction:
\begin{itemize}
\item The distortion action on $\scrM_N$ preserves isotropic meshes.
  \item The apex of each isotropic pyramid with fixed base lies in an
    affine space of dimension at least $2n-3$.
\end{itemize}
These parameters can be used to construct piecewise linear isotropic
immersions, at least when the dimension of the target space is
sufficiently large, which turns out to be   $n\geq 3$. 

\subsubsection{Perturbed meshes without flat faces}
We start by perturbing the  isotropic quadrangular meshes
$\rho_N\in\scrM_N$
constructed at Theorem~\ref{theo:fpt}.
Our goal is to perturb $\rho_N$ by the shear action,
 to make sure that the quadrilateral
associated to each face of the mesh satisfy the following proposition
and, in particular,
are not flat in $\RR^{2n}$.
\begin{prop}
  \label{prop:partimm}
 For every $N$ sufficiently large, there exists  $T_N \in\RR^{2n}\times \RR^{2n}$
 such that for every $s>0$  small enough,
 the quadrangular mesh
 $$\rho_N^s=sT_N\cdot\rho_N$$
 satisfies the
 following properties:
 \begin{enumerate}
   \label{prop:noflatface}
 \item   \label{prop:item1} For each face of the quadrangular mesh $\rho^s_N$, the
   orthogonal projection of 
 the corresponding  quadrilateral 
  onto the complex space generated by its diagonals (cf. ~\eqref{eq:cxquad}) is not
  flat.
  \item \label{prop:itemvert} For every vertex $\vert$ of $\cQ_N(\Sigma)$, the four vectors
    of $\RR^{2n}$, associated via $\rho_N^s$ to the four edges with vertex $\vert$,
 span a $3$-dimensional
    subspace of $\RR^{2n}$. Furthermore any triplet obtained as a
    subset of  the four above
    vectors is a linearly independent family.
    \item \label{prop:item2} The associated
     triangular meshes $(\rho_N^s)'\in \scrM'_N$ have generic
     pyramids. In other words, 
      for every vertex $\vert$ of $\cT_N(\Sigma)$ which is not a vertex
      of $\cQ_N(\Sigma)$, the
 four vectors of $\RR^{2n}$ associated to the four edges of the mesh $(\rho_N^s)'$
 at $\vert$  are linearly independent.
 \end{enumerate}
\end{prop}
\begin{proof}
We use the notations introduced at the beginning of \S\ref{sec:quadtri}: for
a quadrilateral $Q=(A_0,\cdots,A_3)$ of $\RR^{2n}$, we denote by $X$ the
vector defined by~\eqref{eq:gp}, $D_0,
D_1$ the diagonals defined~\eqref{eq:diag}  and by $B_0,B_1$ the
vectors~\eqref{eq:Bi} of $JL$ (cf. \eqref{eq:cxquad}).

The condition that the projection of the quadrilateral onto $L^\CC$ is
flat is equivalent to $\beta_0=\beta_1=0$, where the $\beta_i$ has been defined in (\ref{eq:decV}).
Assume that the projection is flat. Then for
 every $T_+ \in
\RR^{2n}$ not contained in the hyperplanes $\beta_0=0$ or
$\beta_1=0$, the projection of the quadrilateral $Q_s=(A_0+sT_+,A_1,A_2+sT_+,A_3)$ is not
flat for every $s>0$. Furthermore the optimal pyramid with base $Q_s$ is
generic in the sense of Lemma~\ref{lemma:isotroppyr}.

Let $\rho_N$ be the isotropic quadrangular mesh considered in
hypothesis of the
proposition. We choose $T_+\in\RR^{2n}$ which satisfies the above
property, for every quadrilateral associated to faces of the mesh
$\rho_N$ with flat projection onto the space of complexified
directions. This is possible, since we merely need to choose $T_+\in
\RR^{2n}$ away from a finite collection of $2$-planes.
Then $\rho^s_N:=(sT_+,0)\cdot\rho_N $ satisfies the items
\eqref{prop:item1} and \eqref{prop:item2} of the proposition
provided $s>0$ is sufficiently small.

We just have to show that the condition \eqref{prop:itemvert} can be
satisfied for a suitable choice of deformations.
Given a vertex $\vert$ of $\cQ_N(\Sigma)$ we consider the four
diagonals
 $D_{\vert,\face}^{\rho_N}$ for the four quadrangular faces $\face$
with vertex $\vert$.
The renormalised diagonals of the mesh converge
toward the partial derivatives of $\ell$ at $\vert$ (cf.~\eqref{eq:convdiag}) as
$N\to\infty$. Since $\ell$ is an immersion, this shows that the four
vectors $D_{\vert,\face}^{\rho_N}$ span a space of dimension $2$ or
$3$ for every $N$ sufficiently large. If this space is $3$-dimensional,
\eqref{prop:itemvert} is satisfied with $s=0$ and nothing needs to be
done. Assume that the space of diagonals is $2$-dimensional.
The four vertices connected by an edge to $\vert$ define four points
of $\RR^{2n}$ via $\rho_N$. By assumption, these points  lie in an affine plane of
$\RR^{2n}$. If $\rho_N(\vert)$ does not belong to this plane,
then~\eqref{prop:itemvert} is satified. Otherwise, we require the additional condition that $T_+$
does not belong to the plane spanned by the diagonals.
We have to consider every vertex $\vert$ as above and this
adds a finite  number of conditions for  choosing $T_+$. A finite
family of proper subspaces of a vector space never covers the entire
space. Thus it is possible to find the desired $T_+$.
This concludes the proof of the proposition with $T_N=(T_+,0)$.
\end{proof}
\begin{cor}
  \label{cor:immae}
 Given $N$ large enough, for every $s>0$
  sufficiently small, the isotropic triangulation $(\rho_N^s)'$
  defines a piecewise linear map $\ell_{N,s}:\Sigma\to\RR^{2n}$ which
  is an immersion at every point $w\in\Sigma$ which does not belong to
  the   $1$-skeleton of $\cQ_N(\Sigma)$.  In particular $\ell_{N,s}$
  is an immersion at almost every point of $\Sigma$.
\end{cor}
\begin{proof}
By linerarity, it is sufficient to check that $\ell_{N,s}$ is
an immersion at every vertex of $\cT_N(\Sigma)$ which is not a vertex
of $\cQ_N(\Sigma)$. But this is clear for $N$ large enough and $s>0$
sufficiently small, by Proposition~\ref{prop:noflatface}, item \eqref{prop:item2}. 
\end{proof}

\begin{rmk}
  The above corollary proves the second part of
Theorem~\ref{theo:maindiscr} concerning  piecewise linear
isotropic immersions when $n= 2$. Indeed, the  the $1$-skeleton of
$\cQ_{N}(\Sigma)$ is a finite union of meridians of the torus $\Sigma$.
\end{rmk}
\subsubsection{Further perturbations by moving apexes}
We are going to apply further isotropic perturbations to the triangular meshes
$(\rho_N^s)'\in\scrM'_N$, so that  that the corresponding piecewise
linear map is also an immersion along the $1$-skeleton of
$\cQ_N(\Sigma)$.

By definition, $(\rho_N^s)'$ is defined from the
quadrangular mesh $\rho_N^s$, by adding the apex of an optimal
isotropic pyramid for each face of $\cQ_N(\Sigma)$.
The definition of an optimal pyramid is somewhat
arbitrary: for $N$ large enough and $s>0$ sufficiently small,  every
face of $\rho_N^s$  satisfy Proposition~\ref{prop:noflatface},
item~\eqref{prop:item1}.
Hence, for each face of $\rho_N^s$, the affine space of apexes of
isotropic pyramids is  $2n-3$-dimensional.
We deduce the following lemma:
\begin{lemma}
  \label{lemma:moveapex}
  For $N$ large enough and $s>0$ sufficiently small, there exists a
   family of isotropic deformations of the triangular isotropic mesh
   $(\rho^s_N)'$. This family is obtained by moving each vertex of
   $\cT_N(\Sigma)$ which does not belong to $\cQ_N(\Sigma)$ within a
   $2n-3$-dimensional affine space.
\end{lemma}

  The key observation, that will make Lemma~\ref{lemma:moveapex} useful for
  our purpose, is that $2n-3\geq 3$ for  $n\geq 3$.
  In particular, we deduce the following proposition:
  
  \begin{prop}
    \label{prop:imm}
  Assume that $n\geq 3$, and $N$ is sufficiently large. Then, for every $s>0$ sufficiently small,  there exist isotropic triangular meshes arbitrarily
 close to $(\rho_N^s)'$, which define piecewise linear immersions  $\Sigma\hkto\RR^{2n}$.
\end{prop}
\begin{proof}
  As in Corollary~\ref{cor:immae}, showing that a map is an immersion is a purely local
  matter.
  We draw a local picture of the triangular mesh $(\rho^s_N)'$, near
  the image $O$ of vertex
  $\vert$ of $\cQ_N(\Sigma)$. In Figure~\ref{figure:localpert}, the
  bullet labelled $O$ actually represents  $\rho_N^s(\vert)\in
  \RR^{2n}$. Similarly, all be points $P_i$, and $ A_{ij}\in \RR^{2n}$
  of the picture are
  images of corresponding vertices of  $\cT_N(\Sigma)$ by the triangular
  isotropic mesh $(\rho_N^s)'$.
Notice that the black and blue bullets are prescribed by the
  quadrangular mesh $\rho_N^s$, whereas the red bullets are defined by its
  triangular refinement $(\rho_N^s)'$. More specifically, the red bullets are the optimal apexes of
  the corresponding  optimal isotropic pyramids.
    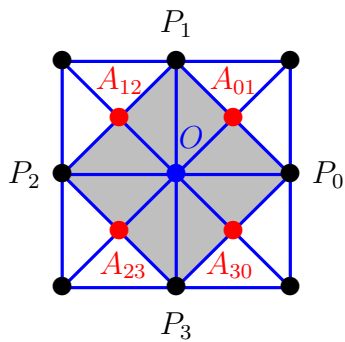
\begin{figure}[H]
  \begin{pspicture}[showgrid=false](-2,-2)(2,2)

\psscalebox{1.5}{
  \psdiamond[linecolor=lightgray,fillstyle=solid,fillcolor=lightgray](0,0)(1,1)
  \psset{linecolor=blue }
    \psline (-1,1) (1,1)
    \psline (-1,0) (1,0)
    \psline (-1,-1) (1,-1)
    \psline (-1,1) (-1,-1)
    \psline (0,1) (0,-1)
    \psline (1,1) (1,-1)
    \psline (-1,1) (1,-1)
    \psline (0,1) (1,0)
    \psline (-1,0) (0,-1)
    \psline (-1,-1) (1,1)
    \psline (-1,0) (0,1)
    \psline (0,-1) (1,0)
    \color{red}
    \rput (-.5,.5){$\bullet$}
      \rput (.5,.5){$\bullet$}
      \rput (-.5,-.5){$\bullet$}
      \rput (.5,-.5){$\bullet$}
      
      \color{black}
      \rput (-1,1){$\bullet$}
      \rput (-1,0){$\bullet$}
      \rput (-1,-1){$\bullet$}
      \rput (0,1){$\bullet$}
      \rput (0,-1){$\bullet$}
      \rput (1,1){$\bullet$}
      \rput (1,0){$\bullet$}
      \rput (1,-1){$\bullet$}
      \color{blue}
      \rput (0,0){$\bullet$}
}
  \psset{linecolor=black }
      \rput (2,0){$P_0$}
      \rput (-2,0){$P_2$}
      \rput (0,2){$P_1$}
      \rput (0,-2){$P_3$}

        \color{red}
        \rput(-.75,1.25){$A_{12}$}
        \rput(+.75,1.25){$A_{01}$}
        \rput(-.7,-1.2){$A_{23}$}
        \rput(+.7,-1.2){$A_{30}$}
        
        \color{blue}
        \rput(.2,.5){$O$}
  \end{pspicture}
  \caption{Local perturbations of a triangular mesh}
  \label{figure:localpert}
    \end{figure}

    We are now looking for a perturbation $(\rho_N^s)''$ of $(\rho_N^s)'$
    by moving  the  points $A_{ij}$.
    We denote $\ell_{N,s}'$ (resp. $\ell_{N,s}''$) the piecewise
    linear maps associated the triangular mesh $(\rho_N^s)''$
    (resp. $(\rho_N^s)''$).

    \begin{enumerate}
    \item \label{item:immae}The property of being an immersion is stable under small
deformations. Thus, for sufficiently small perturbation,
Corollary~\ref{cor:immae} holds for $\ell''_{N,s}$ as well. In
particular,
 $\ell''_{N,s}$ is an immersion at every point $w\in\Sigma$  contained
in the interior of one of the four faces of $\cQ_N(\Sigma)$, with vertex
$\vert$ (the four smaller square in the figure).
\item \label{item:perthyp}Suppose that we can choose a perturbation so that
 that $\ell''_{N,s}$
 is  an  immersion at the vertex
$\vert$ (corresponding to the point $O$). By linearity, this implies that $\ell''_{N,s}$ 
 is  an immersion at every interior point $w\in\Sigma$
 of  the union of shaded faces of the triangulation $\cT_N(\Sigma)$  (with gray color on the picture). 
    \end{enumerate}
If we are able to show that there exists a perturbation, which
satifies the condition \eqref{item:perthyp} as above, we deduce,
together with the above property \eqref{item:immae}, that the piecewise
linear map $\ell''_{N,s}$ is an immersion at every interior point $w$
 of the union of the four faces of $\cQ_N(\Sigma)$ with vertex $\vert$
 (the big square in Figure~\ref{figure:localpert}).
 In conclusion, if we have proved the following lemma:
 \begin{lemma}
   \label{lemma:immew}
   If $(\rho_{N}^s)''$  is a triangular mesh sufficiently close to
   $(\rho_{N}^s)'$, such that the correponding piecewise linear map
   $\ell''_{N,s}:\Sigma\to\RR^{2n}$ is an immersion at every vertex of
   $\cQ_N(\Sigma)$, then $\ell''_{N,s}$ is an immersion at every point
   of $\Sigma$.
 \end{lemma}
 We merely need to show that there exists an isotropic perturbation
 $(\rho_{N,s})''$ which satisfies the hypothesis of
 Lemma~\ref{lemma:immew} and the proof of the proposition will be
 complete.

Consider the mesh $(\rho_N^s)'$ represented locally by
Figure~\ref{figure:localpert}. There are $2n-3$ degrees of freedom for
perturbing each red vertex $A_{ij}$, in such a way that the triangular
mesh remains isotropic.
We would like
to put them in general position, so that the piecewise linear map
 is an
 immersion at $O$.
 First, notice that the local injectivity is partially satisfied by
$(\rho_N^s)'$ for every $s>0$ sufficiently small. Indeed, by
 Corollary~\ref{cor:immae}, two contiguous triangles of the mesh $(\rho_N^s)'$
 in a common pyramid, for instance $(OP_0A_{01})$ and $(OA_{01}P_1)$, are contained
in distinct planes intersecting along a line of $\RR^{2n}$, which in this particular
case is  $(OA_{01})$.

Consider now the two triangles of 
$(OP_0A_{01})$ and $(OA_{30}P_0)$ of $\RR^{2n}$. There are two
possibilities:
\begin{enumerate}
\item The line  $(OA_{30})$ is not contained in the plane of the triangle
$(OP_0A_{01})$, the two triangles lie in distinct plane intersecting
  along the line $(OP_0)$.
  \item  The line  $(OA_{30})$ is
 contained in the plane of the triangle $(OP_0A_{01})$. In this case,
 the associated piecewise linear map is not locally injective at $O$.
\end{enumerate}
In the second situation, we can always find an arbitrarily small
perturbation of the point $A_{30}$ which brings us back to the first situation, such that the associated piecewise
linear map is still isotropic.
Indeed, as pointed out there is a $2n-3\geq 3$ dimensional family of
points $A_{30}$ such that provide isotropic perturbation. There is a least. 
Such space cannot be contained in the plane of $(OP_0A_{01})$
for obvious dimensional reasons. Thus, we may find the wanted
arbitrarily small isotropic perturbations of $A_{30}$ such that we are
in the first situation.

We consider now the case where we have two non contiguous triangles,
for instance $(OP_0A_{01})$ and 
 $OP_1A_{12}$.  We know that the three lines $OP_0$, $OA_{01}$ and
$OP_1$ span a $3$-dimensional space by Corollary~\ref{cor:immae}.
By moving slightly $A_{12}$ within its~$2n-3$-dimensional family
of isotropic perturbation,  we can make sure that the intersection of
the planes containing the triangles $(OP_0A_{01})$ and $(OP_1A_{12})$ reduces to the point $O$.

There are other situations that we should handle as well.
For instance, we consider the triangles $(OP_0A_{01})$ and
$(OA_{12}P_2)$.
By Propositin~\ref{prop:noflatface}, item \eqref{prop:itemvert},
 the lines $(OP_0)$ and
$(OP_2)$ are distinct. Up to a small isotropic perturbation by moving
 $A_{01}$ within its $2n-3$-dimensional family,  we may assume that
 $A_{01}$ does not belong to the plane $(OP_0P_2)$.
 By moving $A_{12}$ similarly, we may assume that $A_{12}$ does not
 belong to the plane that contains the triangle $(OP_0A_{01})$.
Eventually, the two planes that contain $(OP_0A_{01})$
and $(OA_{12}P_2)$, after perturbation, intersect at a single point $O$. 

Other cases are dealt with similarly. Eventually we have proved that
there are arbitrarily small isotropic deformations of $(\rho_N^s)'$,
obtained by moving the points $A_{ij}$, 
such that  the eight triangles of the mesh around $O$ lie in distinct
planes. In particular, the corresponding piecewise linear map is
an immersion at the vertex $\vert$.

By induction, we can apply further similar perturbation, so that the
isotropic piecewise linear map is an immersion at every vertex of
$\cQ_N(\Sigma)$.
This proves the proposition.
\end{proof}
\subsection{Proof of Theorem~\ref{theo:maindiscr}}
Gathering our results provides a complete proof one of our main
results:
\begin{proof}[Proof of Theorem~\ref{theo:maindiscr}]
  The existence of $\cC^0$-approximations of smooth isotropic
  immersions $\ell:\Sigma\to\RR^{2n}$ by piecewise  linear isotropic
  maps is proved in Theorem~\ref{theo:maindiscrweak}.
The statement for existence of approximations by piecewise linear
isotropic immersions is a proved at Proposition~\ref{prop:imm}.
The remaining case, for $n=2$, is a consequence of Corollary~\ref{cor:immae}.
\end{proof}

\section{Discrete moment map flow}
\label{sec:dmmf}
The moduli space $\scrM=\{ f:\Sigma\to M, f^*[\omega]=0\}$, where $\Sigma$ is a closed
surface endowed with an area form $\sigma$ was introduced at
\S\ref{sec:dream}. If $M$ is a K\"ahler
manifold, then $\scrM$ has an induced formal K\"ahler structure
$(\scrM,\fJ,\fg,\Omega)$. The 
group $\cG=\Ham(\Sigma,\sigma)$ acts isometrically on $\scrM$.
The action is Hamiltonian, with moment map $\mu:\scrM\to
C^\infty_0(\Sigma)$, given by $\mu(f)=\frac{f^*\omega}\sigma$. In this
setting, a natural moment map flow is defined (cf. \S\ref{sec:mmf}) by 
$$
\frac{df}{dt} = -\frac 12  \grad \|\mu\|^2.
$$
The properties of the above flow shall be studied
 in a sequel to this
 work~\cite{JRT}. For the time being,  we merely provide a numerical
 simulation of the above flow, implemented in the program
  \emph{Discrete Moment Map Flow (DMMF)}, hosted on the
 webpage:
\begin{center}
\href{http://www.math.sciences.univ-nantes.fr/~rollin/index.php?page=flow}{http://www.math.sciences.univ-nantes.fr/\textasciitilde{}rollin/}.  
\end{center}
The idea is to approximate the flow, which is an evolution equation on
an infinite dimensional space of maps, by an analogue finite dimensional
approximation.
 We
do not try to compare the two flows from a mathematical
perspective. The finite dimensional flow is expected to
converge in some sense to the infinite dimensional flow as
$N\to\infty$, but this is part of a broader project
to be expanded later in~\cite{JRT}.

\subsection{Definition of the finite dimensional flow}
In the rest of this section, we focus on the case where $M=\RR^4$,
 with its standard Kähler structure and
$\Sigma$ is a surface diffeomorphic to a  torus, endowed
covering map $p:\RR^2\to\Sigma$ with $\Gamma$, its group of deck
transformation, which is a lattice of $\RR^2$. This data allows to define the
quadrangulations $\cQ_N(\Sigma)$.
The space 
of quadrangular meshes $\scrM_N$
is seen as a discrete analogue of the moduli space $\scrM$. The moment map $\mu$  has a discrete
version as well, given by
$$
\mu_N^r:\scrM_N\to C^2(\cQ_N(\Sigma)).
$$
The space of discrete functions $C^2(\cQ_N(\Sigma))$ is also
understood as a discrete analogue of $C^\infty(\Sigma)$. Recall that
this space of discrete functions is  endowed with an inner product
$\ipp{\cdot,\cdot}$,  which is an analogue of the $L^2$-inner product
induced by $\sigma$ (cf. \S\ref{sec:ip}) and denoted
$\ipp{\cdot,\cdot}$ as well. We denote by $\|\cdot\|$ the norm induced by the
inner product $\ipp{\cdot,\cdot}$. Then
\begin{align*}
D\|\mu_N^r \|^2|_\tau\cdot V  & =2\ipp{D\mu_N^r|_\tau  \cdot V,
  \mu_N^r(\tau) }\\
&=- 2\ipp{D\mu_N^r|_\tau\circ J  \cdot JV, \mu_N^r(\tau) }\\
   &=2  \ipp{\delta_\tau (JV), \mu_N^r(\tau)}\\
  &=2\ipp{JV,\delta_\tau^\star \mu_N^r(\tau)}
\end{align*}
hence
\begin{equation}
  \label{eq:gradphi}
-\frac 12 D\|\mu_N^r \|^2|_\tau\cdot V  =\ipp{V,J\delta_\tau^\star \mu_N^r(\tau)}.  
\end{equation}
where
$$
\delta_\tau = - D\mu_N^r|_\tau\circ J.
$$
Its adjoint $\delta_\tau^*$ is defined by
$\ipp{\delta_\tau V,\phi}=\ipp{V,\delta_\tau^\star\phi}$.
For each map $u:\scrM_N\to C^2(\cQ_N(\Sigma))$,
we may define a formal gradient vector field on the moduli space
$$\grad\, u :\scrM_N\longrightarrow
C^0(\cQ_N(\Sigma))\otimes\RR^4$$
by
$Du|_\tau\cdot V =\ipp {\grad\, u|_\tau , V }$.
Thus, by~\eqref{eq:gradphi}
$$
-\frac 12 \grad \|\mu_N^r(\tau)\|^2 = J\delta_\tau^\star \mu_N^r(\tau).
 $$
and we can define a downward gradient flow by
$$
\frac{d\tau}{dt} = - \frac 12 \grad \|\mu_N^r\|^2
$$
which is equivalent to
\begin{equation}
  \label{eq:discrf}
\boxed{\frac{d\tau}{dt} = J\delta_\tau^\star\mu_N^r(\tau).  }
\end{equation}
\begin{dfn}
 A solution $\tau_t:I\to\scrM_N$ of the ordinary differential
 equation~\eqref{eq:discrf}, where $I$ is an open interval of $\RR$, is called a solution of the discrete moment
 map flow. 
\end{dfn}
\begin{rmk}
  The discrete moment map flow is an ordinary differential equation
  with smooth coefficients on
  the affine space $\scrM_N$. The solution exists for short time but
  might blowup in finite time. The general behavior of the flow
  will be addressed in~\cite{JRT}.
\end{rmk}

The flow has typical properties of ODE with smooth coefficients:
\begin{prop}
  \label{prop:convflow}
  Assume that $\tau_t:[0,C)\to\scrM_N$ is a maximal solution of the
    the discrete moment map flow. If $\tau_t$ is bounded for $t\in [0,C)$, then
    $C=+\infty$. If in addition $\tau_t$ converges to some
      $\tau_\infty$, then 
  $\mu_N^r(\tau_\infty)\in\ker\delta_{\tau_\infty}^*$.
\end{prop}
\begin{rmk}
  If the function $\|\mu^r_N\|^2$ on $\scrM_N$ was Morse, any bounded
  flow $\tau_t$
  would  automatically converge toward a critical point of the function.
 Although we are not trying to prove this fact, all our experiments with the DMMF program seem to indicate that the
  flow is generically bounded and convergent.  If $\ker\delta_{\tau_\infty}
  = 0$, the conclusion of Proposition~\ref{prop:convflow} implies that
  $\mu_N^r(\tau_\infty)$ is a constant discrete function and, by
  Stokes theorem, $\tau_\infty$ must be an
  isotropic quadrangular mesh. Notice that the fact that the kernel of the
  operator $\delta_\tau^\star$ is $1$-dimensional holds for generic
  $\tau$ according to Proposition~\ref{prop:generickernel}. This is also
  confirmed by all the experiments using the DMMF program.
\end{rmk}
\begin{proof}
  If $C$ is finite, and $\frac {d\tau}{dt}$ is bounded, then $\tau_t$
  must converge to some $\tau_C$ as $t\to C$. 
  This contradicts the fact that $C$ is
  maximal.
  Hence, if $C$ is finite, $\frac {d\tau}{dt}$ must be unbounded.
  In particular $\tau_t$ cannot remain in a bounded set, as the RHS of
  the evolution equation would be bounded.
In conclusion, if $\tau_t$ is bounded,   we have $C=+\infty$.
If $\tau_t$ converges towards $\tau_\infty$, the limit must be a fixed
point of the flow and 
$\delta^\star_{\tau_\infty}\mu_N^r(\tau_\infty)=0$.
\end{proof}

\subsection{Implementation of the discrete flow}
\label{sec:discflow}
\subsubsection{Particular lattices}
Recall that the quadrangulations $\cQ_N(\Sigma)$ are defined by identifying the
torus $\Sigma$ with a quotient, via the diffeomorphism
 $\Phi:\RR^2/\Gamma\to\Sigma$ induced by the covering map $p:\RR^2\to\Sigma$.
We merely  have to make a choice for the
lattice $\Gamma$, in order to define $\cQ_N(\Sigma)$ and a
corresponding discrete moment map flow. This choice is arbitrary and
a sufficiently  sophisticated program could deal with any choice.
This is not the case of the DMMF program, however, which is base on
the choice of lattice
$$
\Gamma'' = \ZZ e_1\oplus \ZZ e_2,
$$
and surface  $\Sigma
=\RR^2/\Gamma''$ already introduced at~\S\ref{sec:examples}.
Then $\Gamma''\subset
\Lambda_N$ for every positive integer $N$. The quadrangulation $\cQ_N(\RR^2)$
descends as a quadrangulation of the quotient $\cQ_N(\Sigma)$.
The quadrangulation has $N^2$ vertices and a mesh in $\scrM_N$ can be
stored as an $N\times N$ array with entries in $\RR^4$.

\subsubsection{The Euler method}
It is easy to provide numerical approximations of an ODE such as the
discrete moment map flow by the Euler method.
We consider discrete time values $t_i=i\Delta t$, where $i$ is an
integer and
 $\Delta t>0$ is a small time step increment. Starting at
time $t_0=0$ with a mesh $\tau_0\in\scrM_N$, we compute $\tau_1$,
$\tau_2$, etc... as follows: given $\tau_i \in \scrM_N$
at time $t=i\Delta t$ we compute
$$V_{i}=\delta_{\tau_i}^\star\mu_N^r(\tau_i)$$
and define
$$
\tau_{i+1}= \tau_i +\Delta t\cdot JV_i .
$$
 The above computations are easy to carry out and
the operator $\delta_\tau^\star$ is explicitly given  by Lemma~\ref{lemma:operators}.
Starting from any quadrangular mesh, we can compute the above flow very
quickly in real time on an ordinary machine, whenever $N$ is not too
big (for instance $N\leq 100$ on our laptop).

\subsubsection{Visualization}
A choice has to be made for the visualization of each mesh
$\tau\in\scrM_N$ on a computer screen. The basic idea is to choose a
projection of $\RR^4$ on a $3$-dimensional manifold and represent a
mesh as a surface in a $3$-dimensional space.
We explain now the choice made in the DMMF program, which may not be
the best for certain situations:
we perform a radial projection of the vertices of $\tau$ onto the unit
sphere $\SS^3$ of $\RR^4$, centered at the origin. This projection is
followed by a stereographic projection of the sphere minus a point
onto one of its tangent spaces identified to
$\RR^3$. Once the positions of the  projections of the vertices of
$\tau$ in $\RR^3$ are computed, we can draw the quadrilateral
associated to the
faces in $\RR^3$. A library like OpenGL allows to represent a
quadrangular mesh of $\RR^3$ in perspective.
We fill the faces with a range of colors which depends on the symplectic
density of each face (i.e. the value of $\mu_N^r(\tau)$ on this face).
In addition, motions of the mouse are used to precompose these projections
with Euclidean rotations of $\RR^4$. This technique allows the user to look at surfaces from
various angles using the mouse.

\subsubsection{The DMMF code}
We found out the  \emph{processing} language,
which is a  java dialect, was extremely efficient to code the DMMF program.
The source code and more information on the technical aspects of  the
DMMF program are available on the homepage:
\begin{center}
\href{http://www.math.sciences.univ-nantes.fr/~rollin/index.php?page=flow}{http://www.math.sciences.univ-nantes.fr/\textasciitilde{}rollin/}.  
\end{center}
 The program starts the flow by sampling various examples of parametrized 
tori in $\RR^4$. From an experimental point of view, our numerous
observations seem to 
indicate that the flow should always converges, that the convergence is
fast, and that the limits are
isotropic. Figure~\ref{figure:dmmf} shows an example of  (static) output of the
DMMF program. This quadrangular mesh has diameter of order $1$ and
symplectic density of order $10^{-8}$. The reader is encouraged to
experiment directly  the more interactive and dynamic aspects of the program.
\begin{figure}[H]
  \includegraphics[width=400bp]{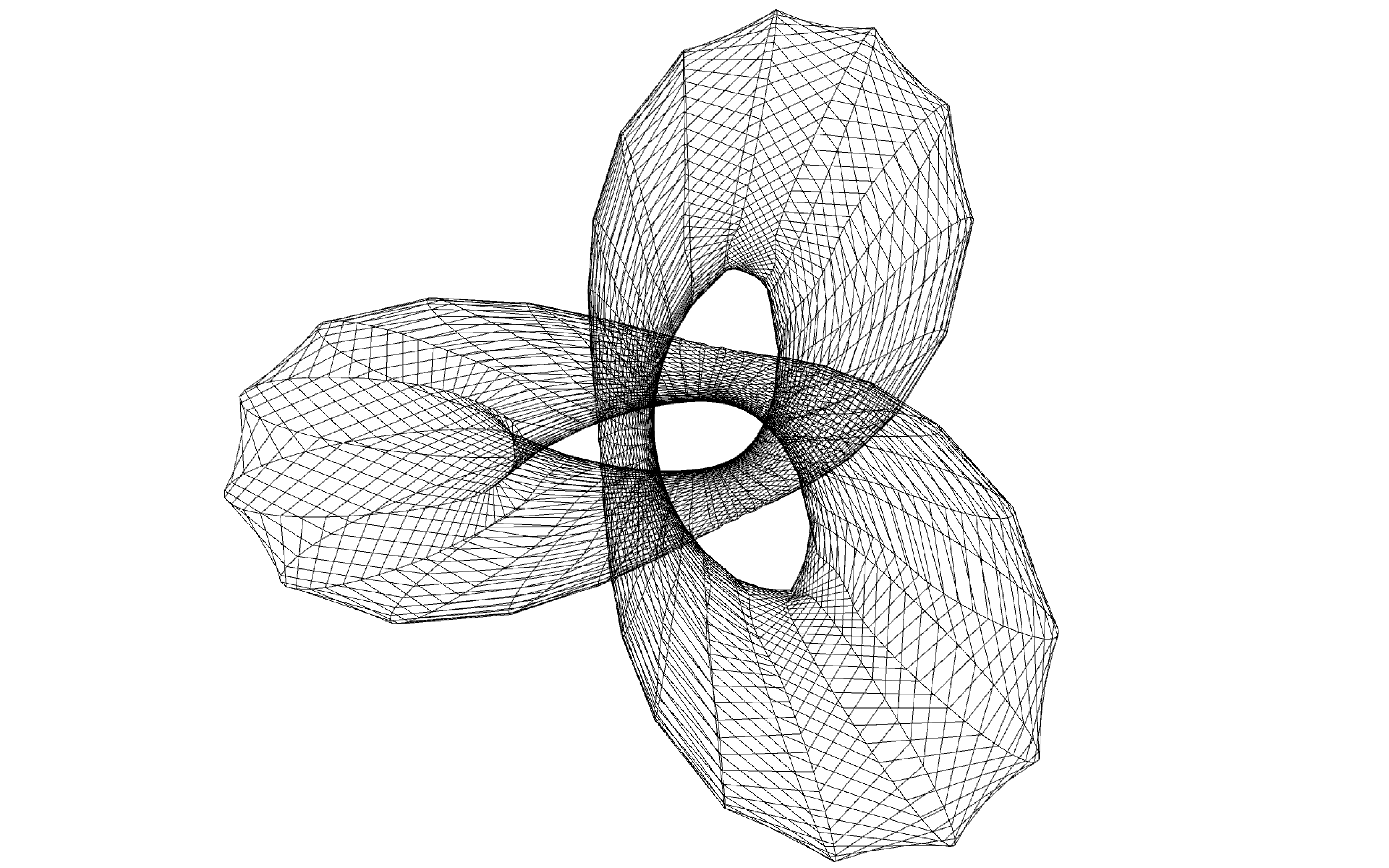}
  \caption{Isotropic quadrangular mesh}
  \label{figure:dmmf}
\end{figure}

\vspace{10pt}
\bibliographystyle{abbrv}
\bibliography{dislag}

\end{document}